\documentclass[10pt,a4paper]{article}
\usepackage[margin=1in]{geometry}
\usepackage{graphicx}
\usepackage[colorlinks=true, allcolors=black]{hyperref}
\usepackage{amsmath,amsthm,amsfonts,amssymb,amscd,mathrsfs,bm}
\usepackage{enumerate}
\usepackage{mathrsfs}
\usepackage{xcolor}
\usepackage{graphicx}
\usepackage{hyperref}
\usepackage{enumitem}
\usepackage{bbm}
\usepackage{float}
\usepackage{fancyvrb}
\usepackage{booktabs}
\usepackage{subcaption}
\usepackage{titlesec}
\usepackage{url} 

\usepackage[ruled,vlined]{algorithm2e}
\def\algbackskip{\hskip-\ALG@thistlm}
\newtheorem{theorem}{Theorem}[section]

\newtheorem{corollary}[theorem]{Corollary}

\newtheorem{lemma}[theorem]{Lemma}

\newtheorem*{theorem*}{Problem}
\theoremstyle{remark}
\newtheorem{remark}[theorem]{Remark}

\newcommand\keywords[1]{\textbf{Keywords}: #1}
\allowdisplaybreaks[4]
\usepackage{algorithmic}
\numberwithin{equation}{section}
\numberwithin{algocf}{section}
\newenvironment{MSCcodes}{\par\noindent\textbf{MSC codes:}\ }{\par}
\definecolor{revisionbrown}{RGB}{150,75,0}
\definecolor{REVISIONBROWN}{RGB}{150,75,0}
\colorlet{PURPLE}{purple}
\newcommand{\ResAt}[2]{\operatorname{Res}_{s=#1}\!\left(#2\right)}
\newcommand{\ResLHP}[2][revisionbrown]{\textcolor{black}{\operatorname*{Res}_{\Re(s)<0}}\!\left(#2\right)}
\AtBeginDocument{\overfullrule=0pt}
\hypersetup{
  hypertexnames=false,
  pdftitle={Policy Optimization in the Linear Quadratic Gaussian Problem: A Frequency Domain Perspective},
  pdfauthor={Haoran Li, Xun Li, Yuan-Hua Ni, and Xuebo Zhang}
}

\title{Policy Optimization in the Linear Quadratic Gaussian Problem: A Frequency Domain Perspective}
\date{\today}
\author{
	Haoran Li\footnote{Department of Applied Mathematics, The Hong Kong Polytechnic University, China. Email: \url{25039255r@connect.polyu.hk}}\qquad 
	Xun Li\footnote{Department of Applied Mathematics, The Hong Kong Polytechnic University, China. Email:  \url{li.xun@polyu.edu.hk.} This author is supported by RGC of Hong Kong grants 15226922 and 15225124, and partially from PolyU 4-ZZVB and 4-ZZP4.} \qquad 
	Yuan-Hua Ni\footnote{{Corresponding author.} College of Artificial Intelligence, Nankai University, Tianjin, China. Email:  \url{yhni@nankai.edu.cn}}\qquad
	Xuebo Zhang\footnote{College of Artificial Intelligence, Nankai University, Tianjin, China. Email:  \url{zhangxuebo@nankai.edu.cn}}}
	
\allowdisplaybreaks[4]
\begin{document}
	
	\maketitle
	
	\begin{abstract}
		 \textcolor{black}{Recent studies on directly parameterizing dynamic controllers and
		 	solving the linear quadratic Gaussian (LQG) problem over the resulting parameter
		 	set have exposed a central difficulty: the parameterized LQG problem typically has
		 	a highly nonconvex landscape, and no general method is currently available for
		 	certifying the global optimality of a candidate controller under such a
		 	parameterization.} Moreover, most existing numerical methods lack rigorous guarantees of global convergence. In this work, we address the aforementioned gaps and make the following contributions. First, we derive a necessary and sufficient condition for the global optimality of stationary points in a parameterized LQG problem. This condition reduces the verification of optimality to a test of the controllability and observability for a novel, specially constructed transfer function, yielding a precise and computationally tractable certificate. Furthermore, our condition provides a rigorous explanation for why traditional parameterizations can lead to suboptimal stationary points. Second, we elevate the controller parameter space from conventional finite-dimensional settings to the infinite-dimensional $\mathcal{RH}_\infty$ space and develop a gradient-based algorithm in \textcolor{black}{this space}, for which we provide a theoretical analysis establishing \textcolor{black}{sublinear convergence of the objective values to the global optimum}. Finally, representative numerical experiments validate the theoretical findings and demonstrate the practical viability of the proposed approach. 
	\end{abstract}
	

	
	\section{Introduction}
	In recent years, both theoretical advances and the widespread adoption of policy optimization methods in reinforcement learning (RL)~\cite{Mnih:2015jgp,Silver:2016gnq,rajeswaran2018learningcomplexdexterousmanipulation} have motivated a growing interest in applying policy optimization approaches to classic optimal control problems. Among these control problems, linear quadratic (LQ) formulations---including the Linear Quadratic Regulator (LQR) and Linear Quadratic Gaussian (LQG)---have become principal subjects of theoretical investigation for policy optimization methods, due to their foundational importance in both theory and practice~\cite{Corriou2004, Bertsekas1995Dynamic, 1995Robust}.
	
	For LQR problem, the explicit form and the global uniqueness of the optimal policy, as well as the problem's hidden convexity facilitate rigorous theoretical analysis of policy optimization methods~\cite{Fazel2018GlobalCO, Recht2018PG}. Furthermore, the analytical techniques and insights developed from policy optimization in LQR problem offer powerful tools and conceptual frameworks for more challenging problems, such as LQG problem and other areas in optimal and robust control~\cite{Hu2022TowardsAT,pmlr-v144-tang21a,doi:10.1137/140990309}. However, when moving from LQR problem to the more general LQG problem, which features partial state observation as well as process and measurement noises, most of the benign properties in LQR problem no longer hold. The optimization landscape of LQG problem is highly nonconvex and more difficult to characterize, making rigorous theoretical analysis of policy optimization methods much more challenging~\cite{zheng2023benignnonconvexlandscapesoptimal, pmlr-v144-tang21a}. In light of these differences, we briefly review key theoretical progress on policy optimization for both LQR and LQG problems, with emphasis on the distinctive analytical features and technical challenges of each case.
	
    Fatkhullin and Polyak (2020) \cite{Fatkhullin2020OptimizingSL} \textcolor{black}{analyze} gradient-based algorithm convergence rates for continuous-time and output feedback LQR, establishing linear convergence in these contexts. Fazel et al. (2018) \cite{Fazel2018GlobalCO} propose a model-free policy gradient approach for discrete-time LQR. Their work provides convergence guarantees for various policy gradient methods with different update rules (natural gradient descent, gradient descent, and Gauss-Newton step), as well as the sample complexity of gradient estimation via zeroth-order methods. Mohammadi et al. (2019) \cite{Mohammadi2019ConvergenceAS} extend the theoretical guarantees established by Fazel et al. (2018) \cite{Fazel2018GlobalCO} to the continuous-time case.
	The above studies systematically deepen our understanding of the optimization landscape for LQR problems. Furthermore, the analytical techniques and convergence tools developed {for LQR problem} provide a crucial theoretical foundation and motivating methodology for using policy optimization methods to \textcolor{black}{tackle} more general LQ problems. Beyond standard LQR, a variety of works generalize analysis to LQR variants with additional constraints or more complex system structures, such as $\mathcal{H}_2$/$\mathcal{H}_\infty$ mixed objectives~\cite{Hu2022TowardsAT}, Markov jump systems~\cite{9779410}, risk constraints~\cite{10005813}, and finite horizon problems~\cite{doi:10.1137/20M1382386}. Collectively, these studies further enrich the theoretical understanding of policy optimization in LQ problems and show that many analytical tools from standard LQR problem can be adapted to the broader contexts.
	
	Researchers begin exploring the applicability boundaries of policy optimization methods in LQG problem, building upon their successful application in solving LQR problem. However, the convergence guarantees for {solving LQR problems cannot be easily extended to the LQG setting}. Compared with LQR, LQG problem is subject to partial state observation as well as process and measurement noises, and its optimal solution is a dynamic regulator. These differences make the analysis of LQG problem more challenging than that of LQR. In response to these increasing complexities, prior work has primarily focused on two directions: the analysis of optimization landscape~\cite{pmlr-v144-tang21a,10124022,9993305} and convex reformulation techniques~\cite{Umenberger2022GloballyCP,zheng2023benignnonconvexlandscapesoptimal,zheng2024benignnonconvexlandscapesoptimal}.
	
	Tang et al. (2021) \cite{pmlr-v144-tang21a} provide an analysis of LQG connectivity, demonstrating through convex parameterization and manifold theory that the solution space consists of at most two path-connected components, each containing global optima. The work also identifies differences between LQG and LQR problems: the stationary points of LQG problem include many saddle points, suggesting that gradient-based methods may converge to suboptimal stationary point, and the lack of coerciveness and inconsistent stability margins in the LQG cost functional can potentially lead to {the divergence of policy gradient methods}. They further show that stationary points become optimal when the Lyapunov equations admit non-degenerate solutions.
	Subsequent researches further explore the optimality conditions of LQG problem's stationary points. Duan et al. (2024) \cite{10124022} extend the analysis of \cite{pmlr-v144-tang21a} to the Dynamical LQR (DLQR) problem, and their numerical experiments demonstrate that standard policy optimization methods may fail to converge to optimal policies under certain parameter settings. In the follow-up work \cite{9993305}, stochastic perturbations are introduced into the LQG setting to help policy optimization methods escape strict saddle points (i.e., stationary points with indefinite Hessians). Although this approach increases the probability of convergence, guarantees on the convergence rate of policy optimization methods are not provided.
	
	Notably, the optimal dynamic filtering problem, as a simplified variant of \textcolor{black}{the} LQG problem, gives rise to policy optimization methods with global convergence and sublinear convergence rates. Building on this problem, Umenberger et al. (2022) \cite{Umenberger2022GloballyCP} propose the Differential Convex Lifting (DCL) method, which extends \textcolor{black}{the} earlier convex parameterization theory. Their approach embeds the original optimization variables into an extended space and constructs a diffeomorphic mapping between the extended problem and a convex formulation, guaranteeing that all the stationary points in the extended domain correspond to globally optimal solutions of optimal dynamic filtering problem. Building on this theory, they prove that the optimal dynamic filtering process satisfies the weak Polyak-Lojasiewicz (PL) condition and further propose a transformation technique to ensure the coerciveness of cost functional. The main theoretical tools for this work are drawn from the broader field of robust optimal control. In this line of research, Scherer et al. \cite{599969} develop convex parameterization methods for LQG problem, which form the basis of DCL. However, the effectiveness of these approaches is often restricted by their heavy reliance on accurate system models, which limits their applicability in structured $\mathcal{H}_\infty$ synthesis \cite{1576856} and risk-sensitive scenarios~\cite{Zhang2021ProvablyEA}. These limitations, in turn, have motivated a recent surge of research on data-driven policy optimization methods. As a result, to address nonconvex problems such as LQG and dynamic filtering, researchers have begun to integrate policy optimization methods with convex parameterization theory to leverage the strengths of both methodologies.
	Zheng et al. (2023,2024) \cite{zheng2023benignnonconvexlandscapesoptimal,zheng2024benignnonconvexlandscapesoptimal} subsequently generalize DCL to LQG and $\mathcal{H}_{\infty}$ controller optimization. Their works demonstrate that the LQG problem is equivalent to a convex optimization problem. However, the construction of this \textcolor{black}{convex} problem requires accurate system models and assumptions of invertibility for \textcolor{black}{some} model parameters. Moreover, several works explore a novel parameterization method via input-output historical data for the discrete-time LQG problem \cite{sadamoto2024policy, fallah2025gradient}. These approaches formulate an objective function that satisfies the PL condition, thereby guaranteeing the linear convergence for policy gradient methods. Nevertheless, \textcolor{black}{this framework has not yet been extended to continuous-time systems.}
	
	The preceding theory on \textcolor{black}{the} LQG problem has largely been confined to time-domain analysis. To overcome the known limitations of classic approaches, we re-examine the problem from a modern, system-level perspective using frequency-domain tools and Linear Fractional Transformations (LFTs). This viewpoint allows us to move beyond treating the controller as a simple set of parameters and instead analyze its systemic interactions with the plant, and uncovers fundamental distinctions between LQR and LQG that explain long-standing challenges in \textcolor{black}{this} field. Our main contributions are as follows. 
	
	\begin{itemize}
		\item \textcolor{black}{Theoretical contribution: A necessary and sufficient condition for global optimality.}
		
		\textit{The Gap}. It is widely observed that policy gradient methods for LQG problem often become trapped in suboptimal stationary points \cite{pmlr-v144-tang21a, 9993305}. This contrasts sharply with LQR problem, where direct parameterization of controller gain is known to yield an optimization landscape with a unique, globally optimal stationary point. This discrepancy raises an open question: \textcolor{black}{Why does direct controller parameterization yield a benign landscape for LQR, yet induce numerous suboptimal stationary points for the seemingly similar LQG problem?} Meanwhile, the tools to certify that a point is globally optimal in LQG problem under direct parameterization are incomplete, as a complete characterization of its global optima (i.e., a necessary and sufficient condition) is still lacking  \cite{9993305, zheng2023benignnonconvexlandscapesoptimal, 599969}.
		
		\textit{Our Contribution.} We resolve both gaps by establishing a necessary and
		sufficient condition for global optimality under direct parameterization. This
		yields a tractable criterion for certifying whether a given controller is globally
		optimal for the LQG problem. \textcolor{black}{Leveraging this complete
			characterization, we identify the root cause of the suboptimal stationary points
			reported in prior work: direct parameterization acts as a projection that destroys the benign structure
			of the LQG objective. This provides a rigorous explanation for the pitfalls of
			applying direct parameterization to LQG control.}
		
		\item \textcolor{black}{Algorithmic contribution: A gradient-based algorithm in $\mathcal{RH}_\infty$ space.}
		
		\textit{The Limitation of Existing Methods}. Current approaches that provide theoretical guarantees for LQG problem typically operate by transforming the controller search space. These established methods can be classified into two main categories.
		
		\begin{itemize}
			\item[\textbf{1.}]  The first category of methods involves approximating the Youla parameterization with a {fixed, finite-order model}. This reduces the problem to a tractable, finite-dimensional convex program \cite{9031317}. However, selecting a controller order {a priori} introduces an inherent approximation error. Consequently, while the algorithm finds an optimal solution within the chosen subspace, this solution is for an approximated problem, creating a trade-off between controller complexity and suboptimality.
			\item[\textbf{2.}] The second category of methods reformulate LQG problem into a convex program (e.g., a semi-definite program), which can yield the exact global solution \cite{zheng2023benignnonconvexlandscapesoptimal}. The applicability of these methods, however, relies on several stringent prerequisites, such as the requirement of a full and precise model of system dynamics and prior knowledge of noise statistics.
		\end{itemize}
		
		\textit{Our Contribution}.
				In this work, we develop a gradient-based algorithm that operates directly in the $\mathcal{RH}_\infty$ space, bypassing the need for initial truncation.  By not confining the search to a fixed-order subspace, the proposed approach avoids the intrinsic approximation error associated with truncation methods. We \textcolor{black}{further} provide the global convergence guarantee for the algorithm and characterize its convergence rate. The framework can also be extended to a model-free setting, where it can learn from input/output data. 
	\end{itemize}
	
	\textit{Organization.} This paper is organized as follows. \textcolor{black}{Section~2 introduces notation.} Section \ref{pf} introduces LQG problem. Section \ref{ana} analyzes the optimality condition for stationary points. Section \ref{diff} presents a comparative analysis between LQR and LQG problems, highlighting structural differences under output feedback constraints. Section \ref{alg} develops a gradient-based optimization method with convergence analysis. Section \ref{num} validates the theoretical findings and {demonstrates} the practical viability of the proposed approach via {numerical experiments}.
	
	
	%
	\section{Notation}
	Let $\mathbb{N} = \{0,1,2,3,\ldots\}$ denote the set of natural numbers, and $\mathbb{R}$ and $\mathbb{C}$ be the sets of real and complex numbers, respectively. For any $z\in\mathbb C$, $\Re(z)$ denotes its real part. For any positive integers $m$ and $n$, we denote by $\mathbb{R}^m$ the set of real $m$-dimensional column vectors, and by $\mathbb{R}^{m \times n}$ and $\mathbb{C}^{m\times n}$ the set of real and complex $m \times n$ matrices, respectively. We let $I_m$ denote the $m \times m$ identity matrix, $e_{m\times n}(i,j)$ the $m\times n$ matrix with $1$ the $(i,j)$ entry and $0$ elsewhere for $1\leq i\leq m$, $1\leq j\leq n$, and $0_{m \times n}$ the $m \times n$ zero matrix. Given integers $1 \leq i \leq m$ and $1 \leq j \leq n$, define $	\mathbb{R}^{m\times n \mid i\times j} = \left\{A \in \mathbb{R}^{m\times n} : A = \begin{bmatrix}
		0_{i\times j} & *\\
		* & *
	\end{bmatrix} \right\}$
	as the set of $m\times n$ real matrices whose top-left $i \times j$ block is zero. Define the operator
	\[
	\mathbf{e}_{m\times n}(\cdot, i, j): \mathbb{R}^{m\times n} \to \mathbb{R}^{m\times n\mid i\times j},
	\]
	which, for any $A \in \mathbb{R}^{m \times n}$, sets the entries in the top-left $i \times j$ block of $A$ to zero, and leaves all other entries unchanged.
	For any matrix $A\in\mathbb{C}^{m\times n}$, $A^\top$, $A^*$, $\|A\|$, $\|A\|_F$ denote the transpose, conjugate transpose,  spectral (operator) norm, and Frobenius norm, respectively. The smallest and largest singular values of $A$ are denoted as $\sigma_{\min}(A)$ and $\sigma_{\max}(A)$, respectively. If $A$ is square and invertible, $A^{-1}$ denotes its inverse. For a real matrix $M$, $\operatorname{im}M$, $\ker M$,
	and $\operatorname{rank}M$ denote its column space,
	null space, and rank, respectively. For $A,B\in\mathbb{R}^{m\times n}$, the Frobenius inner product is defined by
	$\langle A,B\rangle_F:=\operatorname{tr}(A^\top B)$.
	For a linear subspace $\mathcal S$ of a real Euclidean space,
	$\dim\mathcal S$ denotes its dimension and
	$\mathcal S^\perp$ its Euclidean orthogonal complement.
	Given $M\in\mathbb R^{m\times n}$ and a linear subspace
	$\mathcal S\subseteq\mathbb R^n$, we write
	\[
	M\mathcal S:=\{Mx:x\in\mathcal S\},
	\qquad
	M|_{\mathcal S}:\mathcal S\longrightarrow\mathbb R^m,
	\quad x\longmapsto Mx,
	\]
	for the image of $\mathcal S$ under $M$ and the restriction
	of linear map $M$ to $\mathcal S$, respectively.
	In particular, set $\ker(M|_{\mathcal S})=\mathcal S\cap\ker M,\
	\operatorname{rank}(M|_{\mathcal S})=\dim(M\mathcal S)$
	When $M\in\mathbb R^{n\times n}$, a subspace
	$\mathcal S\subseteq\mathbb R^n$ is called $M$-invariant
	if $M\mathcal S\subseteq\mathcal S$.
	For subspaces $\mathcal S_1\subseteq\mathbb R^{n_1}$ and
	$\mathcal S_2\subseteq\mathbb R^{n_2}$, their Cartesian
	product is 
	\[
	\mathcal S_1\times\mathcal S_2
	:=
	\left\{
	\begin{bmatrix}x_1\\x_2\end{bmatrix}:
	x_1\in\mathcal S_1,\ x_2\in\mathcal S_2
	\right\}
	\subseteq\mathbb R^{n_1+n_2}.
	\]
	For any positive integer $d$,
	$L_2(0,\infty;\mathbb R^d)$ denotes the space of
	Lebesgue-measurable functions
	$f:(0,\infty)\to\mathbb R^d$ satisfying
	\[
	\|f\|_{L_2}
	:=
	\left(\int_0^\infty\|f(t)\|^2\,dt\right)^{1/2}
	<\infty.
	\]
	Consider a continuous-time, linear time-invariant (LTI) system
	\begin{equation}
		\begin{aligned}
			\dot{x}(t) &= A x(t) + B u(t),\\ 
			y(t) &= C x(t) + D u(t),
		\end{aligned}
	\end{equation}
	where $A$, $B$, $C$, and $D$ are real matrices with compatible dimensions. The corresponding transfer function matrix is
	\[
	G(s) = C(sI - A)^{-1}B + D,
	\]
	where $s \in \mathbb{C}$ is the Laplace variable. For notational convenience, we also represent this transfer function using the compact state-space form:
	\begin{equation}
		\left[\begin{array}{c|c}
			A& B\\
			\hline
			C& D
		\end{array}\right].
	\end{equation}
A transfer-function matrix is rational if each of its entries is a ratio of polynomials in $s\in\mathbb{C}$ with complex coefficients. A rational transfer-function matrix is stable if all its poles lie in the open left half-plane. It is proper if $\lim_{s\to\infty}G(s)$ is finite, and it is strictly proper if $\lim_{s\to\infty}G(s)=0$. Following the standard definitions in~\cite{1995Robust}, the $\mathcal{H}_2$ norm of a stable, strictly proper rational transfer-function matrix is
	\[
	\|G\|_{\mathcal{H}_2}:=\left(\int_{-\infty}^{\infty}\operatorname{tr}\big(G^*(j\omega)G(j\omega)\big)\,\frac{d\omega}{2\pi}\right)^{{1/2}},
	\]
and the $\mathcal{H}_\infty$ norm of a stable, proper rational transfer-function matrix is $\|G\|_{\mathcal{H}_\infty}:=\sup_{0\leq\omega<\infty}\sigma_{\max}(G(j\omega))$.

For $a,b\in\mathbb{N}\setminus\{0\}$, let $\mathcal R_{\mathrm p}^{a\times b}$ denote the set of
		proper real-rational transfer matrices of size $a\times b$,
		and define
		\[
		\mathcal R_{\mathrm{sp}}^{a\times b}
		:=
		\left\{
		G\in\mathcal R_{\mathrm p}^{a\times b}:G(\infty)=0
		\right\}.
		\]
	$\mathcal R_{\mathrm{sp}}^{a\times b}$ consists of all
		strictly proper real-rational transfer matrices in $\mathcal{R}_{\rm p}^{a\times b}$. Let $\mathcal{RH}_\infty$ denote the set of stable, proper, real-rational transfer-function matrices, and let $\mathcal{RH}_{\infty,0}$ denote its strictly proper subset. For $G,H\in\mathcal{RH}_{\infty,0}$ with compatible input and output dimensions, the $\mathcal{H}_2$ inner product is
	\[
	\langle G, H \rangle_{\mathcal{H}_2} = \int_{-\infty}^{\infty} \operatorname{tr}\big(G^*(j\omega) H(j\omega)\big)\, \frac{d\omega}{2\pi}.
	\]
	For a scalar- or matrix-valued meromorphic function $F$ and an isolated pole $p$, we denote the residue of $F$ at $p$ by
	\[
\ResAt{p}{F}:=\frac{1}{2\pi\mathrm{i}}\oint_{\Gamma_p}F(s)\,\mathrm{d}s,
	\]
	where $\Gamma_p$ is a sufficiently small positively oriented contour enclosing $p$ and no other pole; for a matrix-valued $F$, the definition is applied entrywise. If $F$ is rational and has no poles on the imaginary axis, we denote the sum of its residues over the open left half-plane by
	\[
	\ResLHP[blue]{F}:=\sum_{\substack{p\text{ is a pole of }F\\ \Re(p)<0}}\ResAt{p}{F}.
	\]
	
	\section{Problem formulation}\label{pf}
	
	In this section, we formulate the LQG problem as an optimization problem over a class of parameterized control policies. Let $(\Omega, \mathcal{F},\allowbreak \{\mathcal{F}_t\}_{t\geq 0},\allowbreak P)$ be a complete filtered probability space, where $\{\mathcal{F}_t\}_{t\geq 0}$ is the filtration generated by an $n$-dimensional Brownian motion $\{B_{t,1}\}_{t\geq 0}$ and an $m_2$-dimensional Brownian motion $\{B_{t,2}\}_{t\geq 0}$. Additionally, we assume that $\{B_{t,1}\}_{t\geq 0}$ and $\{B_{t,2}\}_{t\geq 0}$ are mutually independent.
	The continuous-time LQG problem (see \cite{Tang2023, zheng2024benignnonconvexlandscapesoptimal, 9993305}) is
	\begin{equation}
		\begin{aligned}
			&\underset{u_t}{\text{min}} \lim_{T \to \infty} \frac{1}{T} \mathbb{E} \left[ \int_0^T \left( X_t^\top Q X_t + U_t^\top R U_t \right) {\rm d}t \right], \\
			&~\text{s.t.} \quad 
			\left\{
			\begin{array}{l}
				{\rm d}{X}_t =( A X_t + B U_t){\rm d}t + \mathcal{W}^{\frac{1}{2}}{\rm d}B_{t,1}, \\
				{\rm d}Y_t = C X_t {\rm d}t + \mathcal{V}^{\frac{1}{2}}{\rm d}B_{t,2},
			\end{array}
			\right.
		\end{aligned}
		\label{eq_lqg_in}
	\end{equation}
	where \(A \in \mathbb{R}^{n \times n}\), \(B \in \mathbb{R}^{n \times m_1}\), and \(C \in \mathbb{R}^{m_2 \times n}\) are system matrices, and $(X_t, Y_t)$ denotes the state and output. $\{B_{t,1}\}_{t\geq 0}$ and $\{B_{t,2}\}_{t\geq 0}$ denote the stochastic disturbance process and the measurement process of system \eqref{eq_lqg_in}, respectively.  The process and measurement noise covariance matrices satisfy $\mathcal{W} \succeq 0$ and $\mathcal{V} \succ 0$, while \(Q \succeq 0\) and \(R \succ 0\) define the quadratic performance index. The control input $U_t$ is required to be progressively measurable with respect to $\sigma(Y_\tau: 0 \le \tau \le t)$, where $\sigma(Y_\tau: 0 \le \tau \le t)$ denotes the smallest $\sigma$-field generated by the random variable $Y_{\tau},\ 0 \le \tau \le t$. Furthermore, to guarantee that the expected quadratic cost is well-defined, any admissible control $U$ must satisfy the mean-square integrability condition:
		\[
		\mathbb{E} \left[ \int_0^T \| U_t \|^2 {\rm d}t \right] < \infty, \quad \forall T > 0.
		\]
 The well-known optimal controller of problem \eqref{eq_lqg_in} is constructed based on the solutions to the following Riccati equations~\cite[Theorem~14.7]{1995Robust}.
	\begin{equation}
		\begin{aligned}
			&A^\top P + P A - P B R^{-1} B^\top P + Q = 0,\quad K^* = R^{-1} B^\top P, \\
			&A H + H A^\top - H C^\top \mathcal{V}^{-1} C H + \mathcal{W} = 0,\quad L^* = H C^\top \mathcal{V}^{-1}. 
		\end{aligned}
		\label{eq_riccati}
	\end{equation}
	It should be noted that under the stabilizability (i.e., $(A,B)$ and $(A,\mathcal{W}^{1/2})$ are stabilizable) and detectability (i.e., $(C, A)$ and $(Q^{1/2}, A)$ are detectable) assumptions each of~\eqref{eq_riccati} admits a unique stabilizing positive semi-definite solution~\cite[Corollary~13.8]{1995Robust}. The optimal feedback gain $K$ and the estimator gain $L$ are subsequently used to construct the optimal controller:
	\begin{equation}
		\begin{aligned}
			&{\rm d}\hat{X}_t = (A - B K^* - L^* C) \hat{X}_t {\rm d}t + L^* {\rm d}Y_t, \\
			&U_t = -K^* \hat{X}_t.
		\end{aligned}
		\label{opt_lqg}
	\end{equation}
	The optimality of controller \eqref{opt_lqg} follows from the known separation principle~\cite{1995Robust}.
	
	While LQG problem admits the explicit optimal solution, real-world scenarios frequently involve model uncertainty or additional constraints that preclude closed-form solutions. To address these challenges, it has become a common practice to parameterize controllers and synthesize them via policy optimization, enabling data-driven and gradient-based approaches. Motivated by these considerations, recent literature on LQG controller synthesis and policy optimization~\cite{Tang2023,zheng2023benignnonconvexlandscapesoptimal,9993305} investigates the set of controllers admitting the following parameterized state-space realization:
	\begin{equation}
		\left\{
		\begin{array}{l}
			{\rm d}\hat{X}_t = A_{\mathcal{K}} \hat{X}_t {\rm d}t + B_{\mathcal{K}} {\rm d}Y_t, \\
			U_t = C_{\mathcal{K}} \hat{X}_t,
		\end{array}
		\right.
		\label{eq_parameterization}
	\end{equation}
with $
A_{\mathcal{K}} \in \mathbb{R}^{q \times q}, \,
B_{\mathcal{K}} \in \mathbb{R}^{q \times m_2}, \,
C_{\mathcal{K}} \in \mathbb{R}^{m_1 \times q},$
	where $q$ encompasses the minimal realization of the optimal LQG controller \eqref{opt_lqg}. This structural assumption ensures that the policy search space is sufficiently rich to exhibit the globally optimal control as a feasible point.
	
	Building on this parameterization, previous works \cite{zheng2023benignnonconvexlandscapesoptimal,9993305} have proposed the following policy optimization problem over parameterized dynamic controllers \eqref{eq_parameterization}
	\begin{equation} 
		\begin{aligned}
			&\underset{\mathcal{K} \in \mathbb{K}_q}{\text{min}}\quad J(\mathcal{K}) := \lim_{T \to \infty} \frac{1}{T} \mathbb{E} \left[ \int_0^T \left( X_t^\top Q X_t + U_t^\top R U_t \right) {\rm d}t \right] \\
			&~\text{s.t.} \quad
			\begin{cases}
				{\rm d}X_t = (A X_t + B U_t) {\rm d}t + \mathcal{W}^{\frac{1}{2}}{\rm d}B_{t,1}, \\
				{\rm d}\hat{X}_t = A_{\mathcal{K}} \hat{X}_t {\rm d}t + B_{\mathcal{K}} {\rm d}Y_t, \\
				{\rm d}Y_t = C X_t {\rm d}t + \mathcal{V}^{\frac{1}{2}}{\rm d}B_{t,2}, \\
				U_t = C_{\mathcal{K}} \hat{X}_t,
			\end{cases}
		\end{aligned}
		\label{op_lqg}
	\end{equation}
	where 
	\begin{equation}\label{cons_k}
		\mathcal{K} = \begin{bmatrix}
			0 & C_{\mathcal{K}} \\
			B_{\mathcal{K}} & A_{\mathcal{K}}
		\end{bmatrix} \in \mathbb{R}^{(q+m_1)\times (q+m_2)}
	\end{equation}
	parameterizes the dynamic controller, and $\mathbb{K}_q$ denotes the set of all such $\mathcal{K}$ that stabilizes the corresponding closed-loop system~\eqref{op_lqg}, ensuring the problem being well-posed. This parameterization is chosen so that the class $\mathbb{K}_q$ contains  the optimal LQG controller \eqref{opt_lqg}. Specifically, we assume $q$ is at least as large as the dimension of the minimal realization of optimal controller \eqref{opt_lqg}. For any positive integer $r$, let $\mathbb K_r$ denote
		the set of real controller realizations
		$\mathcal L=
		\begin{bmatrix}
			0&C_{\mathcal L}\\
			B_{\mathcal L}&A_{\mathcal L}
		\end{bmatrix}$,
		where
		$A_{\mathcal L}\in\mathbb R^{r\times r}$,
		$B_{\mathcal L}\in\mathbb R^{r\times m_2}$, and
		$C_{\mathcal L}\in\mathbb R^{m_1\times r}$ such that
		\begin{equation}
			\label{eq:internal-stability}
			\mathcal A_{\mathcal L}
			:=
			\begin{bmatrix}
				A&BC_{\mathcal L}\\
				B_{\mathcal L}C&A_{\mathcal L}
			\end{bmatrix}
		\end{equation}
is Hurwitz. We define
		\begin{equation}
			\label{eq:all-order-stabilizing}
			\mathbb K_{\mathrm{stab}}
			:=
			\bigcup_{r=1}^{\infty}\mathbb K_r
		\end{equation}
		for the class of internally stabilizing, strictly proper
		LTI controller realizations of arbitrary finite order.
	\begin{remark}
Some existing works use white noise to describe the LQG problem instead of the SDE formulation in \eqref{eq_lqg_in} (see \cite{Tang2023,zheng2023benignnonconvexlandscapesoptimal,9993305}). Generally,
		the Gaussian white noise is the distributional derivative
		of Brownian motion.
		Its integral against a deterministic $g\in L_2(0,T)$ is
		defined by mean-square extension, rather than pathwise
		Lebesgue integration;
		see
		\cite[Example~3.56]{hairer2009introduction}.
		Hence, the Gaussian white-noise formulation in
		\cite{zheng2023benignnonconvexlandscapesoptimal}
		and the SDE formulation~\eqref{eq_lqg_in} describe the same
		 LQG problem when the system data, noise
		intensities, and controller realizations are matched.
	\end{remark}
	
	\section[Optimality condition of the LQG problem]{Optimality condition of problem \eqref{op_lqg}}\label{ana}
	
	In this section, we first introduce some basic properties of problem \eqref{op_lqg}, and then establish a necessary and sufficient condition for a feasible point of problem \eqref{op_lqg} to be globally optimal.
	\subsection[Properties of the LQG problem]{Properties of problem \eqref{op_lqg}}
	To facilitate the subsequent analysis and computation, it is convenient to represent the closed-loop system of \eqref{op_lqg} using an augmented state variable $\xi_t = [X_t^\top \ \hat{X}_t^\top ]^\top$. With this state variable, the closed-loop system can be expressed as
	\begin{equation}
		{\rm d}\xi_t =
		\underbrace{
			\left[\begin{array}{cc}
				A & B C_{\mathcal{K}} \\
				B_{\mathcal{K}} C & A_{\mathcal{K}}
			\end{array}\right]
		}_{\mathcal{A}}
		\xi_t {\rm d}t +
		\left[\begin{array}{cc}
			\mathcal{W}^{\frac{1}{2}} & 0 \\
			0 & B_{\mathcal{K}}\mathcal{V}^{\frac{1}{2}}
		\end{array}\right]
		\begin{bmatrix}
			{\rm d}B_{t,1} \\
			{\rm d}B_{t,2}
		\end{bmatrix}.
		\label{eq_augmented_dynamics}
	\end{equation}
	Thus, for any $\mathcal{K}\in\mathbb{K}_q$, the performance index in \eqref{op_lqg} admits two equivalent representations in terms of solutions to two Lyapunov equations (see \cite{mei2021optimal,Tang2023}):
	\begin{equation}
		\begin{aligned}
			J(\mathcal{K}) &= \operatorname{tr}\left(
			\begin{bmatrix}
				\mathcal{W} & 0 \\
				0 & B_{\mathcal{K}} \mathcal{V} B_{\mathcal{K}}^\top
			\end{bmatrix}
			\mathcal{P}_{\mathcal{K}}
			\right) \\
			& = \operatorname{tr}\left(
			\begin{bmatrix}
				Q & 0 \\
				0 & C_{\mathcal{K}}^\top R C_{\mathcal{K}}
			\end{bmatrix}
			\Sigma_{\mathcal{K}}
			\right),
		\end{aligned}
		\label{eq_cost_lyap}
	\end{equation}
	where
\begin{equation}\label{pwblock}
		\mathcal{P}_{\mathcal{K}} =
	\begin{bmatrix}
		P^{\mathcal{K}}_{11} & P^{\mathcal{K}}_{12} \\ 
		P^{\mathcal{K}}_{21} & P^{\mathcal{K}}_{22}
	\end{bmatrix},\qquad
	\Sigma_{\mathcal{K}} =
	\begin{bmatrix}
		\Sigma^{\mathcal{K}}_{11} & \Sigma^{\mathcal{K}}_{12} \\ 
		\Sigma^{\mathcal{K}}_{21} & \Sigma^{\mathcal{K}}_{22}
	\end{bmatrix}
\end{equation}
	solve the Lyapunov equations
	\begin{align}
			\mathcal{A}^\top \mathcal{P}_{\mathcal{K}} + \mathcal{P}_{\mathcal{K}} \mathcal{A} &=
			-\begin{bmatrix}
				Q & 0 \\
				0 & C_{\mathcal{K}}^\top R C_{\mathcal{K}}
			\end{bmatrix}, \label{eq_lyap1}\\
			\mathcal{A} \Sigma_{\mathcal{K}} + \Sigma_{\mathcal{K}} \mathcal{A}^\top &=
			-\begin{bmatrix}
				\mathcal{W} & 0 \\
				0 & B_{\mathcal{K}} \mathcal{V} B_{\mathcal{K}}^\top
			\end{bmatrix},
			\label{eq_lyap2}
	\end{align}
	and, for clarity, the block matrices have the following dimensions:
\begin{equation}\label{cov}
		P^{\mathcal{K}}_{11},\, \Sigma^{\mathcal{K}}_{11} \in \mathbb{R}^{n \times n}, 
	\quad P^{\mathcal{K}}_{12},\, \Sigma^{\mathcal{K}}_{12} \in \mathbb{R}^{n \times q}, 
	\quad P^{\mathcal{K}}_{22},\, \Sigma^{\mathcal{K}}_{22} \in \mathbb{R}^{q \times q}.
\end{equation}
	Thus, the search for optimal controllers reduces to identifying the minimizer of \eqref{eq_cost_lyap} over $\mathbb{K}_q$. For simplicity, we let $\mathcal{K}^\star \in \mathbb{K}_q$ as a stationary point of problem \eqref{op_lqg}, if the gradient of  $J(\mathcal{K}^\star)$ satisfies
	\[
	\nabla J(\mathcal{K}^\star)  = 0.
	\]
	
	\noindent While policy gradient and related direct optimization methods have been widely applied to problem \eqref{op_lqg} \cite{Tang2023,zheng2023benignnonconvexlandscapesoptimal,9993305}, and for the theoretical analysis of these methods, most existing results offer only one-sided conditions (that is necessary or sufficient condition) to verify the global optimality of stationary point. Meanwhile, these methods lack the convergence analysis. 
	This work systematically addresses the aforementioned issues through three interconnected contributions. 
	{First}, in this section, we {establish a necessary and sufficient condition} for a point to be a global optimum of problem~\eqref{op_lqg}. This rigorous criterion provides the first definitive tool to verify global optimality. 
	{Second}, leveraging this characterization, the analysis in Section~\ref{diff} reveals the root cause of suboptimal stationary points in the optimization landscape of problem~\eqref{op_lqg}. 
	{Finally}, we introduce in Section~\ref{alg} {a novel gradient-based algorithm to solve problem~\eqref{op_lqg}}. By operating within a parameterization that ensures a broad class of admissible controllers, our algorithm is specifically designed to circumvent the identified landscape challenges and is proven to converge to the optimal value.

	\subsection{Optimal condition analysis} 
	
	The challenge of certifying the optimality for stationary points in problem~\eqref{op_lqg} has spurred significant recent research. A notable line of work involves leveraging {convex parameterization techniques} to reveal hidden convex structures within the problem \cite{zheng2023benignnonconvexlandscapesoptimal,zheng2024benignnonconvexlandscapesoptimal}. These developments have deepened the understanding of optimization landscape and, crucially, enabled the certification of global optimality for certain stationary points under specific conditions.
	
	\begin{lemma}\cite{zheng2024benignnonconvexlandscapesoptimal}\label{lem_conv}
		If $q=n$ and $\mathcal{K}^*$ is a stationary point of problem \eqref{op_lqg} that satisfies
		\begin{equation}
			\begin{aligned}\label{eq_con_ve}
				\mathcal{P}_{\mathcal{K}^*} =
				\begin{bmatrix}
					P^{\mathcal{K}^*}_{11} & P^{\mathcal{K}^*}_{12} \\
					P^{\mathcal{K}^*}_{21} & P^{\mathcal{K}^*}_{22}
				\end{bmatrix}\succ 0, \quad \det(P^{\mathcal{K}^*}_{12})\neq 0,
			\end{aligned}
		\end{equation}
		or
		\begin{equation}\label{eq_con_ve2}
			\Sigma_{\mathcal{K}^*} =
			\begin{bmatrix}
				\Sigma^{\mathcal{K}^*}_{11} & \Sigma^{\mathcal{K}^*}_{12} \\
				\Sigma^{\mathcal{K}^*}_{21} & \Sigma^{\mathcal{K}^*}_{22}
			\end{bmatrix}\succ 0, \quad \det(\Sigma^{\mathcal{K}^*}_{12})\neq 0,
		\end{equation}
		then $\mathcal{K}^*$ is an optimal solution to problem \eqref{op_lqg}.
	\end{lemma}
	
	
 However, the sufficient conditions provided by Lemma~\ref{lem_conv} and the related results in \cite{zheng2023benignnonconvexlandscapesoptimal,zheng2024benignnonconvexlandscapesoptimal} are applicable only to a specific subset of controllers in $\mathbb{K}_q$ that satisfy the structural constraints \eqref{eq_con_ve} or \eqref{eq_con_ve2}. Consequently, for any stationary point corresponding to a controller that fails to meet these strict requirements, existing convex reformulation techniques cannot determine whether it is globally optimal.
	
	To overcome this fundamental limitation and establish a truly general criterion, this work pivots to a different approach founded on frequency-domain analysis and linear fractional transformations (LFTs). As a foundational step for this analysis, we first reformulate problem~\eqref{op_lqg} as an equivalent~$\mathcal{H}_2$ optimization problem. This reformulation provides a necessary framework to derive our necessary and sufficient condition without relying on the restrictive structural assumptions of prior works. Let
	\begin{equation}
		\begin{aligned}
			& \underset{\mathcal{K}\in\mathbb{K}_q}{\text{min}} \; && \|G_{\mathcal{K}}\|_{\mathcal{H}_2}^2 \\
			&~~\text{s.t.} \; && G_{\mathcal{K}}(s) = \mathcal{C}(sI - \mathcal{A})^{-1}\mathcal{B} 
		\end{aligned}
		\label{eq:h2_problem}
	\end{equation}
	with
	\begin{equation}
		\mathcal{A} = \begin{bmatrix} A & BC_\mathcal{K} \\ B_\mathcal{K}C & A_\mathcal{K} \end{bmatrix}, 
		\mathcal{C} = \begin{bmatrix} Q^{\frac{1}{2}} & 0 \\ 0 & R^{\frac{1}{2}}C_\mathcal{K} \end{bmatrix}, 
		\mathcal{B} = \begin{bmatrix} \mathcal{W}^{\frac{1}{2}} & 0 \\ 0 & B_\mathcal{K}\mathcal{V}^{\frac{1}{2}} \end{bmatrix}.
	\end{equation}
	\begin{lemma}
		\label{th1}
		The policy optimization problem \eqref{op_lqg} is equivalent to the $\mathcal{H}_2$ norm optimization problem \eqref{eq:h2_problem} in the sense that both problems have the same feasible set, and the two cost functionals are equal for any  $\mathcal{K}\in\mathbb{K}_q$.
	\end{lemma}
	\begin{proof}
		Note that the feasible set of \eqref{eq:h2_problem} is the same as that of problem \eqref{op_lqg}. Then, for any $\mathcal{K}\in \mathbb{K}_q$ and by Parseval's theorem, the objective $\|G_{\mathcal{K}}\|_{\mathcal{H}_2}^2$ admits an equivalent time-domain representation
		\begin{equation}
			\|G_{\mathcal{K}}\|_{\mathcal{H}_2}^2 = \operatorname{tr}\left( \int_{0}^{\infty} \mathcal{B}^\top e^{\mathcal{A}^\top t}\mathcal{C}^\top \mathcal{C}e^{\mathcal{A}t}\mathcal{B} dt \right). \label{eq:h2_time_domain}
		\end{equation}
		As established in \eqref{eq_cost_lyap} above, this expression equals 
		\begin{equation}
			\operatorname{tr}\left( \begin{bmatrix} \mathcal{W} & 0 \\ 0 & B_\mathcal{K}\mathcal{V}B_\mathcal{K}^\top \end{bmatrix} \mathcal{P}_{\mathcal{K}} \right)=J(\mathcal{K}). \label{eq:h2_lyap}
		\end{equation}
		This completes the proof.
	\end{proof}

	Building on the equivalence between problem \eqref{op_lqg} and problem \eqref{eq:h2_problem}, we analyze the optimality condition of problem \eqref{op_lqg} by considering the perturbations for the $\mathcal{H}_2$ norm
	optimization problem. This reformulation allows us to analyze the system's performance using the deterministic $\mathcal{H}_2$ norm. In the subsequent analysis, we only analyze the properties of the $\mathcal{H}_2$ representation of problem \eqref{eq:h2_problem} instead of directly analyzing problem \eqref{op_lqg}. Moreover, we define 
		\begin{align*}
			\mathcal{K}_{\text{opt}}:=\left[\begin{matrix}
				0& C_{\mathcal{K}_{\text{opt}}}\\
				B_{\mathcal{K}_{\text{opt}}} & A_{\mathcal{K}_{\text{opt}}}
			\end{matrix}\right]\in \mathbb{R}^{(q+m_1)\times (q+m_2)}
		\end{align*}
		satisfying $C_{\mathcal{K}_{\text{opt}}}(sI-A_{\mathcal{K}_{\text{opt}}})^{-1}B_{\mathcal{K}_{\text{opt}}}=\textcolor{black}{-}K^*(sI-(A - B K^* - L^* C))^{-1}L^*$ and $A_{\mathcal{K}_{\text{opt}}}\in\mathbb{R}^{q\times q}$. Based on the equivalence, it is obvious that $\mathcal{K}_{\text{opt}}$ is the optimal solution to problem \eqref{eq:h2_problem}. 
	
	Let $\mathcal{K}$ be any controller parameter belonging to $\mathbb{K}_q$, and consider a perturbed controller $\mathcal{K}' = \mathcal{K} + \Delta\mathcal{K}\in\mathbb{K}_q$, where the perturbation $\Delta\mathcal{K}$ is defined as
	\begin{align*}
		&\Delta C_\mathcal{K} := C_{\mathcal{K}'} - C_\mathcal{K},\quad 
		\Delta B_\mathcal{K} := B_{\mathcal{K}'} - B_\mathcal{K},\quad
		\Delta A_\mathcal{K} := A_{\mathcal{K}'} - A_\mathcal{K},\\
		&\Delta \mathcal{K} :=\begin{bmatrix}
			0 & \Delta C_\mathcal{K}\\
			\Delta B_\mathcal{K} & \Delta A_\mathcal{K}
		\end{bmatrix}.
	\end{align*}
Unlike conventional optimization methods, we explicitly extract the controller perturbation $\Delta\mathcal{K}$ as an external feedback block.
	Setting $\Delta\mathcal K=0$, 
	we have the following deterministic state-space realization
	of the nominal transfer matrix $G_{\mathcal K}$;
	see \cite[Sections~2.2.2 and~2.5]{boyd1991linear}:
\begin{align}\label{sp_G}
	\begin{bmatrix} \dot{x}_t \\ \dot{\hat{x}}_t \end{bmatrix} = 
	\begin{bmatrix} A & BC_\mathcal{K} \\ B_\mathcal{K}C & A_\mathcal{K} \end{bmatrix}
	\begin{bmatrix} x_t \\ \hat{x}_t \end{bmatrix} + 
	\begin{bmatrix} \mathcal{W}^{\frac{1}{2}} & {0}  \\ {0} & B_\mathcal{K}\mathcal{V}^{\frac{1}{2}}  \end{bmatrix}
	\begin{bmatrix} \omega_t\ \\ v_t\ \end{bmatrix},\quad 
	\begin{bmatrix} y_{1t} \\ y_{2t} \end{bmatrix} = 
	\begin{bmatrix} Q^\frac{1}{2} & {0} \\ {0} & R^\frac{1}{2}C_\mathcal{K} \end{bmatrix}
	\begin{bmatrix} x_t \\ \hat{x}_t \end{bmatrix}.
\end{align}
	Here, $\omega\in L_2(0,\infty;\mathbb R^n)$ and
	$v\in L_2(0,\infty;\mathbb R^{m_2})$ are auxiliary
	deterministic inputs, and $x_0=\hat{x}_0=0$.
	Taking the Laplace transform of \eqref{sp_G} gives
	$G_{\mathcal K}(s)=
	\mathcal C(sI-\mathcal A)^{-1}\mathcal B$
	as the transfer matrix from
	$(\omega^\top,v^\top)^\top$ to
	$(y_1^\top,y_2^\top)^\top$.
	Then, for a perturbed controller $\mathcal{K}' = \mathcal{K} + \Delta\mathcal{K}$, the input-output behavior of  closed-loop system $G_{\mathcal{K}'}$ can be reformulated as a feedback interconnection between the state-space representation of $G_{\mathcal{K}}$ (parameterized by $\mathcal{K}$) and the perturbation block $\Delta\mathcal{K}$  \cite{1995Robust}. 
	 The corresponding LFT representation of ${G}_{\mathcal{K}^\prime}$ is given by
	\begin{align}
		&\begin{bmatrix} \dot{x}_t \\ \dot{\hat{x}}_t \end{bmatrix} = 
		\begin{bmatrix} A & BC_\mathcal{K} \\ B_\mathcal{K}C & A_\mathcal{K} \end{bmatrix}
		\begin{bmatrix} x_t \\ \hat{x}_t \end{bmatrix} + 
		\begin{bmatrix} \mathcal{W}^{\frac{1}{2}} & {0} & B & {0} \\ {0} & B_\mathcal{K}\mathcal{V}^{\frac{1}{2}} & {0} & I \end{bmatrix}
		\begin{bmatrix} \omega_t\ \\ v_t\ \\ \hat{y}_{1t} \\ \hat{u}_{1t} \end{bmatrix}, \label{eq:lft_state} \\
		&\begin{bmatrix} y_{1t} \\ y_{2t} \\ y_t \\ y_{3t} \end{bmatrix} = 
		\begin{bmatrix} Q^\frac{1}{2} & {0} \\ {0} & R^\frac{1}{2}C_\mathcal{K} \\ C & {0} \\ {0} & I \end{bmatrix}
		\begin{bmatrix} x_t \\ \hat{x}_t \end{bmatrix} + 
		\begin{bmatrix} 
			{0} & {0} & {0} & {0} \\ 
			{0} & {0} & R^\frac{1}{2} & {0} \\ 
			{0} & \mathcal{V}^{\frac{1}{2}} & {0} & {0} \\ 
			{0} & {0} & {0} & {0} 
		\end{bmatrix}
		\begin{bmatrix} \omega_t \\ v_t \\ \hat{y}_{1t} \\ \hat{u}_{1t} \end{bmatrix}, \label{eq:lft_output} \\
		&\begin{bmatrix} \hat{y}_{1t} \\ \hat{u}_{1t} \end{bmatrix} = 
		\Delta\mathcal{K} \begin{bmatrix} y_t \\ y_{3t} \end{bmatrix}. \label{eq:lft_perturb}
	\end{align}
	In \eqref{sp_G}-\eqref{eq:lft_perturb},
	$\omega$ and $v$ are auxiliary deterministic
	exogenous inputs used to realize the corresponding
	closed-loop transfer matrices.
	Accordingly, lowercase letters $x_t$ and $\hat{x}_t$
	denote the states of these deterministic realizations,
	whereas $X_t$ and $\hat{X}_t$ denote the stochastic states
	in \eqref{op_lqg}.
	By Lemma~4.2, the transfer-matrix formulation preserves
	the original objective value for each admissible controller:
	$J(\mathcal K)=\|G_{\mathcal K}\|_{\mathcal H_2}^2$. The statistical properties and intensities of the original stochastic noises are precisely preserved by scaling these deterministic inputs with the Cholesky factors of their respective covariance matrices, $\mathcal{W}^{1/2}$ and $\mathcal{V}^{1/2}$. Here, $(y_t, y_{3t})$ serves as the output signals from $G_{\mathcal{K}}$ to $\Delta\mathcal{K}$, while $(\hat{y}_{1t}, \hat{u}_{1t})$ is the corresponding feedback input.
	This frequency-domain formulation is particularly advantageous for sensitivity and robustness analysis, enabling a systematic investigation of controller optimality condition under perturbations.
	Then, the corresponding transfer function $G_{\mathcal{K}'}(s)$ is given based on the lower LFT
	\begin{align}
		G_{\mathcal{K}'}(s) &= \mathcal{F}_l \left( \mathbf{M}_\mathcal{K}(s), \Delta\mathcal{K} \right) = \mathbf{M}_{11}^\mathcal{K}(s) + \mathbf{M}_{12}^\mathcal{K}(s) \Delta\mathcal{K} (I - \mathbf{M}_{22}^\mathcal{K}(s) \Delta\mathcal{K})^{-1} \mathbf{M}_{21}^\mathcal{K}(s), \label{eq:lft_transfer}
	\end{align}
	where 
	\[\mathbf{M}_\mathcal{K}(s):=\begin{bmatrix}
		\mathbf{M}_{11}^\mathcal{K}(s) & \mathbf{M}_{12}^\mathcal{K}(s)\\
		\mathbf{M}_{21}^\mathcal{K}(s) & \mathbf{M}_{22}^\mathcal{K}(s)
	\end{bmatrix}
	\]
	is the nominal transfer function matrix with realization
	\begin{equation}
		\begin{aligned}\label{m_k}
			\mathbf{M}_{11}^\mathcal{K}(s) &:= \mathcal{C}(sI - \mathcal{A})^{-1}\mathcal{B}=G_\mathcal{K}(s), \\
			\mathbf{M}_{12}^\mathcal{K}(s) &:= \mathcal{C}(sI - \mathcal{A})^{-1} \begin{bmatrix} B & {0} \\ {0} & I \end{bmatrix} + \begin{bmatrix} {0} & {0} \\ R^\frac{1}{2} & {0} \end{bmatrix}, \\
			\mathbf{M}_{21}^\mathcal{K}(s) &:= \begin{bmatrix} C & {0} \\ {0} & I \end{bmatrix} (sI - \mathcal{A})^{-1} \mathcal{B} + \begin{bmatrix} {0} & \mathcal{V}^{\frac{1}{2}} \\ {0} & {0} \end{bmatrix}, \\
			\mathbf{M}_{22}^\mathcal{K}(s) &:= \begin{bmatrix} C & {0} \\ {0} & I \end{bmatrix} (sI - \mathcal{A})^{-1} \begin{bmatrix} B & {0} \\ {0} & I \end{bmatrix}.
		\end{aligned}
	\end{equation}
	Letting $\Delta_{\mathcal{K}^\prime} G(s):=\mathbf{M}_{12}^\mathcal{K}(s) \Delta\mathcal{K} (I - \mathbf{M}_{22}^\mathcal{K}(s) \Delta\mathcal{K})^{-1} \mathbf{M}_{21}^\mathcal{K}(s)$, the $\mathcal{H}_2$ norm $\|G_{\mathcal{K}'}\|_{\mathcal{H}_2}$ can be expanded as
	\begin{equation}
		\|G_{\mathcal{K}'}\|_{\mathcal{H}_2}^2 = \|G_\mathcal{K}\|_{\mathcal{H}_2}^2 + 2 \left\langle G_\mathcal{K}, \Delta_{\mathcal{K}^\prime} G \right\rangle_{\mathcal{H}_2} + \mathcal{O}(\|\Delta\mathcal{K}\|^2), \label{eq:h2_expansion}
	\end{equation}
	where the first-order variation term is
	\begin{equation}
		\left\langle G_\mathcal{K}, \Delta_{\mathcal{K}^\prime} G \right\rangle_{\mathcal{H}_2} = \operatorname{tr}\left( \int_{-\infty}^{\infty} G_\mathcal{K}^*(j\omega) \mathbf{M}_{12}^\mathcal{K}(j\omega) \Delta\mathcal{K} (I - \mathbf{M}_{22}^\mathcal{K}(j\omega) \Delta\mathcal{K})^{-1} \mathbf{M}_{21}^\mathcal{K}(j\omega) \frac{d\omega}{2\pi}\right). \label{eq:first_order}
	\end{equation}
	Observing that $G_{\mathcal{K}'}$ is linear in $\Delta\mathcal{K} (I - \mathbf{M}_{22}^\mathcal{K} \Delta\mathcal{K})^{-1}$, we can express the higher-order term as
	\begin{equation}
		\mathcal{O}(\|\Delta\mathcal{K}\|^2) = \left\| \mathbf{M}_{12}^\mathcal{K} \Delta\mathcal{K} (I - \mathbf{M}_{22}^\mathcal{K} \Delta\mathcal{K})^{-1} \mathbf{M}_{21}^\mathcal{K} \right\|_{\mathcal{H}_2}^2 \geq 0.\label{eq:high_order}
	\end{equation}
	The nonnegativity of this higher-order term implies that a sufficient and necessary condition for $\mathcal{K}$ to be an optimal solution of problem~\eqref{op_lqg} is the vanishing of first-order variation:
	\begin{equation}
		\left\langle G_\mathcal{K}, \Delta_{\mathcal{K}^\prime} G \right\rangle_{\mathcal{H}_2} = 0, \quad \forall \mathcal{K}' \in \mathbb{K}_q. \label{eq:stationarity_condition}
	\end{equation}
	This is proved in the following lemma.
	\begin{lemma}\label{lem1}
		Controller $\mathcal{K} \in \mathbb{K}_q$ is an optimal solution to problem~\eqref{op_lqg}, if and only if \eqref{eq:stationarity_condition} holds.
	\end{lemma}
	\begin{proof}
If \eqref{eq:stationarity_condition} holds, then it follows from \eqref{eq:h2_expansion} and \eqref{eq:high_order} that for any $\mathcal{K}^\prime\in\mathbb{K}_q$,
			\begin{align*}
				\|G_{\mathcal{K}'}\|_{\mathcal{H}_2}^2
				=\|G_\mathcal{K}\|_{\mathcal{H}_2}^2 + \left\| \mathbf{M}_{12}^\mathcal{K} \Delta\mathcal{K} (I - \mathbf{M}_{22}^\mathcal{K} \Delta\mathcal{K})^{-1} \mathbf{M}_{21}^\mathcal{K} \right\|_{\mathcal{H}_2}^2 
				\geq\|G_\mathcal{K}\|_{\mathcal{H}_2}^2. 
			\end{align*}
			This completes the proof of sufficiency. 
		Conversely, suppose that $\mathcal K$ is optimal over
			$\mathbb K_q$.
			Theorem~14.7 in
			\cite{1995Robust}
			implies that $\|G_{\mathcal K_{\mathrm{opt}}}\|_{\mathcal{H}_2}=\min_{\mathcal{L}\in\mathbb K_{\mathrm{stab}}}\|G_{\mathcal{L}}\|_{\mathcal{H}_2}$.
			Since $\mathcal K_{\mathrm{opt}}\in\mathbb K_q$,
			optimality of $\mathcal K$ over $\mathbb K_q$ gives
			\begin{equation}
				\label{eq:optimality-all-orders}
				\|G_{\mathcal K}\|_{\mathcal H_2}^2
				=
				\|G_{\mathcal K_{\mathrm{opt}}}\|_{\mathcal H_2}^2
				\le
				\|G_{\mathcal L}\|_{\mathcal H_2}^2,
				\qquad
				\mathcal L\in\mathbb K_{\mathrm{stab}}.
			\end{equation}
			Thus, $\mathcal K$ also attains the minimum over
			$\mathbb K_{\mathrm{stab}}$.
			Fix any $\mathcal K'\in\mathbb K_q$ and set
			$\Delta\mathcal K=\mathcal K'-\mathcal K$.
			Define the corresponding Youla parameter by
			\begin{equation}
				\label{eq:relative-youla-parameter}
				\mathbf Q_{\mathcal K'}(s)
				:=
				\Delta\mathcal K
				\bigl(I-\mathbf M_{22}^{\mathcal K}(s)
				\Delta\mathcal K\bigr)^{-1}.
			\end{equation}
			Since $\mathcal A_{\mathcal K}$ is Hurwitz,
			$\mathbf M_{22}^{\mathcal K}$ is stable and strictly proper.
			Moreover, it holds that
			\begin{equation}
				\label{eq:relative-youla-realization}
				\begin{aligned}
					\mathbf Q_{\mathcal K'}(s)
					={}&\Delta\mathcal K+\Delta\mathcal K
					\begin{bmatrix}
						C&0\\
						0&I_q
					\end{bmatrix}
					(sI-\mathcal A_{\mathcal K'})^{-1}
					\begin{bmatrix}
						B&0\\
						0&I_q
					\end{bmatrix}
					\Delta\mathcal K.
				\end{aligned}
			\end{equation}
			Because $\mathcal A_{\mathcal K'}$ is Hurwitz,
			\eqref{eq:relative-youla-realization} shows that
			$\mathbf Q_{\mathcal K'}\in\mathcal{RH}_\infty$
			and
			$\mathbf Q_{\mathcal K'}(\infty)=\Delta\mathcal K$.
		We therefore have
			\begin{equation}
				\label{eq:relative-youla-increment}
				\Delta_{\mathcal K'}G
				=
				\mathbf M_{12}^{\mathcal K}
				\mathbf Q_{\mathcal K'}
				\mathbf M_{21}^{\mathcal K}.
			\end{equation}
			We next construct controllers
			$\mathcal K_t\in\mathbb K_{\mathrm{stab}}$, whose closed-loop
			transfer matrices satisfy
			$G_t=G_{\mathcal K}+t\Delta_{\mathcal K'}G$
			for all sufficiently small real $t$.
			Here, $t$ is a perturbation parameter.
		Choose $\varepsilon>0$ such that
		\begin{equation}
			\label{eq:youla-smallness}
			\varepsilon\max\left\{
			\|\mathbf Q_{\mathcal K'}\mathbf M_{22}^{\mathcal K}
			\|_{\mathcal H_\infty},
			\|\mathbf M_{22}^{\mathcal K}\mathbf Q_{\mathcal K'}
			\|_{\mathcal H_\infty}
			\right\}<1.
		\end{equation}
		For every $|t|<\varepsilon$, the Neumann-series theorem
		\cite[Theorem~10.7]{rudin1991functional}
		implies that
		$(I+t\mathbf Q_{\mathcal K'}\mathbf M_{22}^{\mathcal K})^{-1}$
		and
		$(I+t\mathbf M_{22}^{\mathcal K}\mathbf Q_{\mathcal K'})^{-1}$
		exist and are bounded and analytic in the open right
		half-plane of $\mathbb{C}$. These inverses are real-rational and tend to
		the identity as $s\to\infty$, so both belong to
		$\mathcal{RH}_\infty$.
		We therefore define
		\begin{equation}
			\label{Youla}
			\begin{aligned}
				\Delta\mathbf K_t
				&=
				(I+t\mathbf Q_{\mathcal K'}\mathbf M_{22}^{\mathcal K})^{-1}
				t\mathbf Q_{\mathcal K'}=
				t\mathbf Q_{\mathcal K'}
				(I+t\mathbf M_{22}^{\mathcal K}\mathbf Q_{\mathcal K'})^{-1}.
			\end{aligned}
		\end{equation}
		In particular, $\Delta\mathbf K_t\in\mathcal{RH}_\infty$.
		A direct calculation gives
		\begin{equation}
			\label{eq:youla-feedback-identities}
			\begin{aligned}
				&(I-\mathbf M_{22}^{\mathcal K}\Delta\mathbf K_t)^{-1}
				=I+t\mathbf M_{22}^{\mathcal K}\mathbf Q_{\mathcal K'},\quad(I-\Delta\mathbf K_t\mathbf M_{22}^{\mathcal K})^{-1}
				=I+t\mathbf Q_{\mathcal K'}\mathbf M_{22}^{\mathcal K},\\
				&\Delta\mathbf K_t
				(I-\mathbf M_{22}^{\mathcal K}\Delta\mathbf K_t)^{-1}
				=t\mathbf Q_{\mathcal K'}.
			\end{aligned}
		\end{equation}
		Here, $I+t\mathbf M_{22}^{\mathcal K}\mathbf Q_{\mathcal K'}$ and $I+t\mathbf Q_{\mathcal K'}\mathbf M_{22}^{\mathcal K}$ are stable and proper.
		Together with the stability of $\mathbf M_{22}^{\mathcal K}$ and
		$\Delta\mathbf K_t$, these identities ensure internal
		stability of their feedback interconnection when stable
		state-space realizations are used.
		The interconnection is also well posed because
		$\mathbf M_{22}^{\mathcal K}$ is strictly proper.
		Moreover,
		$\mathbf M_{22}^{\mathcal K}(\infty)=0$ and
		$\mathbf Q_{\mathcal K'}(\infty)=\Delta\mathcal K$
		imply
		\begin{equation}
			\label{eq:youla-feedthrough}
			\Delta\mathbf K_t(\infty)=t\Delta\mathcal K.
		\end{equation}
		Thus, the upper-left block of $	\Delta\mathbf K_t(\infty)$ remains zero.
		Choose a stable state-space realization
		\begin{equation}
			\label{eq:youla-perturbation-realization}
			\Delta\mathbf K_t(s)
			=
			\begin{bmatrix}
				0&D_{t,12}\\
				D_{t,21}&D_{t,22}
			\end{bmatrix}
			+
			\begin{bmatrix}
				C_{t,1}\\
				C_{t,2}
			\end{bmatrix}
			(sI-A_t)^{-1}
			\begin{bmatrix}
				B_{t,1}&B_{t,2}
			\end{bmatrix},
		\end{equation}
		where $A_t$ is Hurwitz.
		The corresponding augmented 
		controller is
		\begin{equation}
			\label{eq:controller-reconstruction}
			\mathcal K_t
			=
			\begin{bmatrix}
				0&C_{\mathcal K}+D_{t,12}&C_{t,1}\\
				B_{\mathcal K}+D_{t,21}
				&A_{\mathcal K}+D_{t,22}&C_{t,2}\\
				B_{t,1}&B_{t,2}&A_t
			\end{bmatrix}.
		\end{equation}
		By \eqref{eq:youla-feedback-identities}, this controller
		is internally stabilizing, although its order may exceed $q$.
		Substituting the last identity in
		\eqref{eq:youla-feedback-identities} into the LFT formula
		yields
		\begin{equation}
			\label{eq:youla-affine-path}
			\begin{aligned}
				G_t
				&=
				G_{\mathcal K}
				+\mathbf M_{12}^{\mathcal K}\Delta\mathbf K_t
				(I-\mathbf M_{22}^{\mathcal K}\Delta\mathbf K_t)^{-1}
				\mathbf M_{21}^{\mathcal K}\\
				&=
				G_{\mathcal K}
				+t\mathbf M_{12}^{\mathcal K}\mathbf Q_{\mathcal K'}
				\mathbf M_{21}^{\mathcal K}\\
				&=
				G_{\mathcal K}+t\Delta_{\mathcal K'}G.
			\end{aligned}
		\end{equation}
		Since
		$\Delta_{\mathcal K'}G=G_{\mathcal K'}-G_{\mathcal K}$
		and both $G_{\mathcal K}$ and $G_{\mathcal K'}$ belong to
		$\mathcal{RH}_{\infty,0}$, the same holds for $G_t$.
		In particular, $\|G_t\|_{\mathcal H_2}<\infty$.		
		As established above, $\mathcal K$ is globally optimal over
		all internally stabilizing controllers, including the
		possibly higher-order controllers $\mathcal K_t$.
		Hence, the real-valued quadratic function
		\begin{equation}
			\label{eq:youla-cost-path}
			\varphi(t)
			:=
			\|G_{\mathcal K}+t\Delta_{\mathcal K'}G\|_{\mathcal H_2}^{2}
		\end{equation}
		has a minimum at $t=0$ on $(-\varepsilon,\varepsilon)$.
		Therefore, we have
		\begin{equation}
			\label{eq:youla-path-stationarity}
			\begin{aligned}
				0=\varphi'(0)
				=
				2\langle G_{\mathcal K},\Delta_{\mathcal K'}G
				\rangle_{\mathcal H_2}.
			\end{aligned}
		\end{equation}
		Since $\mathcal K'\in\mathbb K_q$ is arbitrary, the desired
		orthogonality condition follows.
	\end{proof}
	
	
	This stationarity condition admits a clear frequency-domain interpretation, which is established in Theorem~\ref{main_res} below. Though our derivation is closely related to the Youla parameterization \cite{1104023}, a key advantage of our LFT-based approach is that it explicitly retains the controller structure from~\eqref{eq_parameterization}. This structural preservation is significant because it builds a direct bridge between classic control theory and modern policy optimization methods \cite{9993305}.
	
	To leverage this frequency-domain perspective and rigorously analyze the stationarity condition \eqref{eq:stationarity_condition}, we express the $\mathcal{H}_2$ inner product using residue calculus. The following lemma provides an explicit formula required for our subsequent derivations.
	
	\begin{lemma}\cite{1995Robust}
		\label{res}
		Let strictly proper real-coefficient rational vector functions \( f(s) \) and \( g(s) \) be square-integrable with no poles on the purely imaginary axis. Then, the $\mathcal{H}_2$ inner product $\langle f,g\rangle_{\mathcal{H}_2}$ equals $\ResLHP{f(-s)^\top g(s)}$, the sum of residues of $f(-s)^\top g(s)$ at poles in the left half-plane.
	\end{lemma}
	
	Building on Lemma~\ref{res}, we note that the $\mathcal{H}_2$ inner product between two strictly proper rational functions is zero, if one is stable (i.e., all the poles in the open left half-plane), while the other is anti-stable (i.e., all the poles in the open right half-plane); see \cite{1995Robust} for more details. This observation enables us to restate the stationarity condition~\eqref{eq:stationarity_condition} as an equivalent frequency-domain criterion, which we formalize in the following theorem.

	\begin{theorem}\label{main_res}
		Controller $\mathcal{K} \in \mathbb{K}_q$ is an optimal solution to problem~\eqref{op_lqg}, if and only if
		\begin{equation}\label{value}
			(\mathcal{C}_1-\mathcal{B}_1\mathcal{P}_\mathcal{K})(sI-\mathcal{A})^{-1}(\mathcal{B}_0-\Sigma_\mathcal{K}\mathcal{C}_0)\equiv0,\quad \forall s\in\mathbb{C}
		\end{equation}
		holds, where
		\begin{equation}\label{matrices}
			\mathcal{B}_0=\begin{bmatrix}
				0 & 0\\
				B_\mathcal{K}\mathcal{V} & 0
			\end{bmatrix},\ 
			\mathcal{C}_0
			=-\begin{bmatrix}
				C^\top& 0\\
				0 & I
			\end{bmatrix},\ 
			\mathcal{B}_1=-\begin{bmatrix}
				B^\top& 0\\
				0 & I 
			\end{bmatrix},\ 
			\mathcal{C}_1
			=\begin{bmatrix}
				0& RC_\mathcal{K}\\
				0 & 0
			\end{bmatrix}.
		\end{equation}
		Moreover, \eqref{value} indicates that $\mathcal{K}$ is a stationary point of problem \textcolor{black}{\eqref{op_lqg}}.
	\end{theorem}
	
	\begin{proof}
		\textcolor{black}{We first prove that \eqref{value} implies \eqref{eq:stationarity_condition} } Using Lemma~\ref{res}, the $\mathcal{H}_2$ inner product $\langle G_\mathcal{K}, \Delta_{\mathcal{K}^\prime} G \rangle_{\mathcal{H}_2}$ can be rewritten as
		\begin{equation*}
			\begin{aligned}
				&\left\langle G_\mathcal{K}, \Delta_{\mathcal{K}^\prime} G \right\rangle_{\mathcal{H}_2} \\
				&= \operatorname{tr}\left( \int_{-\infty}^{\infty} ((\mathbf{M}_{12}^\mathcal{K}(j\omega))^* \mathbf{M}_{11}^\mathcal{K}(j\omega)(\mathbf{M}_{21}^\mathcal{K}(j\omega))^*) (\Delta\mathcal{K} (I - \mathbf{M}_{22}^\mathcal{K}(j\omega) \Delta\mathcal{K})^{-1}) ^*\frac{d\omega}{2\pi}\right)\\
				&=\,\operatorname{tr}\left(\ResLHP{(\mathbf{M}_{12}^\mathcal{K}(-s))^\top \mathbf{M}_{11}^\mathcal{K}(s)(\mathbf{M}_{21}^\mathcal{K}(-s))^\top(\Delta\mathcal{K}(I-\mathbf{M}_{22}^\mathcal{K}(-s)\Delta\mathcal{K})^{-1})^\top}\right).
			\end{aligned}
		\end{equation*}
		Applying \eqref{m_k} and the definition of $\Delta\mathcal{K}$ show that  $\Delta\mathcal{K}(I-\mathbf{M}_{22}^\mathcal{K}(s)\Delta\mathcal{K})^{-1}$ is stable . Thus, if $(\mathbf{M}_{12}^\mathcal{K}(-s))^\top \mathbf{M}_{11}^\mathcal{K}(s)(\mathbf{M}_{21}^\mathcal{K}(-s))^\top$ is anti-stable,  it holds that $\left\langle G_{\mathcal{K}}, \Delta_{\mathcal{K}^\prime} G\right\rangle_{\mathcal{H}_2}=0$. It remains to show that this property is equivalent to condition~\eqref{value}. To this end, we further examine the explicit structure of $(\mathbf{M}_{12}^\mathcal{K}(-s))^\top \mathbf{M}_{11}^\mathcal{K}(s)(\mathbf{M}_{21}^\mathcal{K}(-s))^\top$.  According to the concatenated representation of transfer function, we have
		\begin{equation*}
			\begin{aligned}\label{h2_lqg}
				\mathbf{M}_{11}^\mathcal{K}(s)(\mathbf{M}_{21}^\mathcal{K}(-s))^\top=& \mathcal{C}(sI - \mathcal{A})^{-1}\mathcal{B}\left(\begin{bmatrix} C & {0} \\ {0} & I \end{bmatrix} (-sI - \mathcal{A})^{-1} \mathcal{B} + \begin{bmatrix} {0} & \mathcal{V}^{\frac{1}{2}} \\ {0} & {0} \end{bmatrix}\right)^\top\\
				=&\left[\begin{array}{cc|c}
					\mathcal{A} & \mathcal{B}\mathcal{B}^\top & \mathcal{B}_0\\
					0& -\mathcal{A}^\top & \mathcal{C}_0\\
					\hline
					\mathcal{C}& 0& 0
				\end{array}\right].	\end{aligned}
		\end{equation*}
		In the same way, we obtain
		\begin{equation}
			\begin{aligned}\label{eq_real}
				&(\mathbf{M}_{12}^\mathcal{K}(-s))^\top \mathbf{M}_{11}^\mathcal{K}(s)(\mathbf{M}_{21}^\mathcal{K}(-s))^\top=\left[\begin{array}{ccc|c}
					-\mathcal{A}^\top & \mathcal{C}^\top\mathcal{C} & 0 & 0\\
					0&\mathcal{A} & \mathcal{B}\mathcal{B}^\top & \mathcal{B}_0\\
					0&0& -\mathcal{A}^\top & \mathcal{C}_0\\
					\hline
					\mathcal{B}_1& \mathcal{C}_1& 0&0
				\end{array}\right].	\end{aligned}
		\end{equation}
		Letting 
				$$
		T=\begin{bmatrix}
			I& 0&0\\
			0&I & -\Sigma_{\mathcal{K}}\\
			0&0& I
		\end{bmatrix}
		$$
		and applying the similar transformation
		\begin{equation}
			T\begin{bmatrix}
				-\mathcal{A}^\top & \mathcal{C}^\top\mathcal{C} & 0 \\
				0&\mathcal{A} & \mathcal{B}\mathcal{B}^\top \\
				0&0& -\mathcal{A}^\top 
			\end{bmatrix}T^{-1},\quad T\begin{bmatrix}
				0\\
				\mathcal{B}_0\\
				\mathcal{C}_0
			\end{bmatrix},\quad \begin{bmatrix}
				\mathcal{B}_1& \mathcal{C}_1& 0
			\end{bmatrix}T^{-1}
		\end{equation}
		to \eqref{eq_real}, we obtain
		\begin{equation}\label{eq_lyap_r}
			\begin{aligned}
				&(\mathbf{M}_{12}^\mathcal{K}(-s))^\top \mathbf{M}_{11}^\mathcal{K}(s)(\mathbf{M}_{21}^\mathcal{K}(-s))^\top\\
				&=\left[\begin{array}{ccc|c}
					-\mathcal{A}^\top & \mathcal{C}^\top\mathcal{C} & \mathcal{C}^\top\mathcal{C}\Sigma_\mathcal{K} & 0\\
					0&\mathcal{A} & \Sigma_\mathcal{K}\mathcal{A}^\top+\mathcal{A}\Sigma_\mathcal{K}+\mathcal{B}\mathcal{B}^\top & \mathcal{B}_0-\Sigma_\mathcal{K}\mathcal{C}_0\\
					0&0& -\mathcal{A}^\top & \mathcal{C}_0\\
					\hline
					\mathcal{B}_1& \mathcal{C}_1& \mathcal{C}_1\Sigma_\mathcal{K}&0
				\end{array}\right].
			\end{aligned}
		\end{equation}
		From Lyapunov equation \eqref{eq_lyap2}, it follows that $\Sigma_\mathcal{K}\mathcal{A}^\top+\mathcal{A}\Sigma_\mathcal{K}+\mathcal{B}\mathcal{B}^\top=0$. 
		Letting
				\[
		T'=\begin{bmatrix}
			I & \mathcal{P}_\mathcal{K} & 0 \\
			0 & I & 0 \\
			0 & 0 & I
		\end{bmatrix}
		\]
		and applying the similar similarity transformation
		\begin{equation}
			\begin{gathered}
				T'\begin{bmatrix}
					-\mathcal{A}^\top & \mathcal{C}^\top\mathcal{C} & \mathcal{C}^\top\mathcal{C}\Sigma_\mathcal{K} \\
					0&\mathcal{A} & \Sigma_\mathcal{K}\mathcal{A}^\top+\mathcal{A}\Sigma_\mathcal{K}+\mathcal{B}\mathcal{B}^\top \\
					0&0& -\mathcal{A}^\top 
				\end{bmatrix}{T'}^{-1},\\
				T'\begin{bmatrix}
					0\\
					\mathcal{B}_0-\Sigma_\mathcal{K}\mathcal{C}_0\\
					\mathcal{C}_0
				\end{bmatrix},\quad \begin{bmatrix}
					\mathcal{B}_1& \mathcal{C}_1& \mathcal{C}_1\Sigma_\mathcal{K}
				\end{bmatrix}{T'}^{-1}
			\end{gathered}
		\end{equation}
	to \eqref{eq_lyap_r}, we obtain
		\begin{equation*}
			\begin{aligned}
				&(\mathbf{M}_{12}^\mathcal{K}(-s))^\top \mathbf{M}_{11}^\mathcal{K}(s)(\mathbf{M}_{21}^\mathcal{K}(-s))^\top\\ 
				&= \left[\begin{array}{ccc|c}
					-\mathcal{A}^\top & \mathcal{A}^\top \mathcal{P}_\mathcal{K}+\mathcal{P}_\mathcal{K}\mathcal{A}+\mathcal{C}^\top\mathcal{C} & \mathcal{C}^\top\mathcal{C}\Sigma_\mathcal{K} & \mathcal{P}_\mathcal{K}(\mathcal{B}_0-\Sigma_\mathcal{K}\mathcal{C}_0)\\
					0 & \mathcal{A} & 0 & \mathcal{B}_0-\Sigma_\mathcal{K}\mathcal{C}_0\\
					0 & 0 & -\mathcal{A}^\top & \mathcal{C}_0\\
					\hline
					\mathcal{B}_1 & \mathcal{C}_1-\mathcal{B}_1\mathcal{P}_\mathcal{K} & \mathcal{C}_1\Sigma_\mathcal{K} & 0
				\end{array}\right].
			\end{aligned}
		\end{equation*}
		Similarly, Lyapunov equation \eqref{eq_lyap1} gives $\mathcal{A}^\top \mathcal{P}_\mathcal{K}+\mathcal{P}_\mathcal{K}\mathcal{A}+\mathcal{C}^\top\mathcal{C}=0$. Consequently, based on the above analysis, the expression $(\mathbf{M}_{12}^\mathcal{K}(-s))^\top \mathbf{M}_{11}^\mathcal{K}(s)(\mathbf{M}_{21}^\mathcal{K}(-s))^\top$ can be further simplified as
		\begin{equation}
			\begin{aligned}
				&(\mathbf{M}_{12}^\mathcal{K}(-s))^\top \mathbf{M}_{11}^\mathcal{K}(s)(\mathbf{M}_{21}^\mathcal{K}(-s))^\top\\
				&{=}\mathcal{B}_1(sI+\mathcal{A}^\top)^{-1}\mathcal{P}_\mathcal{K}(\mathcal{B}_0-\Sigma_\mathcal{K}\mathcal{C}_0)+(\mathcal{C}_1-\mathcal{B}_1\mathcal{P}_\mathcal{K})(sI-\mathcal{A})^{-1}(\mathcal{B}_0-\Sigma_\mathcal{K}\mathcal{C}_0)\\
				&\hphantom{=}+\mathcal{C}_1\Sigma_\mathcal{K}(sI+\mathcal{A}^\top)^{-1}\mathcal{C}_0+\mathcal{B}_1(sI+\mathcal{A}^\top)^{-1}\mathcal{C}^\top\mathcal{C}\Sigma_\mathcal{K}(sI+\mathcal{A}^\top)^{-1}\mathcal{C}_0.
			\end{aligned}\label{eq_4}
		\end{equation}
		In RHS of \eqref{eq_4}, only the second term $(\mathcal{C}_1-\mathcal{B}_1\mathcal{P}_\mathcal{K})(sI-\mathcal{A})^{-1}(\mathcal{B}_0-\Sigma_\mathcal{K}\mathcal{C}_0)$ is stable, while the first, third, and fourth terms are anti-stable. Consequently, in light of Lemma~\ref{res}, we conclude that
		\begin{align*}
			&\left\langle G_\mathcal{K}, \Delta_{\mathcal{K}^\prime} G \right\rangle_{\mathcal{H}_2}\\
			&= \operatorname{tr}\left( \int_{-\infty}^{\infty} ((\mathbf{M}_{12}^\mathcal{K}(j\omega))^* \mathbf{M}_{11}^\mathcal{K}(j\omega)(\mathbf{M}_{21}^\mathcal{K}(j\omega))^*)( \Delta\mathcal{K} (I - \mathbf{M}_{22}^\mathcal{K}(j\omega) \Delta\mathcal{K})^{-1})^* \frac{d\omega}{2\pi}\right)\\
			&=\operatorname{tr}\big(\ResLHP{\mathcal{B}_1(sI+\mathcal{A}^\top)^{-1}\mathcal{P}_\mathcal{K}(\mathcal{B}_0-\Sigma_\mathcal{K}\mathcal{C}_0)(\Delta\mathcal{K} (I - \mathbf{M}_{22}^\mathcal{K}(-s) \Delta\mathcal{K})^{-1})^\top}\\
			&\hphantom{=}+\ResLHP{(\mathcal{C}_1-\mathcal{B}_1\mathcal{P}_\mathcal{K})(sI-\mathcal{A})^{-1}(\mathcal{B}_0-\Sigma_\mathcal{K}\mathcal{C}_0)(\Delta\mathcal{K} (I - \mathbf{M}_{22}^\mathcal{K}(-s) \Delta\mathcal{K})^{-1})^\top}\\
			&\hphantom{=}+\ResLHP{\mathcal{C}_1\Sigma_\mathcal{K}(sI+\mathcal{A}^\top)^{-1}\mathcal{C}_0(\Delta\mathcal{K} (I - \mathbf{M}_{22}^\mathcal{K}(-s) \Delta\mathcal{K})^{-1})^\top}\\
			&\hphantom{=}+\ResLHP{\mathcal{B}_1(sI+\mathcal{A}^\top)^{-1}\mathcal{C}^\top\mathcal{C}\Sigma_\mathcal{K}(sI+\mathcal{A}^\top)^{-1}\mathcal{C}_0(\Delta\mathcal{K} (I - \mathbf{M}_{22}^\mathcal{K}(-s) \Delta\mathcal{K})^{-1})^\top}\\
			&=\operatorname{tr}\left(\ResLHP{(\mathcal{C}_1-\mathcal{B}_1\mathcal{P}_\mathcal{K})(sI-\mathcal{A})^{-1}(\mathcal{B}_0-\Sigma_\mathcal{K}\mathcal{C}_0)(\Delta\mathcal{K} (I - \mathbf{M}_{22}^\mathcal{K}(-s) \Delta\mathcal{K})^{-1})^\top}\right).
		\end{align*}
		Therefore, a sufficient condition for $(\mathbf{M}_{12}^\mathcal{K}(-s))^\top \mathbf{M}_{11}^\mathcal{K}(s)(\mathbf{M}_{21}^\mathcal{K}(-s))^\top$ to be anti-stable is that the second term in \eqref{eq_4} vanishes.  The preceding analysis has established the sufficiency of this condition.
		
		\textcolor{black}{		We now use a proof by contradiction to show that \eqref{eq:stationarity_condition} implies \eqref{value}. }
		To facilitate subsequent derivations, we recall the matrices introduced earlier in \eqref{matrices}, which implies
		\begin{align*}
			\mathcal{C}_1 - \mathcal{B}_1 \mathcal{P}_\mathcal{K} &= 
			\begin{bmatrix} 
				B^\top P_1^\mathcal{K} + R \bar{C}_\mathcal{K} \\ 
				P_2^\mathcal{K} 
			\end{bmatrix}, \\
			\mathcal{B}_0 - \Sigma_\mathcal{K} \mathcal{C}_0 &= 
			\begin{bmatrix} 
				(C \Sigma_1^\mathcal{K})^\top + \bar{B}_\mathcal{K} \mathcal{V} & \Sigma_2^{\mathcal{K}\top} 
			\end{bmatrix};
		\end{align*}
		here, the involved matrices are defined by the following partitions
		\begin{equation}\label{P2}
			\begin{aligned}
				P_1^\mathcal{K} &= \begin{bmatrix} P_{11}^\mathcal{K} & P_{12}^\mathcal{K} \end{bmatrix}, & 
				P_2^\mathcal{K} &= \begin{bmatrix} P_{21}^\mathcal{K} & P_{22}^\mathcal{K} \end{bmatrix}, \\
				\Sigma_1^\mathcal{K} &= \begin{bmatrix} \Sigma_{11}^\mathcal{K} & \Sigma_{12}^\mathcal{K} \end{bmatrix}, & 
				\Sigma_2^\mathcal{K} &= \begin{bmatrix} \Sigma_{21}^\mathcal{K} & \Sigma_{22}^\mathcal{K} \end{bmatrix}, \\
				\bar{C}_\mathcal{K} &= \begin{bmatrix} 0_{m_1 \times n} & C_\mathcal{K} \end{bmatrix}, &
				\bar{B}_\mathcal{K} &= \begin{bmatrix} 0_{n \times m_2} \\ B_\mathcal{K} \end{bmatrix}.
			\end{aligned}
		\end{equation}
		\textcolor{black}{Then, $	\hat{\mathcal{F}}(s):=(\mathcal{C}_1-\mathcal{B}_1\mathcal{P}_\mathcal{K})(sI-\mathcal{A})^{-1}(\mathcal{B}_0-\Sigma_\mathcal{K}\mathcal{C}_0)$ admits the following block expansion:}
		\begin{equation*}
			\begin{aligned}\label{eq_value_de}
				&\begin{bmatrix}
					B^\top P_1^\mathcal{K}+R\bar{C}_\mathcal{K}\\
					P_2^\mathcal{K}
				\end{bmatrix}
				(sI-\mathcal{A})^{-1}
				\begin{bmatrix}
					(C\Sigma_1^\mathcal{K})^\top+\bar{B}_\mathcal{K}\mathcal{V} & \Sigma_2^{\mathcal{K}\top}
				\end{bmatrix} \\
				&=\!
				\begin{bmatrix}
					(B^\top P_1^\mathcal{K}+R\bar{C}_\mathcal{K})(sI-\mathcal{A})^{-1}((C\Sigma_1^\mathcal{K})^\top+\bar{B}_\mathcal{K}\mathcal{V})\! & (B^\top P_1^\mathcal{K}+R\bar{C}_\mathcal{K})(sI-\mathcal{A})^{-1}\Sigma_2^{\mathcal{K}\top}\\
					P_2^\mathcal{K}(sI-\mathcal{A})^{-1}\!((C\Sigma_1^\mathcal{K})^\top+\bar{B}_\mathcal{K}\mathcal{V})\! & P_2^\mathcal{K}(sI-\mathcal{A})^{-1}\Sigma_2^{\mathcal{K}\top}
				\end{bmatrix}\!.
			\end{aligned}
		\end{equation*}
		Let $\mathcal{K}\in\mathbb{K}_q$ be an optimal solution to problem \eqref{op_lqg} and assume that condition~\eqref{value} does not hold, namely, the corresponding stable transfer function $\hat{\mathcal{F}}(s)$ is not identically zero. By the bijectivity of Youla parameterization over all stabilizing controllers, for every sufficiently small $\delta>0$ there exists an augmenting stabilizing controller $\Delta\mathcal{K}(s)$ such that
		\begin{equation*}
			\Delta\mathcal{K}(s)\bigl(I-\mathbf{M}_{22}^{\mathcal{K}}(s)\Delta\mathcal{K}(s)\bigr)^{-1}
			=-\delta\hat{\mathcal{F}}(s).
		\end{equation*}
		Let $\mathcal{K}'$ denote the composite controller of $\mathcal{K}$ and $\Delta{\mathcal{K}}$ generated in the same way of \eqref{eq:controller-reconstruction}. Using~\eqref{eq_4} and Lemma~\ref{res}, we have $			\left\langle G_\mathcal{K},\Delta_{\mathcal{K}'}G\right\rangle_{\mathcal{H}_2}
		=-\delta\|\hat{\mathcal{F}}\|_{\mathcal{H}_2}^{2}.$
		Since $\hat{\mathcal{F}}$ is nonzero, $\|\hat{\mathcal{F}}\|_{\mathcal{H}_2}^{2}>0$. We therefore have
		\begin{align*}
			\|G_{\mathcal{K}'}\|_{\mathcal{H}_2}^{2}
			={}&\|G_\mathcal{K}\|_{\mathcal{H}_2}^{2}
			-2\delta\|\hat{\mathcal{F}}\|_{\mathcal{H}_2}^{2}+\delta^{2}\left\|\mathbf{M}_{12}^{\mathcal{K}}\hat{\mathcal{F}}\mathbf{M}_{21}^{\mathcal{K}}\right\|_{\mathcal{H}_2}^{2}
			<\|G_\mathcal{K}\|_{\mathcal{H}_2}^{2}
		\end{align*}
		for every sufficiently small $\delta>0$. Although $\mathcal{K}'$ may have higher order and need not belong to $\mathbb{K}_q$,   $\mathcal{K}_{\mathrm{opt}}$ belongs to $\mathbb{K}_q$ and is optimal over all stabilizing controllers. Hence,
		\[
			\|G_{\mathcal{K}_{\mathrm{opt}}}\|_{\mathcal{H}_2}^{2}
			\leq\|G_{\mathcal{K}'}\|_{\mathcal{H}_2}^{2}
			<\|G_\mathcal{K}\|_{\mathcal{H}_2}^{2},
		\]
		which contradicts the assumed optimality of $\mathcal{K}$ within $\mathbb{K}_q$. Therefore, condition~\eqref{value} must hold.
		Finally, we show that \eqref{value} implies $\nabla J(\mathcal{K})=0$. For $|s|$ sufficiently large, \eqref{value} can be expanded into the following Laurent series:
			\begin{equation}
				\sum_{i=1}^{\infty}\frac{(\mathcal{C}_1-\mathcal{B}_1\mathcal{P}_\mathcal{K})\mathcal{A}^{i-1}(\mathcal{B}_0-\Sigma_\mathcal{K}\mathcal{C}_0)}{s^i}\equiv 0.
			\end{equation}
			According to the identity theorem for power series, all coefficients of the series must be zero, which inherently yields
			\begin{equation}
				(\mathcal{C}_1-\mathcal{B}_1\mathcal{P}_\mathcal{K})\mathcal{A}^{i}(\mathcal{B}_0-\Sigma_\mathcal{K}\mathcal{C}_0) =0, \quad \forall i=0, 1, \dots
			\end{equation}
			The proof is then completed by examining the case for $i=0$:
			\begin{align*}
				&(\mathcal{C}_1-\mathcal{B}_1\mathcal{P}_\mathcal{K})(\mathcal{B}_0-\Sigma_\mathcal{K}\mathcal{C}_0)\\
				=&\begin{bmatrix}
					(B^\top P_1^\mathcal{K}+R\bar{C}_\mathcal{K})((C\Sigma_1^\mathcal{K})^\top+\bar{B}_\mathcal{K}\mathcal{V})\! & (B^\top P_1^\mathcal{K}+R\bar{C}_\mathcal{K})\Sigma_2^{\mathcal{K}\top}\\
					P_2^\mathcal{K}\!((C\Sigma_1^\mathcal{K})^\top+\bar{B}_\mathcal{K}\mathcal{V})\! & P_2^\mathcal{K}\Sigma_2^{\mathcal{K}\top}
				\end{bmatrix}\\
				=&\begin{bmatrix}
					(B^\top P_1^\mathcal{K}+R\bar{C}_\mathcal{K})((C\Sigma_1^\mathcal{K})^\top+\bar{B}_\mathcal{K}\mathcal{V})\! & \frac{1}{2}\nabla_{C_\mathcal{K}} J(\mathcal{K})\\
					\frac{1}{2}\nabla_{B_\mathcal{K}} J(\mathcal{K})\! & \frac{1}{2}\nabla_{A_\mathcal{K}} J(\mathcal{K})
				\end{bmatrix}=0.
		\end{align*}
	That completes the proof.
	\end{proof}
	
	\begin{remark} 
		We now
		present an equivalent algebraic formulation in the state-space domain to gain further insight into the optimality condition \eqref{value}.  
		This reformulation transforms the frequency-domain condition~\eqref{value} into a set of finite-dimensional equations, thereby providing a complementary perspective and facilitating the analysis of structural properties of optimal solutions of problem \eqref{op_lqg}.
		Standard results in state-space theory imply that the optimality condition~\eqref{value} holds, if and only if the following system is anti-stable
		\begin{equation}\label{sys_value}
			\left[\begin{array}{c|c}
				\mathcal{A}& \mathcal{B}_0-\Sigma_\mathcal{K}\mathcal{C}_0\\
				\hline
				\mathcal{C}_1-\mathcal{B}_1\mathcal{P}_\mathcal{K}& 0
			\end{array}\right].
		\end{equation}
		Moreover, under the assumption that $\mathcal{A}$ is stable, standard definitions of controllability and observability imply that $	(\mathcal{C}_1-\mathcal{B}_1\mathcal{P}_\mathcal{K})(sI-\mathcal{A})^{-1}(\mathcal{B}_0-\Sigma_\mathcal{K}\mathcal{C}_0)\equiv0,~ \forall s\in\mathbb{C}$ if and only if its Kalman decomposition contains no subsystem that is both controllable and observable. It is equivalent to
		\begin{equation}\label{eq_equiv_opt}
			(\mathcal{C}_1-\mathcal{B}_1\mathcal{P}_\mathcal{K})\mathcal{A}^{i}(\mathcal{B}_0-\Sigma_\mathcal{K}\mathcal{C}_0)=0, \quad \forall i=0,1,\dots, n+q-1.
		\end{equation}
		The aforementioned controllability and observability analysis yields the following corollary. Specifically, this corollary establishes a sufficient optimality condition that generalizes the optimality conditions of~\cite{11442293} to the case $q\ge n$ , shedding light on the connection between the frequency-domain perspective and the Extended Convex Lifting (ECL) technique \cite{zheng2023benignnonconvexlandscapesoptimal,zheng2024benignnonconvexlandscapesoptimal}.
	\end{remark}
	
	\begin{corollary}\label{lqg}
	Controller $\mathcal{K} \in \mathbb{K}_q$ is an optimal solution to problem~\eqref{op_lqg}, if it is a stationary point of problem \eqref{op_lqg} and satisfies the following rank condition
			\begin{equation}
				\max\{\operatorname{rank}\left(  P_{21}^\mathcal{K}  \right),  \operatorname{rank}\left(  \Sigma_{21}^\mathcal{K}   \right)\} \geq n; \label{rank}
		\end{equation}
	see \eqref{pwblock} for the definitions of $P_{21}^{\mathcal{K}}$ and $\Sigma_{21}^{\mathcal{K}}$.
	\end{corollary}
	\begin{proof}
		Since $\mathcal{K}\in\mathbb{K}_q$, the matrix $\mathcal A$ is
		Hurwitz and the solutions of the Lyapunov equations in \eqref{eq_lyap1} and \eqref{eq_lyap2} satisfy
		$\mathcal P_{\mathcal K}\succeq0$ and $\Sigma_{\mathcal K}\succeq0$.
	 Calculating the gradients of $J(\mathcal{K})$ with respect to $A_\mathcal{K}$, $B_\mathcal{K}$, and $C_\mathcal{K}$ gives
		\begin{equation}
			\label{eq:cor47_stationarity}
			\begin{aligned}
				0&=\tfrac12\nabla_{A_{\mathcal K}}J(\mathcal K)
				=P_2^{\mathcal K}(\Sigma_2^{\mathcal K})^\top,\\
				0&=\tfrac12\nabla_{B_{\mathcal K}}J(\mathcal K)
				=P_2^{\mathcal K}
				\bigl((C\Sigma_1^{\mathcal K})^\top
				+\bar B_{\mathcal K}\mathcal V\bigr),\\
				0&=\tfrac12\nabla_{C_{\mathcal K}}J(\mathcal K)
				=\bigl(B^\top P_1^{\mathcal K}+R\bar C_{\mathcal K}\bigr)
				(\Sigma_2^{\mathcal K})^\top.
			\end{aligned}
		\end{equation}
		 We first consider the case
		$\operatorname{rank}P_{21}^{\mathcal K}=n$.
		Set
		\[
		r=\operatorname{rank}\Sigma_{22}^{\mathcal K},\qquad
		\mathcal X=\mathbb R^n\times
		\operatorname{im}\Sigma_{22}^{\mathcal K},\qquad
		\mathcal N=\operatorname{im}
		\bigl((\Sigma_2^{\mathcal K})^\top\bigr).
		\]
		For $z\in\ker\Sigma_{22}^{\mathcal K}$, let
		$\zeta=(0^\top_n,z^\top)^\top$.
		Positive semidefiniteness of $\Sigma_{\mathcal{K}}$ implies
		$\Sigma_{\mathcal K}\zeta=0$, and hence
		$\Sigma_{12}^{\mathcal K}z=0$.
		Meanwhile, we multiply Lyapunov equation \eqref{eq_lyap2}
		by $\zeta^\top$ and  $\zeta$ in the following way
		\begin{align*}
			\zeta^\top\left(\mathcal{A} \Sigma_{\mathcal{K}} + \Sigma_{\mathcal{K}} \mathcal{A}^\top\right)\zeta &=
			-\zeta^\top\begin{bmatrix}
				\mathcal{W} & 0 \\
				0 & B_{\mathcal{K}} \mathcal{V} B_{\mathcal{K}}^\top
			\end{bmatrix}\zeta.
		\end{align*}
	Using $\Sigma_\mathcal{K}\zeta=0$, we have $0=z^\top B_{\mathcal K}\mathcal V B_{\mathcal K}^\top z.$
		Since $\mathcal V\succ0$, it follows that
		$B_{\mathcal K}^\top z=0$.
		Applying $B_{\mathcal K}^\top z=0$ and
		right-multiplying Lyapunov equation \eqref{eq_lyap2} by $\zeta$
		then yield $\Sigma_{\mathcal K}\mathcal A^\top\zeta=0$.
		As
		$\mathcal A^\top\zeta=(0^\top,(A_{\mathcal K}^\top z)^\top)^\top$,
		we obtain $\Sigma_{22}^{\mathcal K}A_{\mathcal K}^\top z=0$.
		Using $	(\ker B_{\mathcal K}^{\top})^\perp
		=\operatorname{im}B_{\mathcal K},
		(\ker\Sigma_{12}^{\mathcal K})^\perp
		=\operatorname{im}\Sigma_{21}^{\mathcal K}$, and $(\ker\Sigma_{22}^{\mathcal K})^\perp
		=\operatorname{im}\Sigma_{22}^{\mathcal K}$ shows that
		\begin{equation}
			\label{eq:cor47_support}
			\begin{aligned}
				\operatorname{im}B_{\mathcal K}
				&\subseteq\operatorname{im}\Sigma_{22}^{\mathcal K},\quad
				A_{\mathcal K}\operatorname{im}\Sigma_{22}^{\mathcal K}\subseteq\operatorname{im}\Sigma_{22}^{\mathcal K},\quad
				\operatorname{im}\Sigma_{21}^{\mathcal K}\subseteq\operatorname{im}\Sigma_{22}^{\mathcal K}.
			\end{aligned}
		\end{equation}
		Furthermore, noting that
		$\ker\bigl((\Sigma_2^{\mathcal K})^\top\bigr)
		=\ker\Sigma_{22}^{\mathcal K}$,
		we have $\dim\mathcal N=r$ and $\mathcal N\subseteq\mathcal X$.
		By \eqref{eq:cor47_support} and the block forms of
		$\mathcal A$ and $\mathcal B_0-\Sigma_{\mathcal K}\mathcal C_0$, we have
		\begin{equation}
			\label{eq:cor47_in_X}
			\mathcal A\mathcal X\subseteq\mathcal X,
			\qquad
			\operatorname{im}(\mathcal B_0-\Sigma_{\mathcal K}\mathcal C_0)
			\subseteq\mathcal X.
		\end{equation}
		The first identity in \eqref{eq:cor47_stationarity} gives
		$\mathcal N\subseteq\ker(P_2^{\mathcal K}|_{\mathcal X})$.
		Since $\mathbb R^n\times\{0\}\subseteq\mathcal X$ and
		$\operatorname{rank}P_{21}^{\mathcal K}=n$, we have
		\[
		n\le\operatorname{rank}(P_2^{\mathcal K}|_{\mathcal X})
		\le\dim\mathcal X-\dim\mathcal N=(n+r)-r=n.
		\]
		Consequently, the rank-nullity theorem (see Theorem 3.21 in \cite{axler1997linear}) gives
		\begin{equation}
			\label{eq:cor47_kernel}
			\ker(P_2^{\mathcal K}|_{\mathcal X})=\mathcal N.
		\end{equation}
	Meanwhile,
		the lower block row of Lyapunov equation \eqref{eq_lyap1}
		is
		\[
		P_2^{\mathcal K}\mathcal A
		=-A_{\mathcal K}^\top P_2^{\mathcal K}
		-C_{\mathcal K}^\top
		\bigl(B^\top P_1^{\mathcal K}+R\bar C_{\mathcal K}\bigr).
		\]
		Right-multiplying by $(\Sigma_2^{\mathcal K})^\top$ and applying
		\eqref{eq:cor47_stationarity} yield
$P_2^{\mathcal K}\mathcal A(\Sigma_2^{\mathcal K})^\top=0.$
		Together with \eqref{eq:cor47_in_X} and
		\eqref{eq:cor47_kernel}, it holds that
		$\mathcal A\mathcal N\subseteq\mathcal N$.
		Moreover, recalling
		\[
		\begin{aligned}
			\mathcal B_0-\Sigma_{\mathcal K}\mathcal C_0
			&=\left[
			(C\Sigma_1^{\mathcal K})^\top+\bar B_{\mathcal K}\mathcal V
			\quad (\Sigma_2^{\mathcal K})^\top
			\right],\quad
			\mathcal C_1-\mathcal B_1\mathcal P_{\mathcal K}
			&=\begin{bmatrix}
				B^\top P_1^{\mathcal K}+R\bar C_{\mathcal K}\\
				P_2^{\mathcal K}
			\end{bmatrix},
		\end{aligned}
		\]
		\eqref{eq:cor47_stationarity}, \eqref{eq:cor47_in_X}, and
		\eqref{eq:cor47_kernel} imply
		\[
		\operatorname{im}(\mathcal B_0-\Sigma_{\mathcal K}\mathcal C_0)
		\subseteq\mathcal N
		\subseteq
		\ker(\mathcal C_1-\mathcal B_1\mathcal P_{\mathcal K}).
		\]
		Furthermore, by using $
		\mathcal A\mathcal N \subseteq \mathcal{N},$ it holds that
		\[
		(\mathcal C_1-\mathcal B_1\mathcal P_{\mathcal K})
		\mathcal A^j
		(\mathcal B_0-\Sigma_{\mathcal K}\mathcal C_0)=0,
		\qquad j=0,1,2,\ldots,
		\]
		thereby implying 
		\begin{equation}
			\label{eq:cor47_certificate}
			(\mathcal C_1-\mathcal B_1\mathcal P_{\mathcal K})
			(sI-\mathcal A)^{-1}
			(\mathcal B_0-\Sigma_{\mathcal K}\mathcal C_0)\equiv0.
		\end{equation}
Finally, for the case where $\operatorname{rank}\Sigma_{21}^{\mathcal K}=n$.
Using the similar argument as above and by the algebraic duality of feedback systems
\cite[Section~12.2.1]{1995Robust}, we get
\[
(\mathcal B_0-\Sigma_{\mathcal K}\mathcal C_0)^\top
(sI-\mathcal A^\top)^{-1}
(\mathcal C_1-\mathcal B_1\mathcal P_{\mathcal K})^\top
\equiv0.
\]
Taking transposes gives~\eqref{eq:cor47_certificate}.
	\end{proof}
	
	The optimality condition in Theorem \ref{main_res} has an interesting yet implicit connection with Theorem 2 of \cite{9993305}, indicating that condition \eqref{value} may find applications in the perturbed policy gradient methods for escaping suboptimal points. Specifically, by combining Theorem \ref{main_res} of this paper with Theorem 2 of \cite{9993305}, we can characterize the optimality condition in terms of the second-order derivatives of the objective function in \eqref{op_lqg}. This result is formalized in the following corollary. 
	To simplify the presentation, for any controller $\mathcal{K} \in \mathbb{K}_q$ and any positive integer $p$, we define the augmented controller
	\begin{equation*}
		\widetilde{\mathcal{K}}(\lambda) :=
		\left[\begin{array}{c|cc}
			0 & C_{\mathcal{K}} & 0\\ \hline
			B_{\mathcal{K}} & A_{\mathcal{K}} & 0\\
			0 & 0 & \lambda I_{p}
		\end{array}\right], \quad \lambda\in\mathbb{R}
	\end{equation*} 
	and the corresponding set
	\begin{equation*}
		\mathcal{M}_{p} := \left\{ 
		\Delta = \left[\begin{array}{c|cc}
			0 & 0 & \Delta_{C}\\ \hline
			0 & 0 & \Delta_{A1}\\
			\Delta_{B} & \Delta_{A2} & \Delta_{A3}
		\end{array}\right] \;\middle|\;
		\begin{aligned}
			& \Delta_B \in \mathbb{R}^{p \times m_2}, \Delta_{A1} \in \mathbb{R}^{q \times p}, \Delta_{A2} \in \mathbb{R}^{p \times q}, \\
			& \Delta_{A3} \in \mathbb{R}^{p \times p}, \Delta_C \in \mathbb{R}^{m_1 \times p}
		\end{aligned}
		\right\}.
\end{equation*}
\begin{corollary}\label{hessian}
		Let $\mathcal{K} \in \mathbb{K}_q$ and $p$ be a positive integer. Then, $\mathcal{K}$ is a globally optimal solution to problem~\eqref{op_lqg}, if and only if the second-order directional derivative of $J\!\left(\widetilde{\mathcal{K}}(\lambda)\right)$ with respect to $\Delta$ is identically zero, i.e.,
		\begin{equation}\label{value2}
			D_{\Delta}^2J\left(\widetilde{\mathcal{K}}(\lambda)\right) = 0, \quad \forall \lambda< 0,\ \Delta\in\mathcal{M}_{p}.
		\end{equation}
		Moreover, if $\mathcal{K}$ is not a globally optimal solution to problem \eqref{op_lqg}, then there exist $\Delta\in\mathcal{M}_{p}$ and $\lambda< 0$ such that $D_{\Delta}^2J(\widetilde{\mathcal{K}}(\lambda)) < 0$.
\end{corollary}
\begin{proof}
	See Appendix \ref{hessian2}.
\end{proof}

	\section{The Difference between optimality conditions of LQG and LQR}\label{diff}
	
	In this section, we present our second main contribution: a structural explanation for the existence of suboptimal stationary points in the LQG problem. It is well-known that stationary points of the LQG problem \eqref{op_lqg} are not necessarily globally optimal \cite{pmlr-v144-tang21a}. Our frequency-domain criterion, established in Theorem~\ref{main_res}, provides a precise tool to analyze this phenomenon. Specifically, a stationary point is optimal if and only if the optimality condition
\[
(\mathcal{C}_1-\mathcal{B}_1\mathcal{P}_\mathcal{K})(sI-\mathcal{A})^{-1}(\mathcal{B}_0-\Sigma_\mathcal{K}\mathcal{C}_0)\equiv 0, \quad \forall s\in\mathbb{C}
\]
is satisfied. This issue stands in stark contrast to the classic LQR problem, which is well-behaved in that, \textcolor{black}{when $W\succ0$}, its unique stationary point is always globally optimal. 
We first consider the LQR problem described by SDE
\begin{equation}\label{op_lqr}
	\begin{aligned}
		&\underset{K\in \mathbb{K}_{\rm LQR}}{\text{min}}\quad  \lim_{T \to \infty} \frac{1}{T} \mathbb{E} \left[ \int_0^T \left( X_t^\top Q X_t + U_t^\top R U_t \right) {\rm d}t \right] \\
		&~\text{s.t.} \quad
		\begin{cases}
			{\rm d}X_t = (A X_t + B U_t) {\rm d}t + W^{\frac{1}{2}}{\rm d}B_{t}, \\
			U_t = K X_t,
		\end{cases}
	\end{aligned}
\end{equation}
where $A\in\mathbb{R}^{n\times n}$ and
$B\in\mathbb{R}^{n\times m_1}$ are system matrices,
and $X_t$ and $U_t$ denote the state and control input, respectively.
The process $\{B_t\}_{t\geq0}$ is an $n$-dimensional standard
Brownian motion.
The process noise covariance matrix satisfies $W\succeq0$,
while $Q\succeq0$ and $R\succ0$ are the cost weighting matrices. $\mathbb{K}_{\rm LQR}$ denotes the set $\{K\in\mathbb{R}^{m_1\times n}: A+BK \text{ is Hurwitz}\}$.
By the same analysis in Lemma \ref{th1}, problem~\eqref{op_lqr} is equivalent to the following
$\mathcal{H}_2$ optimization problem.
\begin{equation}\label{lqr}
	\underset{K\in\mathbb{K}_{\text{LQR}}}{\mathrm{min}}\quad f_{LQR}(K) := \|G_K\|_{\mathcal{H}_2}^2.
\end{equation}
with \begin{align*}
	&G_K(s):=
	\begin{bmatrix}
		Q^{1/2}\\
		R^{1/2}K
	\end{bmatrix}
	(sI-A-BK)^{-1}W^{1/2}.
\end{align*}
Accordingly, a stationary point for the $\mathcal{H}_2$-formulated LQR problem is naturally defined as any stabilizing feedback gain~\( K \) that satisfies the first-order stationarity condition: $\nabla_{K} \|G_K\|_{\mathcal{H}_2}^2 = 0$.
Using the frequency perspective, we obtain the following sufficient and necessary optimality condition of problem \eqref{lqr}.
	%
	%
	
\begin{theorem}\label{th_lqr}
			The feedback gain $K$ is the optimal solution of problem \textcolor{black}{\eqref{lqr}}, if and only if
			\begin{align}\label{LQR_sn}
				(B^\top P_K + R K) (sI - (A+BK))^{-1} \Sigma_K\equiv 0,\  \forall s\in\mathbb{C}
			\end{align}
		holds, where $P_K$ and $\Sigma_K$ solve the Lyapunov equations
			\begin{equation}
				\begin{aligned}
					(A+BK)^\top P_K + P_K(A+BK) &= -(Q + K^\top R K), \\
					(A+BK) \Sigma_K + \Sigma_K (A+BK)^\top &= \textcolor{black}{-W}.\label{lqr_lyp}
				\end{aligned}
			\end{equation}
			Moreover, it holds that
			\begin{align*}
				\nabla_K f_{LQR}(K) &= 2\ResLHP{(B^\top P_K + R K) (sI - (A+BK))^{-1} \Sigma_K} \\
				&=2 (B^\top P_K + R K)\Sigma_K.
			\end{align*}
	\end{theorem}
	\begin{proof}
		To analyze the stationary points of problem~\textcolor{black}{\eqref{lqr}}, 
		we examine how the cost functional changes when the feedback gain is perturbed.  
		Specifically, for any $K \in \mathbb{K}_{LQR}$, we consider a small perturbation $\Delta K$ such that the new controller $K' = K + \Delta K$ also belongs to $\mathbb{K}_{LQR}$.
		As in the LQG case (see the proof of Lemma~\ref{th1} and Eqs.~\eqref{eq:lft_state}--\eqref{eq:h2_expansion} there), the LQR performance functional $f_{\text{LQR}}(K')$ can also be represented as the square of a $\mathcal{H}_2$ norm:
		\begin{equation}\label{lqr_square}
			f_{\text{LQR}}(K') = \left\| \mathbf{M}_{11}^K + \mathbf{M}_{12}^K \Delta K (I - \mathbf{M}_{22}^K \Delta K)^{-1} \mathbf{M}_{21}^K \right\|_{\mathcal{H}_2}^2
		\end{equation}
		with transfer matrices
		\begin{equation}
			\begin{aligned}\label{eq_LQR_trans}
				\mathbf{M}_{11}^K(s) &= \begin{bmatrix} Q^{1/2} \\ R^{1/2} K \end{bmatrix} (sI - (A + BK))^{-1}\textcolor{black}{W^{1/2}}, \\
				\mathbf{M}_{12}^{K}(s) &= \begin{bmatrix} Q^{1/2} \\ R^{1/2} K \end{bmatrix} (sI - (A + BK))^{-1} B + \begin{bmatrix} 0 \\ R^{1/2} \end{bmatrix}, \\
				\mathbf{M}_{21}^{K}(s) &= (sI - (A + BK))^{-1}\textcolor{black}{W^{1/2}}, \\
				\mathbf{M}_{22}^{K}(s) &= (sI - (A + BK))^{-1} B.
			\end{aligned}
		\end{equation}
		As shown in \eqref{eq_LQR_trans}, the squared $\mathcal{H}_2$ norm in \eqref{lqr_square} can be expanded as
		\begin{equation}\label{lqr_h2}
			\begin{aligned}
				& \|\mathbf{M}_{11}^{K}\|_{\mathcal{H}_2}^2 + \|\mathbf{M}_{12}^{K}\Delta K (I-\mathbf{M}_{22}^{K} \Delta K)^{-1} \mathbf{M}_{21}^{K}\|_{\mathcal{H}_2}^2 \\
				& + 2 \left\langle \mathbf{M}_{11}^{K},~ \mathbf{M}_{12}^{K} \Delta K (I - \mathbf{M}_{22}^{K} \Delta K)^{-1} \mathbf{M}_{21}^{K} \right\rangle_{\mathcal{H}_2}.
			\end{aligned}
		\end{equation}
		With similar derivation of Lemma \ref{lem1}, we can show that $K\in\mathbb{K}_{LQR}$ is global optimal if and only if
		\begin{equation}\label{lqr_cross}
			\left\langle \mathbf{M}_{11}^K,~ \mathbf{M}_{12}^K \Delta K (I - \mathbf{M}_{22}^K \Delta K)^{-1} \mathbf{M}_{21}^K \right\rangle_{\mathcal{H}_2}=0
		\end{equation}
		for all the admissible $\Delta K$. By Lemma \ref{res}, \eqref{lqr_cross} occurs precisely when the matrix-function
		\[
		(\mathbf{M}_{12}^K(-s))^{\top}\, \mathbf{M}_{11}^K(s) \, (\mathbf{M}_{21}^K(-s))^{\top}
		\]
		is anti-stable.
		To connect the condition of anti-stable to system matrices, we explicitly construct a state-space realization of $(\mathbf{M}_{12}^K(-s))^{\top}\, \mathbf{M}_{11}^K(s) \, (\mathbf{M}_{21}^K(-s))^{\top}$:
		\begin{equation}\label{eq_lqrd}
			\begin{aligned}
				&(\mathbf{M}_{12}^K(-s))^{\top}\, \mathbf{M}_{11}^K(s) \, (\mathbf{M}_{21}^K(-s))^{\top}\\
				& = 
				\left[
				\begin{array}{ccc|c}
					-(A+BK)^\top & Q + K^\top R K & 0 & 0 \\
					0 & A+BK & \textcolor{black}{W} & 0 \\
					0 & 0 & -(A+BK)^\top & -I \\
					\hline
					-B^\top & R K & 0 & 0
				\end{array}
				\right].
			\end{aligned}
		\end{equation}
		This realization highlights the connection between the system dynamics under $K$ and the instability condition. 
		
		Following the proof of Theorem \ref{main_res}, we introduce the transformation matrices
		\begin{equation*}
			T=\begin{bmatrix}
				I& 0&0\\
				0&I & -\Sigma_K\\
				0&0& I
			\end{bmatrix} \quad \text{and} \quad T'=\begin{bmatrix}
				I & {P}_K & 0 \\
				0 & I & 0 \\
				0 & 0 & I
			\end{bmatrix}.
		\end{equation*}
		Applying a similarity transformation with matrix $T'T$ to system matrices in \eqref{eq_lqrd} yields
		\begin{equation}
			\begin{gathered}
				T'T\begin{bmatrix}
					-(A+BK)^\top & Q + K^\top R K & 0  \\
					0 & A+BK & \textcolor{black}{W}  \\
					0 & 0 & -(A+BK)^\top
				\end{bmatrix}(T'T)^{-1},\\
				T'T\begin{bmatrix}
					0 \\
					0 \\
					I
				\end{bmatrix},\quad \begin{bmatrix}
					B^\top & R K & 0
				\end{bmatrix}(T'T)^{-1}
			\end{gathered}
		\end{equation}
		respectively. Then, the transformed realization is
		\begin{equation}
			\begin{aligned}
				&\label{eq_lqr} 
				(\mathbf{M}_{12}^K(-s))^{\top}\!\mathbf{M}_{11}^K(s) \! (\mathbf{M}_{21}^K(-s))^{\top}\\
				& =\! \!\!
				\left[
				\begin{array}{ccc|c}
					-(A+BK)^\top & 0 & (Q + K^\top R K)\Sigma_K & P_K \Sigma_K \\
					0 & (A+BK) & 0 & \Sigma_K \\
					0 & 0 & -(A+BK)^\top & -I \\
					\hline
					-B^\top & B^\top P_K + R K & R K \Sigma_K & 0
				\end{array}
				\right].
			\end{aligned}
		\end{equation}
		Based on \eqref{eq_lqr}, we can decompose $(\mathbf{M}_{12}^K(-s))^{\top}\, \mathbf{M}_{11}^K(s) \, (\mathbf{M}_{21}^K(-s))^{\top}$ into four subsystems:
		\begin{equation}\label{lqr_nab}
			\begin{aligned}
				&(\mathbf{M}_{12}^K(-s))^{\top}\, \mathbf{M}_{11}^K(s) \, (\mathbf{M}_{21}^K(-s))^{\top} \\
				&= -B^\top (sI + (A+BK)^\top)^{-1} P_K \Sigma_K - R K \Sigma_K (sI + (A+BK)^\top)^{-1} \\
				&\quad + B^\top(sI + (A+BK)^\top)^{-1}(Q + K^\top R K)\Sigma_K(sI + (A+BK)^\top)^{-1}\\
				&\quad + (B^\top P_K + R K) (sI - (A+BK))^{-1} \Sigma_K.
			\end{aligned}
		\end{equation}
		It is straightforward to verify that, among the terms of \eqref{lqr_nab}, only last term is stable, as it is associated with stable part of the closed-loop dynamics.
		Substituting \eqref{lqr_nab} into the cross-term $\langle\mathbf{M}_{11}^{K}, \mathbf{M}_{12}^{K} \Delta K (I - \mathbf{M}_{22}^K \Delta K)^{-1} \mathbf{M}_{21}^K \rangle_{\mathcal{H}_2}$, we obtain
		\begin{equation}
			\begin{aligned}
				&\langle\mathbf{M}_{11}^{K}, \mathbf{M}_{12}^{K} \Delta K (I - \mathbf{M}_{22}^K \Delta K)^{-1} \mathbf{M}_{21}^K \rangle_{\mathcal{H}_2}\\
				&=\operatorname{tr}\left(\int_{-\infty}^{\infty} \mathbf{M}_{12}^K(j\omega)^* \mathbf{M}_{11}^K(j\omega) \mathbf{M}_{21}^K(j\omega)^* (\Delta K (I - \mathbf{M}_{22}^K(j\omega) \Delta K)^{-1})^* \frac{d\omega}{2\pi} \right)\\
				&=\operatorname{tr}\left(\ResLHP{(B^\top P_K + R K) (sI \!-\! (A+BK))^{-1} \Sigma_K (\Delta K (I\! -\! \mathbf{M}_{22}^K(-s) \Delta K)^{-1})^{\top}}\right)\\
				&=\operatorname{tr}\left(\ResLHP{(B^\top P_K + R K) (sI - (A+BK))^{-1} \Sigma_K \Delta K^\top}\right)+o(\|\Delta K\|)\\
				&=\frac{1}{2} \operatorname{tr}(\nabla_K f_{LQR}(K)\Delta K^{\top}) + o(\|\Delta K\|),
			\end{aligned}\label{eq_lqr_value}
		\end{equation}
		where the last equality is due to
		\begin{equation*}
			\begin{aligned}
				\Delta K (I - \mathbf{M}_{22}^K(s) \Delta K)^{-1}=\Delta K+\Delta K((I - \mathbf{M}_{22}^K(s) \Delta K)^{-1}-I)
			\end{aligned}
		\end{equation*}
		and 
		\begin{equation*}
			\begin{aligned}
				\|\mathbf{M}_{12}^{K}\Delta K(I-\mathbf{M}_{22}^{K}\Delta K)^{-1}\mathbf{M}_{21}^{K}\|_{\mathcal{H}_2}^2=o(\|\Delta K\|).
			\end{aligned}
		\end{equation*}
		Following the argument in the proof of Theorem~\ref{main_res}, we establish that $K$ is an optimal solution to problem~\textcolor{black}{\eqref{lqr}} if and only if 
		\begin{equation}
			(B^\top P_K + R K) (sI - (A+BK))^{-1} \Sigma_K \equiv 0, \quad \forall s\in\mathbb{C}.\label{eq_lqr_in2}
		\end{equation}
		Comparing~\eqref{eq_lqr_value} with~\eqref{eq_lqr_in2}, we observe that the stationary point condition
		\[
			\textcolor{black}{\nabla_K f_{LQR}(K)=0}
		\]
		and the optimality condition for problem~\textcolor{black}{\eqref{lqr}} coincide. Specifically, we have
		\begin{equation}
			\begin{aligned}
				\nabla_K f_{LQR}(K) &= 2\ResLHP{(B^\top P_K + R K) (sI - (A+BK))^{-1} \Sigma_K} \\
				&=2 (B^\top P_K + R K)\Sigma_K.
			\end{aligned}\label{eq_nab_Lqr}
		\end{equation}
		This completes the proof.
	\end{proof}
	\begin{remark}
			Theorem \ref{th_lqr} provides the necessary and sufficient optimality condition \eqref{LQR_sn} for LQR problem and gives the relationship between the gradient $\nabla_{K}f_{LQR}$ and condition \eqref{LQR_sn}.  Condition \eqref{LQR_sn} is derived using an argument similar to that used in Theorem \ref{main_res} and can interpret why some controllers are not the solution of Riccati equation but still be optimal solution of LQR problem. Specifically, in Example 1 of \cite{watanabe2026revisiting}, an interesting model is proposed:
			\[
			A = \begin{bmatrix} -1 & a \\ a & -1 \end{bmatrix}, \ 
			B = \begin{bmatrix} 1 \\ 1 \end{bmatrix}, \  
			W = \begin{bmatrix} 1 & -1 \\ -1 & 1 \end{bmatrix} \succeq 0, \  
			Q = I \succ 0, \ 
			R = 1,\  a>-1.
			\]
			Note that the feedback gain $K = [k, k]$ with $k < -(a - 1)/2$ stabilizes the system as
			\[
			A + BK = \begin{bmatrix} -(1-k) & a+k \\ a+k & -(1-k) \end{bmatrix} 
			= \frac{1}{2}(a+2k-1) \begin{bmatrix} 1 \\ 1 \end{bmatrix} \begin{bmatrix} 1 & 1 \end{bmatrix} 
			- \frac{1}{2}(1+a) \begin{bmatrix} -1 \\ 1 \end{bmatrix} \begin{bmatrix} -1 & 1 \end{bmatrix}.
			\]
			For $K$, we can verify that  $\Sigma_K$ is given by
			\[
			\Sigma_K = \int_0^\infty e^{(A+BK)t} W e^{(A+BK)^\top t} dt = \frac{1}{2(1+a)} \begin{bmatrix} 1 & -1 \\ -1 & 1 \end{bmatrix},
			\]
			and $P_K$ is expressed as
			\begin{align*}
				P_K =& \int_0^\infty e^{(A+BK)^\top t} (Q+K^\top RK) e^{(A+BK)t} dt \\
				=& \frac{1+2k^2}{4(1-a-2k)} \begin{bmatrix} 1 & 1 \\ 1 & 1 \end{bmatrix} + \frac{1}{4(1+a)} \begin{bmatrix} 1 & -1 \\ -1 & 1 \end{bmatrix}.
			\end{align*}
			Thus, $f_{\text{LQR}}(K) = \text{tr}((Q + K^\top R K)\Sigma_K) = 1/(1+a), \forall k < -(a-1)/2$. From the symmetry of system matrix, it is straightforward to verify that the optimal gain $K^\star = -R^{-1}B^\top P^\star$ from Riccati equation \eqref{eq_riccati} must be of the form $K^\star = [k^\star, k^\star]$. Thus, all feedback gains of the form $K = [k, k]$ with $k < -(a-1)/2$ are globally optimal to problem \textcolor{black}{\eqref{lqr}}. In fact, if we take $K=[k,k]$ and the corresponding system parameters of problem into  \eqref{LQR_sn}, we can get
			\begin{align*}
				(B^\top P_K + R K) (sI - (A+BK))^{-1} \Sigma_K &= \begin{bmatrix} 0 & 0 \end{bmatrix}.
			\end{align*}
			Therefore, $K=[k,k],k<-(a-1)/2$ must be an optimal solution of problem \textcolor{black}{\eqref{lqr}}.
	\end{remark}
	Unlike the output-feedback LQG setting, when $W\succ 0$, the unconstrained LQR problem structurally admits a unique stationary point that constitutes the global optimum. The following corollary formalizes the insight gained from our framework (Theorem~\ref{th_lqr}). Although this result is well known and has been previously established, by applying the new proof developed through the analysis in Theorem \ref{th_lqr}, we demonstrate that  under the inherent structure of LQR problem, any stationary point satisfies the global optimality condition. This perspective not only recovers the classic result but also sets the stage for our subsequent analysis in Remark~\ref{rem}, which will show why suboptimal stationary points arise in the more challenging LQG setting.
		\begin{corollary}
				Assume $W\succ 0$. Then, there exists a unique stationary point of problem~\textcolor{black}{\eqref{lqr}}, and such a stationary point is the unique globally optimal solution of \textcolor{black}{\eqref{lqr}}.
		\end{corollary}
		\begin{proof}
				Equation~\eqref{eq_nab_Lqr} shows that the gradient of~\textcolor{black}{\eqref{lqr}} with respect to $K$ is precisely the residue associated with the optimality condition in~\eqref{eq_lqr_value}.
				Since $\Sigma_K \succ 0$, the transferred function in~\eqref{eq_lqr_value} is absolutely controllable. Therefore, \eqref{eq_lqr_value}  vanishes if and only if $\nabla_K f_{LQR}(K)=2(B^\top P_K + R K)\Sigma_K = 0$.
				Consequently, all the stationary points of problem~\textcolor{black}{\eqref{lqr}} correspond to the optimal solutions. The uniqueness of stationary point is characterized by $B^\top P_K + R K = 0$ together with Lyapunov equation~\eqref{lqr_lyp}. This completes the proof.
		\end{proof}
		
		\begin{remark}\label{rem}
			In the LQR problem with $W\succ 0$,  $\nabla_Kf_{\text{LQR}}(K)=0$ is equivalent to the global optimality condition \eqref{LQR_sn}. 
			For the general LQG problem, this equivalence no longer holds. The condition of stationarity, $\nabla J(\mathcal{K})=0$, is necessary for optimality but generally not sufficient. This gap between stationarity and optimality can be attributed to two distinct structural properties of the LQG problem.
			First, $\nabla J(\mathcal{K})$ can be rewritten as:
			\begin{equation}\label{grad_deg}
				\nabla J(\mathcal{K})=
				2\mathbf{e}_{(q+m_1)\times (q+m_2)}\big(
				\underbrace{\ResLHP{(\mathcal{C}_1 - \mathcal{B}_1 \mathcal{P}_\mathcal{K})(sI - \mathcal{A})^{-1}
						(\mathcal{B}_0 - \Sigma_\mathcal{K} \mathcal{C}_0)}}_{\mathcal{R}_{\mathcal{K}}},\, m_1, m_2 \big),
			\end{equation}
			where $\mathbf{e}(\cdot,\cdot,\cdot)$ is the projection operator (see Notation). This projection is the source of the first gap: $\nabla J(\mathcal{K})$ can be zero even if the residue $\mathcal{R}_{\mathcal{K}}$ is non-zero. Second, the global optimality condition requires the full transfer function to be identically zero:
			\begin{equation}\label{eq5}
				(\mathcal{C}_1 - \mathcal{B}_1 \mathcal{P}_\mathcal{K}) (sI - \mathcal{A})^{-1}
				(\mathcal{B}_0 - \Sigma_\mathcal{K} \mathcal{C}_0) \equiv 0, \quad \forall s \in \mathbb{C}.
			\end{equation}
		Note that $\mathcal{R}_{\mathcal{K}}=0$ is not sufficient to ensure that the transfer function in \eqref{eq5} is identically zero.
            The aforementioned analysis explains why stationary points that are not globally optimal can exist in the LQG optimization landscape.
		\end{remark}
		
		\section{Optimization algorithm based on the optimality condition}\label{alg}
		
		In the previous section, we identified a fundamental structural flaw in the direct parameterization \eqref{cons_k} of LQG problem. As detailed in Remark~\ref{rem}, essential frequency-domain information required for global optimality is lost in the gradients of traditional, direct parameterizations \eqref{cons_k}. This is not merely a theoretical curiosity; it directly explains why popular gradient-based methods can be trapped at suboptimal points in the LQG setting \cite{pmlr-v144-tang21a,9993305}.
		
		To overcome this limitation, this section develops a new controller synthesis framework that preserves this critical information. Our approach is built upon the Youla parameterization, a central tool that recasts the original non-convex synthesis problem into a convex optimization over an infinite-dimensional space of stable transfer functions (the Youla parameter) \cite{MAHTOUT202081}. Crucially, rather than approximating this problem as a finite-dimensional one, a step that introduces its own unavoidable errors, we design a gradient method that operates directly on transfer function space.
		By working in this infinite-dimensional setting, the Fréchet-gradient-based algorithm developed here retains the full frequency-domain structure of the problem. This allows us to establish \textcolor{black}{a global sublinear convergence guarantee for the objective values}, thereby addressing the principal weaknesses of existing methods. Our analysis in this section proceeds under the standard assumption that all the system parameters ($\mathcal{A}$, $\mathcal{B}_1$, $\mathcal{C}_1$, $\mathcal{B}_0$, $\mathcal{C}_0$, $\mathcal{P}_\mathcal{K}$, and $\Sigma_\mathcal{K}$) are known.
		\subsection[Optimization over the Space RH infinity]{Optimization over the Space $\mathcal{RH}_\infty$}
		
		%
		
		To recover the frequency-domain information in optimality condition~\eqref{value}, we consider controllers in an infinite-dimensional function space. Specifically, given a fixed initial controller $\mathcal{K}\in\mathbb{K}_q$, we define an enlarged controller set~$\mathbb{K}$, consisting of all elements of the form $\Delta\mathcal{K} = \widetilde{\mathcal{K}} - \mathcal{K}$ for $\widetilde{\mathcal{K}}\in\mathbb{K}_q$, a translated copy of~$\mathbb{K}_q$ centered at $\mathcal{K}$. More precisely,
		\begin{equation}
			\label{set_new}
			\mathbb{K}\! :=\! \left\{ \Delta \mathcal{K}_1(s)+\Delta \mathcal{K}_2 \;\middle|\;
			\begin{array}{@{}l@{}}  
				\Delta \mathcal{K}_1(s)\in \mathcal{R}_{sp}^{(m_1+q)\times (q+m_2)}\cup \{0_{(m_1+q)\times (q+m_2)}\}, \\
				\Delta \mathcal{K}_2\in \mathbb{R}^{(q+m_1)\times(q+m_2)\mid m_1\times m_2}, \\
				\left((\Delta \mathcal{K}_1(s)+\Delta \mathcal{K}_2)(I - \mathbf{M}_{22}^\mathcal{K} (\Delta \mathcal{K}_1(s)+\Delta \mathcal{K}_2))^{-1}\right)\ \text{is stable}
			\end{array}
			\right\}.
		\end{equation}
		This expanded set provides a broader search space for controller synthesis, potentially admitting novel solutions outside the original finite-dimensional set~$\mathbb{K}_q$. Importantly, as seen from~\eqref{eq:lft_transfer}, for any $\mathcal{K}' \in \mathbb{K}_q$, the difference $\Delta\mathbf{K} = \mathcal{K}' - \mathcal{K}$ satisfies $\Delta\mathbf{K} \in \mathbb{K}$. Therefore, the translated copy $\{\mathcal{K}' - \mathcal{K}\mid \mathcal{K}'\in\mathbb{K}_q\}$ of $\mathbb{K}_q$ is included in the expanded set $\mathbb{K}$. As a result, the classic optimal LQG controller (as defined in~\eqref{opt_lqg}) is still attained within the enlarged optimization domain, and thus the optimal value of the expanded problem coincides with that of the original problem~\eqref{op_lqg}.
		
		If we directly substitute $\Delta\mathbf{K}\in\mathbb{K}$ into the original performance criterion as in problem~\eqref{eq:h2_problem}, we then have
		\begin{equation*}
			J_\mathcal{K}(\Delta\mathbf{K}) = \|\mathbf{M}_{11}^\mathcal{K} + \mathbf{M}_{12}^\mathcal{K} \Delta\mathbf{K} (I - \mathbf{M}_{22}^\mathcal{K} \Delta\mathbf{K})^{-1} \mathbf{M}_{21}^\mathcal{K}\|_{\mathcal{H}_2}^2,
		\end{equation*}
		which leads to a nonconvex optimization problem in $\Delta\mathbf{K}$.
		To resolve this, we use the Youla parameterization to reformulate the controller in $\mathbb{K}$, resulting in a convex problem in a new parameter space. For any $\Delta\mathbf{K}\in\mathbb{K}$ and a fixed initial controller $\mathcal{K}\in\mathbb{K}_q$, there exists a unique pair $(\mathbf{Q}, Q)$ in the set $\mathbb{U}$ such that
		\[
		\Delta\mathbf{K}(I-\mathbf{M}_{22}^{\mathcal{K}}\Delta\mathbf{K})^{-1} = \mathbf{Q} + Q,
		\]
		where the set $\mathbb{U}$ is defined as
		\begin{equation}
			\label{eq:set_U_pretty}
			\begin{aligned}
				\mathbb{U} :=& \left\{ (\mathbf{Q}(s), Q) \;\middle|\;
				\begin{array}{@{}l@{}}  
					\mathbf{Q}(s)\in\mathcal{RH}_{\infty,0}^{(m_1+q)\times (q+m_2)}\cup\{ 0_{{q+m_1}\times (q+m_2)}\}, \\
					Q \in \mathbb{R}^{(q+m_1)\times(q+m_2)\mid m_1\times m_2}\cup\{0_{(q+m_1)\times (q+m_2)}\}
				\end{array}
				\right\} \\
			\end{aligned}.
		\end{equation}
		{The following lemma formalizes the above one-to-one correspondence.}
		
		\begin{lemma}\label{lem_rea}
			Let $\mathcal{K} \in \mathbb{K}_q$, and $\forall \Delta\mathbf{K} \in\mathbb{K}$. 
			Then, the following assertions hold
			\begin{enumerate}
				\item There exists a unique element $(\mathbf{Q}(\Delta\mathbf{K}), Q(\Delta\mathbf{K})) \in \mathbb{U}$ such that $\Delta\mathbf{K}(I - \mathbf{M}_{22}^\mathcal{K}\Delta\mathbf{K})^{-1} = \mathbf{Q}(\Delta\mathbf{K}) + Q(\Delta\mathbf{K})$.
				\item Conversely, for any $(\mathbf{Q}, Q) \in \mathbb{U}$, there exists a unique $\Delta\mathbf{K}\in\mathbb{K}$ such that $\Delta \mathbf{K} = (I+(\mathbf{Q} + Q)\mathbf{M}_{22}^\mathcal{K})^{-1}(\mathbf{Q} + Q)$.
				\item Letting $\mathcal{K}'\in\mathbb{K}_q$ and $\Delta\mathbf{K}=\mathcal{K}'-\mathcal{K}\in\mathbb{K}$, it holds that $Q(\Delta\mathbf{K}) = \Delta\mathbf{K}$ and 
				\begin{equation*}
					\mathbf{Q}(\Delta\mathbf{K})(s) = \Delta\mathbf{K} 
					\begin{bmatrix}
						C & 0 \\
						0 & I
					\end{bmatrix}
					(sI -
					\begin{bmatrix}
						A & B C_{\mathcal{K}'} \\
						B_{\mathcal{K}'} C & A_{\mathcal{K}'}
					\end{bmatrix})^{-1}
					\begin{bmatrix}
						B & 0 \\
						0 & I
					\end{bmatrix}
					\Delta\mathbf{K}.
				\end{equation*}
			\end{enumerate}
		\end{lemma}
		\begin{proof}
			The first and second conclusions can be verified using the Youla parametrization~\cite{1101223}. Since both $\mathbb{K}_q$ and $\mathbb{K}$ are defined as sets of controllers that ensure internal stability of the closed-loop system, the invertibility of all the relevant matrices (such as $I - \mathbf{M}_{22}^\mathcal{K}\Delta\mathbf{K}$ and $I + (\mathbf{Q} + Q)\mathbf{M}_{22}^\mathcal{K}$) is guaranteed by construction.
			
			For the third conclusion, note that $\Delta\mathbf{K}(I - \mathbf{M}_{22}^\mathcal{K}\Delta\mathbf{K})^{-1}$ has a realization:
			\begin{equation}
				\begin{aligned}
					\dot{\xi}_t = 
					\begin{bmatrix}
						A & B C_{\mathcal{K}'} \\
						B_{\mathcal{K}'} C & A_{\mathcal{K}'}
					\end{bmatrix}
					\xi_t + 
					\begin{bmatrix}
						B & 0 \\
						0 & I
					\end{bmatrix}\Delta\mathbf{K}
					u_t,\ 
					y_t = 
					\Delta\mathbf{K}\begin{bmatrix}
						C & 0 \\
						0 & I
					\end{bmatrix}\xi_t + 
					\Delta \mathbf{K} u_t
				\end{aligned}
			\end{equation}
			with $\mathcal{K}'=\mathcal{K}+\Delta\mathbf{K}$. Based on this realization, we let
			\begin{equation}
				\begin{aligned}
					\mathbf{Q}(\Delta\mathbf{K})(s)=\Delta\mathbf{K}\begin{bmatrix}
						C & 0 \\
						0 & I
					\end{bmatrix}(sI-\begin{bmatrix}
						A & B C_{\mathcal{K}'} \\
						B_{\mathcal{K}'} C & A_{\mathcal{K}'}
					\end{bmatrix})^{-1}\begin{bmatrix}
						B & 0 \\
						0 & I
					\end{bmatrix}\Delta\mathbf{K},\ 
					Q(\Delta\mathbf{K})=\Delta\mathbf{K}.
				\end{aligned}
			\end{equation}
			This completes the proof.
		\end{proof}
	
	Since $\mathbf{M}_{22}^{\mathcal K}$ is strictly proper and $\Delta\mathbf K$ and $\mathbf Q+Q$ are proper,
	both $I-\mathbf{M}_{22}^{\mathcal K}\Delta\mathbf K$
	and $I+(\mathbf Q+Q)\mathbf{M}_{22}^{\mathcal K}$
	tend to the identity matrix as $s\to\infty$. This means $I-\mathbf{M}_{22}^{\mathcal K}\Delta\mathbf K$
	and $I+(\mathbf Q+Q)\mathbf{M}_{22}^{\mathcal K}$ are both invertible.
		Building on the above translation, we lift the feasible set of the original controller optimization problem~\eqref{op_lqg} to an infinite-dimensional parameter space $\mathbb{U}$ associated with an arbitrary initial controller $\mathcal{K} \in \mathbb{K}_q$. Crucially, by the mapping described in Lemma~\ref{lem_rea}, the closed-loop term $\Delta\mathbf{K}(I-\mathbf{M}_{22}^{\mathcal{K}}\Delta\mathbf{K})^{-1}$ in the objective can be represented as $(\mathbf{Q} + Q)$. Replacing this term in the cost functional of problem~\eqref{op_lqg} yields the following equivalent performance index:
		\[
		J_{\mathcal{K}}(\mathbf{Q}, Q) := \left\| \mathbf{M}_{11}^\mathcal{K} 
		+ \mathbf{M}_{12}^\mathcal{K} (\mathbf{Q} + Q) \mathbf{M}_{21}^\mathcal{K}
		\right\|_{\mathcal{H}_2}^2, \qquad \forall\, (\mathbf{Q}, Q) \in\mathbb{U}.
		\]
		It is this new objective function that forms the focus of our subsequent analysis. Notably, the original nonlinear term in the objective is now replaced by a linear parameterization in the convex set $\mathbb{U}$, rendering the problem convex. This convexity permits the application of efficient convex optimization methods and provides a foundation for the theoretical development that follows.
		
		To facilitate the analysis, we introduce the following inner product on $\mathbb{U}$:
		\begin{equation}
			\langle (\mathbf{Q}, Q), (\mathbf{Q}', Q') \rangle_{\mathbb{U}} := \langle \mathbf{Q}, \mathbf{Q}' \rangle_{\mathcal{H}_2} + \langle Q, Q' \rangle_F\label{inner}
		\end{equation}
		with the corresponding norm
		\begin{equation}
			\|(\mathbf{Q}, Q)\|_{\mathbb{U}} := \sqrt{ \|\mathbf{Q}\|_{\mathcal{H}_2}^2 + \|Q\|_F^2 }.\label{norm}
		\end{equation}
		We also define an operator $\mathcal{S}[\cdot]$ that extracts the stable part of any transfer function matrix. That is, for any transfer matrix $\mathbf{N}(s)$, we have a unique decomposition $\mathbf{N}(s) = \mathbf{N}_1(s) + \mathbf{N}_2(s)$ (see section 3 in \cite{1995Robust} for details), where $\mathbf{N}_1(s)$ is stable and $\mathbf{N}_2(s)$ is anti-stable, and then we set $\mathcal{S}[\mathbf{N}(s)] := \mathbf{N}_1(s)$.
		
		The enlargement from $\mathbb{K}_q$ to $\mathbb{U}$ fundamentally resolves the two main limitations discussed in Remark~\ref{rem}. First, by extending the parameter space to the infinite-dimensional set $\mathbb{U}$, all the structural restrictions on the feasible set are removed so that arbitrary search directions can be explored. Second, $J_\mathcal{K}$ becomes convex over $\mathbb{U}$, and thus any stationary point is globally optimal. In particular, as shown in the following theorem, the gradient of $J_\mathcal{K}(\mathbf{Q},Q)$ contains a transfer function matrix that precisely recovers the frequency-domain optimality condition characterized previously (see Theorem~\ref{main_res}). Consequently, the essential frequency-domain information is fully preserved throughout the optimization process.
		With notations \eqref{inner} and \eqref{norm}, we can analyze the gradient, smoothness, and other optimization properties of $J_{\mathcal{K}}(\mathbf{Q},Q)$ in the convex and structurally unconstrained parameter space $\mathbb{U}$.
		
		\begin{theorem}\label{thm:gradient}
			Let $\mathcal{K} \in \mathbb{K}_q$ be a given initial controller. Then, the Fréchet derivative of cost functional $J_{\mathcal{K}}$ with respect to $(\mathbf{Q}, Q) \in \mathbb{U}$ is given by
			\[
			\nabla J_{\mathcal{K}}(\mathbf{Q}, Q) = 2\left( \mathbf{S},\; \mathbf{e}_{(q+m_1)\times(q+m_2)}\left(\ResLHP{\mathbf{S}},m_1,m_2\right) \right),
			\]
			where
			\begin{align*}
				\mathbf{S}(s) =& \mathcal{S}\Bigl[(\mathcal{C}_1 - \mathcal{B}_1 \mathcal{P}_{\mathcal{K}} )( sI - \mathcal{A} )^{-1} ( \mathcal{B}_0 - \Sigma_{\mathcal{K}} \mathcal{C}_0 )\\
				&+ ( \mathbf{M}^{\mathcal{K}}_{12} (-s))^\top \mathbf{M}^{\mathcal{K}}_{12} ( \mathbf{Q}(s) + Q) \mathbf{M}_{21}^{\mathcal{K}}(s) ( \mathbf{M}_{21}^{\mathcal{K}}(-s) )^\top \Bigr].
			\end{align*}
		\end{theorem}
		
		\begin{proof}
			Consider the perturbed cost functional with variation $(\Delta \mathbf{Q},\Delta Q)$
			\begin{align*}
				&J_{\mathcal{K}}(\mathbf{Q}+\Delta \mathbf{Q},Q+\Delta Q) \\
				=& \| \mathbf{M}_{11}^{\mathcal{K}} + \mathbf{M}_{12}^{\mathcal{K}}(\mathbf{Q}(s)+Q + \Delta \mathbf{Q}+\Delta Q)\mathbf{M}_{21}^{\textcolor{black}{\mathcal{K}}} \|_{\mathcal{H}_2}^2 \\
				=& \| \mathbf{M}_{11}^{\mathcal{K}} + \mathbf{M}_{12}^{\mathcal{K}}(\mathbf{Q}+Q)\mathbf{M}_{21}^{\mathcal{K}} \|_{\mathcal{H}_2}^2\!\\
				& +\! 2\int_{-\infty}^{\infty}\!\! \operatorname{tr}\Big[ \Bigl( \mathbf{M}_{11}^{\mathcal{K}}(j\omega)\! +\! \mathbf{M}_{12}^{\mathcal{K}}(j\omega)(\mathbf{Q}(j\omega)+Q)\mathbf{M}_{21}^{\mathcal{K}}(j\omega) \Bigr)^* \mathbf{M}_{12}^{\mathcal{K}}(j\omega) \Delta \mathbf{Q}(j\omega)\\ &\times\mathbf{M}_{21}^{\mathcal{K}}(j\omega) \Big]\frac{d\omega}{2\pi}+2\operatorname{tr}(\ResLHP{\mathbf{S}(s)}\Delta Q^\top)+\mathcal{O}(\|(\Delta \mathbf{Q},\Delta Q)\|^2_\mathbb{U}).
			\end{align*}
			The first-order term can be rewritten as
			\begin{align*}
				&2\!\int_{-\infty}^{\infty}\! \operatorname{tr} \Bigl( \mathbf{M}_{21}^{\mathcal{K}}(j\omega)^*\! \Delta \mathbf{Q}(j\omega)^* \mathbf{M}_{12}^{\mathcal{K}}(j\omega)^*\! \bigl( \mathbf{M}_{11}^{\mathcal{K}}(j\omega) \!\!+\! \mathbf{M}_{12}^{\mathcal{K}}(j\omega)(\mathbf{Q}(j\omega)\!+\!Q)\mathbf{M}_{21}^{\mathcal{K}}(j\omega) \bigr) \Bigr) \frac{d\omega}{2\pi}\\
				&+2\operatorname{tr}(\ResLHP{\mathbf{S}(s)}\Delta Q^\top).
			\end{align*}
			By Lemma \ref{res}, this expression can be simplified to
			\begin{align*}
				&2\int_{-\infty}^{\infty} \operatorname{tr}\left[ \Bigl( \mathbf{S}(j\omega)^* \Delta (\mathbf{Q}(j\omega)) \right] \frac{d\omega}{2\pi} +2\operatorname{tr}(\ResLHP{\mathbf{S}(s)}\Delta Q^{\top}).
			\end{align*}
			Noting $\Delta Q\in\mathbb{R}^{(q+m_1)\times(q+m_2)\mid m_1\times m_2}$, we can get
			\begin{equation*}
				\begin{aligned}
					&2\int_{-\infty}^{\infty} \operatorname{tr} \left( \mathbf{S}(j\omega)^* \Delta (\mathbf{Q}(j\omega)) \right) \frac{d\omega}{2\pi} +2\operatorname{tr}(\ResLHP{\mathbf{S}(s)}\Delta Q^{\top})\\
					=&2\int_{-\infty}^{\infty} \operatorname{tr}\left(  \mathbf{S}(j\omega)^* \Delta (\mathbf{Q}(j\omega))\right)  \frac{d\omega}{2\pi}\\
					&+2\operatorname{tr}(\mathbf{e}_{(q+m_1)\times(q+m_2)}\left(\ResLHP{\mathbf{S}(s)},m_1,m_2\right)\Delta Q^{\top}).
				\end{aligned}
			\end{equation*}
			The result follows from the definition of stable projection $\mathcal{S}[\cdot]$. 
		\end{proof}
		
		\begin{lemma}\label{ass:regularity}
			(Lipschitz smoothness of $J_{\mathcal{K}}$)
			Let $\mathcal{K}$ be any stabilizing controller of system~\eqref{op_lqg}. Then,  the cost functional $J_{\mathcal{K}}(\mathbf{Q}, Q)$ is Lipschitz smooth with respect to $(\mathbf{Q}, Q)$. Specifically, there exists a constant $L_{\mathcal{K}} > 0$, which depends only on the controller $\mathcal{K}$, such that for any pair $(\mathbf{Q}, Q), (\mathbf{Q}', Q') \in \mathbb{U}$, the following inequality holds:
			\begin{equation}\label{eq:lipschitz_smoothness}
				\begin{aligned}
					J_{\mathcal{K}}(\mathbf{Q}, Q) \leq\, & J_{\mathcal{K}}(\mathbf{Q}', Q')  
					+ \langle \nabla J_\mathcal{K}(\mathbf{Q}',Q'), (\mathbf{Q}-\mathbf{Q}',Q-Q') \rangle_{\mathbb{U}} \\
					&+ \frac{L_{\mathcal{K}}}{2} \left( \| \mathbf{Q} - \mathbf{Q}' \|_{\mathcal{H}_2}^2 + \| Q - Q' \|_F^2 \right),
				\end{aligned}
			\end{equation}
			where $\nabla J_\mathcal{K}(\mathbf{Q}',Q')$ denotes the gradient evaluated at $(\mathbf{Q}', Q')$.
		\end{lemma}
		
		\begin{proof}
			To streamline the notation, let us define the perturbations as $\Delta\mathbf{Q} = \mathbf{Q}-\mathbf{Q}'$ and $\Delta Q = Q-Q'$.
			First, we expand $J_{\mathcal{K}}(\mathbf{Q}, Q)$ around $J_{\mathcal{K}}(\mathbf{Q}', Q')$. Using the definition from \eqref{eq:h2_expansion}, we get
			\begin{align*}
				J_{\mathcal{K}}(\mathbf{Q},Q)
				&= J_{\mathcal{K}}(\mathbf{Q}',Q') + \langle \nabla J_\mathcal{K}(\mathbf{Q}',Q'), (\Delta\mathbf{Q}, \Delta Q) \rangle_{\mathbb{U}} + \|\mathbf{M}_{12}^\mathcal{K}(\Delta\mathbf{Q}+\Delta Q)\mathbf{M}_{21}^\mathcal{K}\|^2_{\mathcal{H}_2}.
			\end{align*}
			The next step is to bound the quadratic term $\|\mathbf{M}_{12}^\mathcal{K}(\Delta\mathbf{Q}+\Delta Q)\mathbf{M}_{21}^\mathcal{K}\|^2_{\mathcal{H}_2}$. Using the inequality $\|A+B\|^2 \leq 2\|A\|^2 + 2\|B\|^2$,  $\|\mathbf{M}_{12}^\mathcal{K}(\Delta\mathbf{Q}+\Delta Q)\mathbf{M}_{21}^\mathcal{K}\|_{\mathcal{H}_2}^2$ can be bounded as
			\begin{equation}\label{eq:quad_term_split}
				\|\mathbf{M}_{12}^\mathcal{K}(\Delta\mathbf{Q}+\Delta Q)\mathbf{M}_{21}^\mathcal{K}\|^2_{\mathcal{H}_2} \leq 2\|\mathbf{M}_{12}^\mathcal{K}\Delta\mathbf{Q}\mathbf{M}_{21}^\mathcal{K}\|^2_{\mathcal{H}_2} + 2\|\mathbf{M}_{12}^\mathcal{K}\Delta Q\mathbf{M}_{21}^\mathcal{K}\|^2_{\mathcal{H}_2}.
			\end{equation}
			We now bound each of these two terms using properties of $\mathcal{H}_\infty$ and $\mathcal{H}_2$ norms, as well as the residue theorem for the constant term. This leads to the following detailed bound:
			\begin{align*}
				&\|\mathbf{M}_{12}^\mathcal{K}(\Delta\mathbf{Q}+\Delta Q)\mathbf{M}_{21}^\mathcal{K}\|^2_{\mathcal{H}_2} \\
				&\leq 2 \left( \frac{\|\mathbf{M}_{12}^\mathcal{K}\|_{\mathcal{H}_\infty}^2\|\mathbf{M}_{21}^\mathcal{K}\|_{\mathcal{H}_\infty}^2}{2\pi}\int_{-\infty}^{\infty} \operatorname{tr}\left(\Delta\mathbf{Q}(j\omega)^* \Delta\mathbf{Q}(j\omega)\right)  d\omega \right) \\
				&\quad + 2\operatorname{tr}\left(\ResLHP{(\mathbf{M}_{21}^\mathcal{K}(-s))^\top\mathbf{M}_{21}^\mathcal{K}(s)\Delta Q(\mathbf{M}_{12}^\mathcal{K}(-s))^\top\mathbf{M}_{12}^\mathcal{K}(s)\Delta Q^{\top}}\right).
			\end{align*}
			Recognizing $\|\Delta\mathbf{Q}\|_{\mathcal{H}_2}^2 = \frac{1}{2\pi} \int_{-\infty}^{\infty} \operatorname{tr}(\Delta\mathbf{Q}(j\omega)^* \Delta\mathbf{Q}(j\omega)) d\omega$ and that the residue term is bounded by $C_{\mathcal{K}}'\|\Delta Q\|_F^2$ for some constant $C_{\mathcal{K}}'$ (from Lemma \ref{res}), we can see that there must exist a constant $L_{\mathcal{K}} > 0$ such that
			\begin{equation*}
				\|\mathbf{M}_{12}^\mathcal{K}(\Delta\mathbf{Q}+\Delta Q)\mathbf{M}_{21}^\mathcal{K}\|^2_{\mathcal{H}_2} \leq \frac{L_{\mathcal{K}}}{2} \left( \| \Delta\mathbf{Q} \|_{\mathcal{H}_2}^2 + \| \Delta Q \|_F^2 \right).
			\end{equation*}
			Substituting this bound back into our expansion for $J_{\mathcal{K}}(\mathbf{Q},Q)$ yields the inequality \eqref{eq:lipschitz_smoothness}. This establishes that $J_{\mathcal{K}}$ is Lipschitz smooth, and an explicit value for the constant $L_{\mathcal{K}}$ can be constructed from the bounds derived above. This completes the proof.
		\end{proof}
		
		\subsection{The Gradient-Based Method and Convergence Analysis}
		
		In this section, we present the gradient-based method and analyze its convergence rate. Algorithm \ref{alg:Q_update} presents a method for designing controllers over $\mathcal{RH}_\infty$. Specifically, existing literature offers numerous classic methods for controller design by solving optimization problems (such as the $\mathcal{H}_2$ problem \eqref{eq:h2_problem}) over $\mathcal{RH}_\infty$. A representative classical approach is the Ritz approximation in Ref.~[3], which is used to solve the following controller design problem:
		\begin{equation}\tag{R.1}\label{resp:ritz-inf}
			\min_{H\in\mathcal{D}} \phi(H), 
		\end{equation}	
		where $H$ is the closed-loop transfer matrix, $\mathcal{D}$ is a convex subset of $\mathcal{RH}_\infty$, and $\phi: \mathcal{D} \to \mathbb{R}$ is a quasi-convex (or convex) function. The Ritz method chooses a sequence of transfer matrices $H_0,H_1,\dots,H_N\in\mathcal{RH}_\infty$ and approximates the solution to problem \eqref{resp:ritz-inf} by solving the following finite-dimensional optimization problem as $N$ increases:
		\begin{equation}\tag{R.2}\label{resp:ritz-fi}
			\min_{\substack{a\in\mathbb{R}^N\\H_N(a)\in\mathcal{D}}} \phi(H_N(a)),
		\end{equation}
		where
		\begin{equation*}
			H_N(a) = H_0 + a_1H_1 + \dots + a_NH_N,\qquad a=(a_1,\dots,a_N)^\top\in\mathbb{R}^N.
		\end{equation*}
		When $\phi$ is a squared $\mathcal{H}_2$-norm function of $H$ and $H_0,\dots,H_N$ have finite $\mathcal{H}_2$ norms, the gradient of $\phi(H_N(a))$ with respect to $a$ can be written as frequency-domain integrals involving $H_N(a)$ and $H_1,\dots,H_N$ (see Ref.~[3]), and the standard optimization methods can then be applied to solve the approximate problem \eqref{resp:ritz-fi}, where the approximation error depends on $N$ and the chosen transfer matrices $H_i$. The primary drawbacks of this classical method are the potential approximation error and the increasing dimension of the decision vector $a$ that may be required to accurately solve problem \eqref{resp:ritz-inf}. In contrast, Algorithm \ref{alg:Q_update} directly uses the explicit gradient of the $\mathcal{H}_2$ objective with respect to the transfer matrices. The controller is then updated via a gradient-based method, thereby circumventing the need for a prescribed finite-dimensional Ritz parameterization and eliminating the approximation error associated with this restriction.
			\begin{algorithm}
			\setlength{\linewidth}{\hsize}
				\caption{Iterative Update of $\mathbf{Q}_k(s)$}
				\label{alg:Q_update}
				\begin{algorithmic}
					
					\STATE{\textbf{Input:} Initial controller $\mathcal{K}_0$; system matrices $(\mathcal{A}_0, \mathcal{B}_0^0, \mathcal{B}_1, \mathcal{C}_0, \mathcal{C}_1^0)$; step size $\eta$; maximum iteration $N$;}
					\STATE{\textbf{Output:} Sequences $\{\mathbf{Q}_k(s)\}$, $\{Q_k\}$, \textcolor{black}{the final increment $\Delta\mathbf{K}_N$}, and the final controller $\mathcal{K}_N$;}
					
					\STATE{Initialize $\mathbf{Q}_0(s) \gets 0_{{q+m_1}\times (q+m_2)}$}
					\STATE{Initialize $Q_0 \gets 0_{{q+m_1}\times (q+m_2)}$}
					
					\FOR{$k = 0$ \TO $N-1$}
					\STATE{Compute:
						\[
						\begin{aligned}
							\mathbf{S}_k(s) \gets \mathcal{S} \Big[&\, (\mathcal{C}_1^0 - \mathcal{B}_1 \mathcal{P}_{\mathcal{K}_0}) (sI - \mathcal{A}_0)^{-1} (\mathcal{B}_0^0 - \Sigma_{\mathcal{K}_0} \mathcal{C}_0) \\
							&+ (\mathbf{M}^{\mathcal{K}_0}_{12}(-s))^\top \mathbf{M}^{\mathcal{K}_0}_{12}(s) (\mathbf{Q}_k(s)+Q_k) \mathbf{M}_{21}^{\mathcal{K}_0}(s) (\mathbf{M}_{21}^{\mathcal{K}_0}(-s))^\top
							\Big]
						\end{aligned}
						\]}
					\COMMENT{Obtain the stable part via modal decomposition and removing unstable modes.}
					
					\STATE{Update $\mathbf{Q}_{k+1}(s) \gets \mathbf{Q}_k(s) - \eta\, \mathbf{S}_k(s)$}
					\STATE{Update $Q_{k+1} \gets Q_k - \eta\, \mathbf{e}_{(q+m_1)\times (q+m_2)}(\ResLHP{\mathbf{S}_k(s)},m_1,m_2)$}
					\ENDFOR
					
					\STATE{\textcolor{black}{Recover $\Delta\mathbf{K}_N \gets \bigl(I+(\mathbf{Q}_N+Q_N)\mathbf{M}_{22}^{\mathcal{K}_0}\bigr)^{-1}(\mathbf{Q}_N+Q_N)$}}
					
					\STATE{\color{black} Construct $\mathcal{K}_N$ from $\mathcal{K}_0$ and
						$\Delta\mathbf{K}_N$ using the augmented realization
						in \eqref{eq:controller-reconstruction}.}
					
					\RETURN $\{\mathbf{Q}_k(s)\},\, \{Q_k\},\, \textcolor{black}{\Delta \mathbf{K}_N},\, \mathcal{K}_N$
					
				\end{algorithmic}
				
				\vspace{0.8ex}
				\hrule 
				\vspace{0.8ex}
				
				\noindent \textbf{Remark:} All the inverses in Algorithm~\ref{alg:Q_update} are well-defined in every iteration according to Youla parametrization.
			\end{algorithm}
			{\setlength{\emergencystretch}{1em}
			
			In Algorithm \ref{alg:Q_update} above, we set
			\begin{equation*}
				\mathcal{A}_0 = \begin{bmatrix} A & BC_{\mathcal{K}_0} \\ B_{\mathcal{K}_0}C & A_{\mathcal{K}_0} \end{bmatrix},\qquad
				\mathcal{B}_0^0 = 
				\begin{bmatrix}
					0 & 0\\
					B_{\mathcal{K}_0}\mathcal{V} & 0
				\end{bmatrix},\qquad 
				\mathcal{C}_1^0 =
				\begin{bmatrix}
					0 & RC_{\mathcal{K}_0}\\
					0 & 0
				\end{bmatrix}.
			\end{equation*}  
			At the final iteration, let $G_{\mathcal{K}_0}$ denote the nominal closed-loop system (the generalized plant formed by interconnecting the original plant with the initial controller parameter $\mathcal{K}_0$). The update $\Delta\mathbf{K}_{N}$, designed to improve performance, is then interconnected with $G_{\mathcal{K}_0}$ via a lower linear fractional transformation, as illustrated in \cite[Section~10.1]{1995Robust}.  
The final controller $\mathcal{K}_N$ is obtained by applying
the augmented state-space construction in
\eqref{eq:controller-reconstruction} to the initial controller
$\mathcal{K}_0$ and the perturbation
$\Delta\mathbf{K}_N$.

			\begin{theorem}\label{thm:lip}There exists a constant $L > 0$ such that for any $(\mathbf{Q}, Q),\,(\mathbf{Q}', Q') \in \mathbb{U}$,
				\begin{align}
					\frac{1}{L} \left\| \nabla J_{\textcolor{black}{\mathcal{K}_0}}(\mathbf{Q}', Q') - \nabla J_{\textcolor{black}{\mathcal{K}_0}}(\mathbf{Q}, Q) \right\|_\mathbb{U}^2
					\leq \left\langle \nabla J_{\textcolor{black}{\mathcal{K}_0}}(\mathbf{Q}', Q') - \nabla J_{\textcolor{black}{\mathcal{K}_0}}(\mathbf{Q}, Q),\, \Delta \right\rangle_{\mathbb{U}}
					\label{eq:lip}
				\end{align}
				holds with $\Delta = (\mathbf{Q}', Q') - (\mathbf{Q}, Q)$.
			\end{theorem}
			
			\begin{proof}
				Consider the auxiliary function
				\[
				\phi(\hat{\mathbf{Q}}, \hat{Q}) := J_{\textcolor{black}{\mathcal{K}_0}}(\hat{\mathbf{Q}}, \hat{Q}) - \langle \nabla J_{\textcolor{black}{\mathcal{K}_0}}(\mathbf{Q}, Q), (\hat{\mathbf{Q}}, \hat{Q}) \rangle_\mathbb{U}, \forall (\hat{\mathbf{Q}},\hat{Q})\in\mathbb{U},
				\]
			which is convex and achieves the minimum at $(\mathbf{Q},Q)$. Since $J_{\textcolor{black}{\mathcal{K}_0}}$ is $L_{\mathcal{K}_0}$-smooth, the function $\phi$ is also $L_{\mathcal{K}_0}$-smooth. This implies the following inequality:
				\begin{equation}\label{er1}
					\phi({\mathbf{Q}}, {Q})  \leq \phi(\mathbf{Q}', Q') + \langle \nabla \phi(\mathbf{Q}', Q'),\, (\hat{\mathbf{Q}}, \hat{Q}) - (\mathbf{Q}', Q') \rangle_\mathbb{U} + \frac{L_{\mathcal{K}_0}}{2} \|(\hat{\mathbf{Q}}, \hat{Q}) - (\mathbf{Q}', Q')\|_\mathbb{U}^2. 
				\end{equation}
				\textcolor{black}{Taking $(\hat{\mathbf{Q}}, \hat{Q})=(\mathbf{Q}',Q')-L_{\mathcal{K}_0}^{-1}\nabla\phi(\mathbf{Q}',Q')$ into~\eqref{er1}, we obtain}
				\begin{equation}\label{eq_b1}
					\phi(\mathbf{Q}, Q)  \leq \phi(\mathbf{Q}', Q')-\frac{1}{2L_{\mathcal{K}_0}}\|\nabla \phi(\mathbf{Q}',Q')\|_{\mathbb{U}}^2,
				\end{equation}
				which can be rewritten as
				\begin{equation}\label{eq_er1}
					J_{\mathcal{K}_0}(\mathbf{Q}, Q)  \leq J_{\mathcal{K}_0}(\mathbf{Q}', Q')-\left\langle \nabla J_{\textcolor{black}{\mathcal{K}_0}}(\mathbf{Q}, Q),\, \Delta \right\rangle_{\mathbb{U}}-\frac{1}{2L_{\mathcal{K}_0}}\|\nabla \phi(\mathbf{Q}',Q')\|_{\mathbb{U}}^2.
				\end{equation}
				\textcolor{black}{Repeating the same argument after interchanging $(\mathbf{Q},Q)$ and $(\mathbf{Q}',Q')$ and centering the auxiliary function at $(\mathbf{Q}',Q')$, we obtain another inequality:}
				\begin{equation}\label{eq_b2}
					J_{\mathcal{K}_0}(\mathbf{Q}', Q')\leq 	J_{\mathcal{K}_0}(\mathbf{Q}, Q) +\left\langle \nabla J_{\textcolor{black}{\mathcal{K}_0}}(\mathbf{Q}^\prime, Q^\prime),\, \Delta \right\rangle_{\mathbb{U}}-\frac{1}{2L_{\mathcal{K}_0}}\|\nabla \phi(\mathbf{Q}',Q')\|_{\mathbb{U}}^2.
				\end{equation}
				Setting $L=L_{\mathcal{K}_0}$ and combining~\eqref{eq_er1} with \eqref{eq_b2} lead to the desired result~\eqref{eq:lip}.
			\end{proof}
			\begin{remark}\label{rmk:space_extension}
				This result constitutes an extension of the classic Lipschitz continuity (or smoothness) of the gradient  in finite-dimensional Euclidean space (see, e.g.,~\cite[Sec.~2.1]{Nesterov2018}), to the infinite-dimensional setting with two notable distinctions.
				\begin{itemize}
					\item The gradient operator $\nabla J_{\textcolor{black}{\mathcal{K}_0}}(\mathbf{Q}, Q)$ is defined with respect to transfer functions (see Theorem~\ref{thm:gradient}), in contrast to gradients in Euclidean spaces that are defined with respect to finite-dimensional vectors.
					\item The inner product $\langle\cdot, \cdot\rangle_\mathbb{U}$ is, by construction, a composite of the $\mathcal{H}_2$ inner product and the standard Euclidean inner product, taking into account both the transfer function and the static gain components.
				\end{itemize}
			\end{remark}
			Building on Theorem \ref{thm:lip}, we now establish the convergence of Algorithm~\ref{alg:Q_update}.
			
			\begin{theorem}[Sublinear Convergence]\label{thm:conv}
				Provided that the step size satisfies $\textcolor{black}{0<}\eta < 2/L_{\mathcal{K}_0}$, the sequence of objective values, $\{J_{\mathcal{K}_0}(\mathbf{Q}_k, Q_k)\}_{k\ge0}$, generated by Algorithm~\ref{alg:Q_update} converges sublinearly to the optimal value of the  LQG problem.
			\end{theorem}
			
			\begin{proof}
				Let $\mathcal{K}^*$ denote an optimal controller, and let
				\[
				r_k = \sqrt{ \| \mathbf{Q}_k - \mathbf{Q}_{\mathcal{K}^*} \|_{\mathcal{H}_2}^2 + \| Q_k - Q_{\mathcal{K}^*} \|_F^2 },
				\]
				where $(\mathbf{Q}_{\mathcal{K}^*}, Q_{\mathcal{K}^*})$ is the optimal solution associated with $\mathcal{K}^* \in \mathbb{K}_q$. \textcolor{black}{Because the gradient in Theorem~\ref{thm:gradient} is twice the update direction used in Algorithm~\ref{alg:Q_update}, the algorithm is gradient descent with effective step size $\eta/2$. Thus,}
				\begin{align*}
					r_{k+1}^2 &= \| (\mathbf{Q}_k, Q_{k}) - (\mathbf{Q}_{\mathcal{K}^*}, Q_{\mathcal{K}^*}) - \textcolor{black}{\frac{\eta}{2}} \nabla J_{\textcolor{black}{\mathcal{K}_0}}(\mathbf{Q}_k,Q_k) \|_\mathbb{U}^2 \\
					&= \|(\mathbf{Q}_k, Q_{k}) - (\mathbf{Q}_{\mathcal{K}^*}, Q_{\mathcal{K}^*}) \|_\mathbb{U}^2 
					- \textcolor{black}{\eta} \langle \nabla J_{\textcolor{black}{\mathcal{K}_0}}(\mathbf{Q}_k,Q_k), (\mathbf{Q}_k, Q_{k})- (\mathbf{Q}_{\mathcal{K}^*}, Q_{\mathcal{K}^*}) \rangle_\mathbb{U} \\
					&\quad+ \textcolor{black}{\frac{\eta^2}{4}} \| \nabla J_{\textcolor{black}{\mathcal{K}_0}}(\mathbf{Q}_k,Q_k) \|_\mathbb{U}^2 \\
					&\leq r_k^2 - \textcolor{black}{\left(\frac{\eta}{L_{\mathcal{K}_0}}-\frac{\eta^2}{4}\right)} \| \nabla J_{\textcolor{black}{\mathcal{K}_0}}(\mathbf{Q}_k,Q_k) \|_\mathbb{U}^2 \quad \text{(by Theorem~\ref{thm:lip})}.
				\end{align*}
				\textcolor{black}{Since the step size satisfies}
$\textcolor{black}{0<\eta<2/L_{\mathcal{K}_0}},$
			the coefficient of $\| \nabla J_{\textcolor{black}{\mathcal{K}_0}}(\mathbf{Q}_k, Q_k) \|_\mathbb{U}^2$ in the last inequality above, \textcolor{black}{$\eta/L_{\mathcal{K}_0}-\eta^2/4$}, is positive, which guarantees the monotonic decrease of $r_k^2$.
			Let $\textcolor{black}{\Delta_k:=(\mathbf{Q}_k,Q_k)-(\mathbf{Q}_{\mathcal{K}^*},Q_{\mathcal{K}^*})}$. Due to the convexity of $J_{\textcolor{black}{\mathcal{K}_0}}(\cdot)$, it holds that
			\begin{align*}
				J_{\textcolor{black}{\mathcal{K}_0}}(\mathbf{Q}_k,Q_k) - J_{\textcolor{black}{\mathcal{K}_0}}(\mathbf{Q}_{\mathcal{K}^*}, Q_{\mathcal{K}^*}) 
				\leq \left\langle \nabla J_{\textcolor{black}{\mathcal{K}_0}}(\mathbf{Q}_k,Q_k), \textcolor{black}{\Delta_k} \right\rangle_{\textcolor{black}{\mathbb{U}}} 
				\leq r_0 \| \nabla J_{\textcolor{black}{\mathcal{K}_0}}(\mathbf{Q}_k, Q_k) \|_{\mathbb{U}}.
			\end{align*}
				By the (descent) lemma for $L$-smooth convex functions, one has
				\begin{align}\label{dec}
				J_{\textcolor{black}{\mathcal{K}_0}}(\mathbf{Q}_{k+1}, Q_{k+1}) 
				\leq J_{\textcolor{black}{\mathcal{K}_0}}(\mathbf{Q}_k, Q_k) 
				- \textcolor{black}{\frac{\eta}{2}\left(1 - \frac{L_{\mathcal{K}_0}\eta}{4}\right)} \| \nabla J_{\textcolor{black}{\mathcal{K}_0}}(\mathbf{Q}_k, Q_k) \|_\mathbb{U}^2.
				\end{align}
				Thus,
				\begin{equation}\label{sublinear}
				\Delta J_{\textcolor{black}{\mathcal{K}_0}}(\mathbf{Q}_{k+1}, Q_{k+1})
				\leq \Delta J_{\textcolor{black}{\mathcal{K}_0}}(\mathbf{Q}_k, Q_k)
				- \textcolor{black}{\frac{\frac{\eta}{2}\left(1 - \frac{L_{\mathcal{K}_0}\eta}{4}\right)}{r_0^2}} \left( \Delta J_{\textcolor{black}{\mathcal{K}_0}}(\mathbf{Q}_k, Q_k) \right)^2,
				\end{equation}
			where $\Delta J_{\textcolor{black}{\mathcal{K}_0}}(\mathbf{Q}_k, Q_k) := J_{\textcolor{black}{\mathcal{K}_0}}(\mathbf{Q}_k, Q_k) - J_{\textcolor{black}{\mathcal{K}_0}}(\mathbf{Q}_{\mathcal{K}^*}, Q_{\mathcal{K}^*})$. 
				Inequality~\eqref{sublinear} yields the claimed sublinear convergence.
			\end{proof}
			\begin{remark}
The preceding theoretical analysis guarantees the sublinear convergence of Algorithm~\ref{alg:Q_update} with respect to the objective values. However, this exact gradient approach presents an inherent tradeoff. While the cost monotonically improves, there is no rigorous theoretical guarantee for the convergence of controller sequence $(\mathbf{Q}_k,Q_k)$ in the functional space. Furthermore, a significant practical challenge arises from the monotonic growth in the (state-space) order of $\mathbf{Q}_k$ across iterations.  Consequently, the final synthesized controller can become structurally complicated.
				Addressing this controller complexity and scalability challenge remains a valuable direction for future work. One promising approach is to develop a hybrid framework that integrates Algorithm~\ref{alg:Q_update} with standard fixed-order policy gradient methods or their variants. Such a framework could use policy gradients for initial parameter updates, benefiting from the structural simplicity and computational efficiency of policy gradient method.
				It could then switch to the exact functional updates of Algorithm~\ref{alg:Q_update} to escape local optima or saddle points. Another unexplored direction is to design a novel algorithm based on Corollaries~\ref{lqg} and \ref{hessian}, which provide a sufficient optimality condition for stationary points and a necessary and sufficient optimality condition for LQG problem~\textcolor{black}{\eqref{op_lqg}}, respectively. While this approach could preserve the low order of controller, establishing rigorous theoretical convergence guarantees for such a method remains a challenging open problem.
			\end{remark}

			\section{Numerical Experiments}\label{num}
\subsection{Example 1: Escaping a Suboptimal Stationary Point}
			In this context, we aim to demonstrate two advantageous features of Algorithm \ref{alg:Q_update}: its rapid initial convergence and robust ability to bypass suboptimal solutions corresponding to stationary points of problem \eqref{op_lqg}. To this end, we consider the typical example from~\cite{1101812}, a case where the vanilla policy gradient method has been shown to be trapped at a suboptimal stationary point of problem \eqref{op_lqg}. This example illustrates our algorithm's performance on a complex optimization landscape. The system matrices are given by:
			\begin{equation}\label{eq}
				A = \begin{bmatrix} -0.5 & 0 \\ 0.5 & -1 \end{bmatrix}, \quad B = \begin{bmatrix} -1 \\ 1 \end{bmatrix}, \quad C = \begin{bmatrix} -\frac{1}{6} & \frac{11}{12} \end{bmatrix},
			\end{equation}
			and the performance weights are set as
			\[
			\mathcal{W} = Q = \begin{bmatrix} 1 & 0 \\ 0 & 1 \end{bmatrix}, \quad \mathcal{V} = R = 1.
			\]
			The globally optimal LQG controller obtained by solving equations \eqref{eq_riccati} and \eqref{opt_lqg} is
			\[
			A_{\mathcal{K}^*} = \begin{bmatrix} -1.1 & 0.13 \\ 1.19 & -1.64 \end{bmatrix}, \quad B_{\mathcal{K}^*} = \begin{bmatrix} 0.11 \\ 0.45 \end{bmatrix}, \quad C_{\mathcal{K}^*} = \begin{bmatrix} 0.62 & -0.22 \end{bmatrix}.
			\]
			
			System \eqref{eq} is open-loop stabilizable, and Theorem 5 in \cite{Tang2023} guarantees that all the controllers $\mathcal{K}$ satisfying $B_{\mathcal{K}} = 0_{2\times 1}$ and $C_{\mathcal{K}} = 0_{1\times 2}$ are stationary points of problem \eqref{op_lqg}.
			
			We compare the proposed Algorithm \ref{alg:Q_update} with the vanilla policy gradient method. For Algorithm \ref{alg:Q_update}, we use a step size of $\eta = 0.1$ and set the number of iterations to $N = 14$. For the vanilla policy gradient method, a larger fixed step size of $\eta_2 = 10$ and the same iteration count $N$ are used. To evaluate and compare the performance of both methods, we adopt the normalized relative cost error as the metric. For Algorithm \ref{alg:Q_update}, the error at iteration $k$ is defined as
			\[
			\text{Error}_{\text{Q-update}} = \frac{J_{\mathcal{K}_0}(\mathbf{Q}_k, Q_k) - J(\mathcal{K}^*)}{J(\mathcal{K}^*)},
			\]
			where $J_{\mathcal{K}_0}(\mathbf{Q}_k, Q_k)$ is the cost at iteration $k$ and $J(\mathcal{K}^*)$ is the known optimal cost. For the vanilla policy gradient method, the error is calculated analogously as
			\[
			\text{Error}_{\text{PG}} = \frac{J(\mathcal{K}_k) - J(\mathcal{K}^*)}{J(\mathcal{K}^*)},
			\]
			where $J(\mathcal{K}_k)$ denotes the cost at iteration $k$ for the policy gradient method. We conduct two experiments with controller initialization of increasing difficulty to systematically evaluate the ability of both algorithms to handle more challenging scenarios.
			
			First, we initialize the controller at a point near the suboptimal stationary point ($B_{\mathcal{K}} = 0_{2\times 1}, A_{\mathcal{K}} = -0.5I_2, C_{\mathcal{K}} = 0_{1\times 2}$), as defined by
			\[
			A_{\mathcal{K}_0} = \begin{bmatrix} -0.5 & 0 \\ 0 & -0.5 \end{bmatrix}, \quad B_{\mathcal{K}_0} = \begin{bmatrix} 0 \\ 0.01 \end{bmatrix}, \quad C_{\mathcal{K}_0} = \begin{bmatrix} 0 & -0.01 \end{bmatrix}.
			\]
			The convergence results are presented in Figure 1(a). As shown, within a small number of iterations, the vanilla policy gradient method exhibits virtually no change and fails to make any visible progress, further confirming its inability to leave the basin of attraction of this suboptimal point in the early phase of optimization. In contrast, Algorithm \ref{alg:Q_update} consistently reduces the cost error, demonstrating its ability to successfully navigate away from this suboptimal region.
			
			To create a more definitive test, we then initialize the controller exactly at the suboptimal stationary point 
			\begin{equation*}
				\begin{aligned}\label{eq_stat_exact}
					A_{\mathcal{K}_0} = \begin{bmatrix}
						-0.5 & 0\\
						0 & -0.5
					\end{bmatrix},\quad
					B_{\mathcal{K}_0} = \begin{bmatrix}
						0\\
						0
					\end{bmatrix},\quad
					C_{\mathcal{K}_0} = \begin{bmatrix}
						0 & 0
					\end{bmatrix},
				\end{aligned}
			\end{equation*}
			a location where the gradient $\nabla J(\mathcal{K}_0)$ is identically zero.
			The outcome is shown in Figure~\ref{fig:comparison}(b): as expected, the vanilla policy gradient method makes no progress. Notably, even under this extreme condition where first-order information vanishes, our algorithm identifies an efficient search direction and achieves cost reduction within just 14 iterations. Interestingly, we observe that the convergence curves in this case are nearly identical to those obtained when the initial point is set near (but not exactly at) the stationary point (Figure~\ref{fig:comparison}(a)). The difference in performance between the two cases is tiny, with the convergence curves being nearly identical. This remarkable consistency suggests that the algorithm's performance is robust to the initial conditions within this stationary region.
			\begin{figure}[tbhp] 
				\centering
				\subfloat[Case 1: Initialized near the stationary point.]{%
					\label{fig:case1}%
					\includegraphics[width=0.48\linewidth]{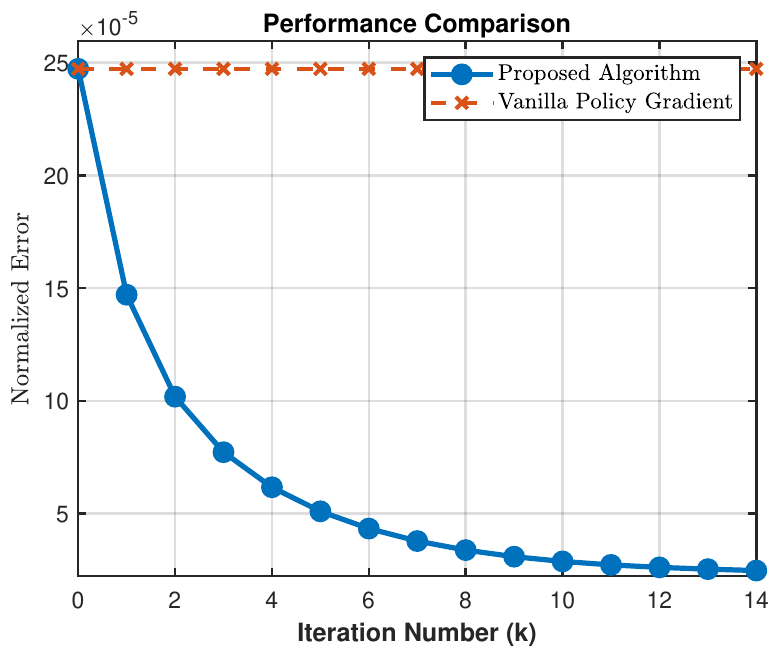}%
				}
				\hfill 
				\subfloat[Case 2: Initialized exactly at the stationary point.]{%
					\label{fig:case2}%
					\includegraphics[width=0.48\linewidth]{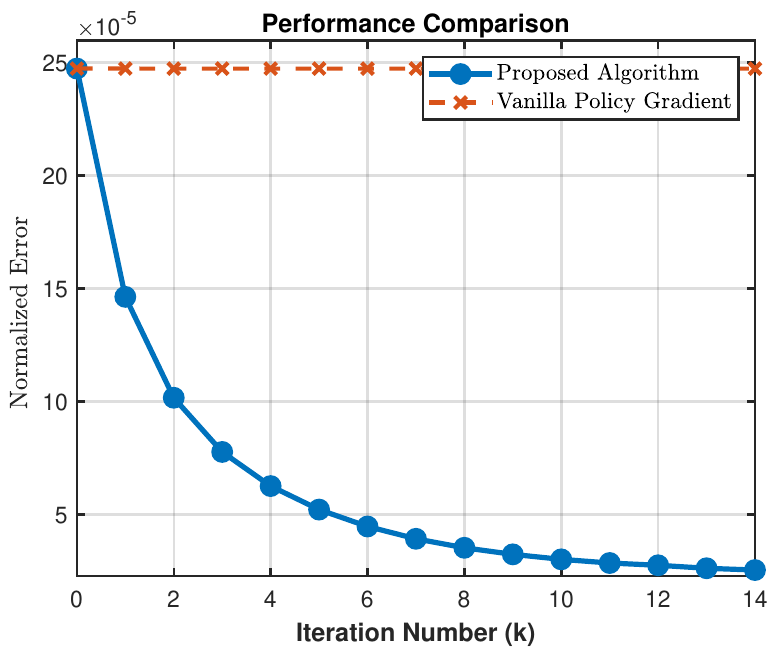}%
				}
				\caption{Convergence comparison between vanilla policy gradient and Algorithm~\ref{alg:Q_update} when initialized (a) near and (b) exactly at a suboptimal stationary point.}
				\label{fig:comparison} 
			\end{figure}

\subsection{Example 2: Validation of Data-Driven Estimation Methods}\label{data_test}

This section validates the data-driven estimation methods proposed in Appendix~\ref{data}, which constitute essential components of our algorithm. Specifically, we evaluate the accuracy of estimating: 1) the transfer function matrix $\mathbf{M}_{22}^{\mathcal{K}_0}(s)$, 2) the sensitivity function $\mathbf{S}_k(s)$, and 3) the residue term
\begin{align*}
&\mathbf{e}_{(q+m_1)\times(q+m_2)}
  \bigl(\ResLHP{\mathbf{S}_k(s)},m_1,m_2\bigr)
\end{align*}
via zeroth-order methods.

\subsubsection{Experimental Setup}
All the subsequent tests are conducted at the first iteration ($k=0$) of Algorithm~\ref{alg:Q_update}, with the system parameters fixed as specified in Example~1. Only the controller $\mathcal{K}_0$ is initialized with the following state-space parameters, which serve as the basis for all the subsequent calculations:
\begin{align*}
&A_{\mathcal{K}_0}=\begin{bmatrix}-0.5&0\\0&-0.5\end{bmatrix},\qquad
B_{\mathcal{K}_0}=\begin{bmatrix}0\\1\end{bmatrix},\qquad
C_{\mathcal{K}_0}=\begin{bmatrix}0&-1\end{bmatrix}.
\end{align*}
Our objective is to verify that these data-driven techniques can accurately estimate the relevant parameters based on this fixed initial controller, while keeping the system parameters unchanged.

\subsubsection[Validation of M22 Estimation]{Validation of $\mathbf{M}_{22}^{\mathcal{K}_0}(s)$ Estimation}
First, we validate the system identification procedure for $\mathbf{M}_{22}^{\mathcal{K}_0}(s)$, as detailed in Section~\ref{est_M}. The true transfer function matrix is given by
\begin{align}
&\mathbf{M}_{22}^{\mathcal{K}_0}(s)=
\begin{bmatrix}
\mathrm{TF}_{11}(s)&0&\mathrm{TF}_{13}(s)\\
0&\mathrm{TF}_{22}(s)&0\\
\mathrm{TF}_{31}(s)&0&\mathrm{TF}_{33}(s)
\end{bmatrix}.
\label{eq_M}
\end{align}
With the common denominator $D(s)$, its non-zero entries are
\begin{align*}
&D(s):=s^3+2s^2+2.333s+0.4167,\\
&\mathrm{TF}_{11}(s)=\frac{1.083s^2+0.7083s+0.08333}{D(s)},\\
&\mathrm{TF}_{13}(s)=\frac{-1.083s-0.1667}{D(s)},\qquad
\mathrm{TF}_{22}(s)=\frac{1}{s+0.5},\\
&\mathrm{TF}_{31}(s)=\frac{1.083s+0.1667}{D(s)},\qquad
\mathrm{TF}_{33}(s)=\frac{s^2+1.5s+0.5}{D(s)}.
\end{align*}
We observe several zero entries in \eqref{eq_M}. These entries represent the decoupled input-output channels where specific inputs have no dynamic influence on certain outputs. Consequently, these channels exhibit a null frequency response, making them irrelevant for our data-driven fitting process. We therefore focus on identifying only the non-zero elements. To generate the dataset for this validation, we simulate the ideal, noise-free scenario discussed in Section~\ref{est_M}. Specifically, we compute the exact frequency response of each non-zero element of the true transfer function matrix $\mathbf{M}_{22}^{\mathcal{K}_0}(s)$ from \eqref{eq_M} at the 200 frequency points specified below. This resulting set of complex values serves as the empirical data for the fitting algorithm. This approach allows us to isolate and verify the intrinsic accuracy of the parameter fitting procedure itself, separate from the effects of stochastic noise.

The estimation is performed under the following settings:
\begin{itemize}
\item \textbf{Model Order:} Numerator degree 2, denominator degree 3.
\item \textbf{Frequency Sampling:} 200 frequency points sampled at uniform intervals from 0.1 to 100 rad/s.
\item \textbf{Frequency Weighting:} Uniform weight $c_\omega=1$.
\end{itemize}
The resulting errors between the true and estimated parameters are presented in Table~\ref{tab:m22_error}, which demonstrates a high degree of accuracy and validates the effectiveness of the fitting procedure.

\begin{table}[htbp]
\centering
\caption{Coefficient errors (\%) for non-zero elements of $\mathbf{M}_{22}^{\mathcal{K}_0}(s)$.}
\label{tab:m22_error}
\begin{tabular}{lccccc}
\toprule
Coefficients & $\mathrm{TF}_{11}$ & $\mathrm{TF}_{13}$ & $\mathrm{TF}_{22}$ & $\mathrm{TF}_{31}$ & $\mathrm{TF}_{33}$\\
\midrule
Numerator & 0.0067 & 0.0016 & 0 & 0.0013 & 0.0002\\
Denominator & 0.0067 & 0.0014 & 0 & 0.0018 & 0\\
\bottomrule
\end{tabular}
\end{table}

\subsubsection[Validation of S0 Estimation]{Validation of $\mathbf{S}_0(s)$ Estimation}
Next, we evaluate the projection-based estimation of the sensitivity function $\mathbf{S}_0(s)$ (from Section~\ref{est_S}). The fitted order model (FOM) is
\begin{align*}
&\mathbf{S}_0(s)=
\begin{bmatrix}
\mathbf{S}_0(s)_{11}&0&\mathbf{S}_0(s)_{13}\\
0&0&0\\
\mathbf{S}_0(s)_{31}&0&\mathbf{S}_0(s)_{33}
\end{bmatrix},\\
&\mathbf{S}_0(s)_{11}=\frac{-0.8676s^2-2.908s-1.671}{D(s)},\\
&\mathbf{S}_0(s)_{13}=\frac{-1.513s^2-2.916s-2.081}{D(s)},\\
&\mathbf{S}_0(s)_{31}=\frac{1.513s^2+2.916s+2.081}{D(s)},\\
&\mathbf{S}_0(s)_{33}=\frac{1.774s^2+2.921s+2.683}{D(s)}.
\end{align*}
Similarly, we focus on identifying the non-zero elements.
For the purpose of this validation, we leverage the known analytical form of the true sensitivity function $\mathbf{S}_0(s)$ to generate the data. Instead of estimating the Laguerre coefficients via directional derivatives of the performance index as described in Section~\ref{est_S}, we compute them directly and with high precision by calculating the $\mathcal{H}_2$ inner product (i.e., projection) of the true $\mathbf{S}_0(s)$ onto each Laguerre basis function. This yields an ideal, error-free set of coefficients for constructing the Laguerre basis approximation, $\hat{\mathbf{S}}_0(s)$. The subsequent reduced-order fitting is then performed on the frequency response of this Laguerre approximation. This approach allows us to decouple the validation of the approximation and reduction steps from the numerical errors inherent in the derivative-based estimation process.

The system is first approximated using a Laguerre basis expansion with order varying from 1 to 11. Next, we fit a reduced-order transfer function model (numerator degree 2 and denominator degree 3) to the Laguerre-based approximation. Figure~\ref{fig:error_convergence_comparison} illustrates the $\mathcal{H}_2$ norm of the estimation error for all the non-zero elements of $\mathbf{S}_0(s)$, showing both the error from the Laguerre expansion (as a function of order) and from the reduced-order model. The plots show that as the Laguerre basis order increases, the estimation error decreases rapidly for both approaches. This confirms that the Laguerre basis efficiently represents system dynamics, and the reduced-order model accurately captures behavior when the Laguerre approximation is sufficiently accurate.
In practice, a Laguerre basis of order around 10 to 15 suffices for a highly accurate reduced-order model.

\begin{figure}[htbp]
\centering
\subfloat[Estimation error for $\mathbf{S}_0(s)_{11}$.\label{fig:s_11}]{%
\includegraphics[width=0.48\linewidth]{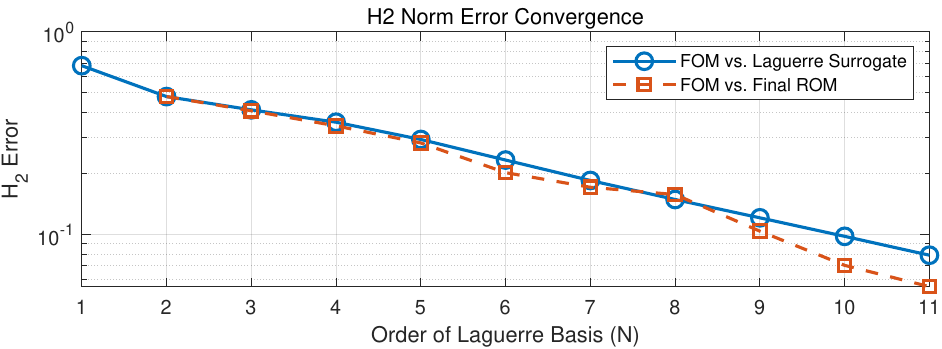}}\hfill
\subfloat[Estimation error for $\mathbf{S}_0(s)_{13}$.\label{fig:s_13}]{%
\includegraphics[width=0.48\linewidth]{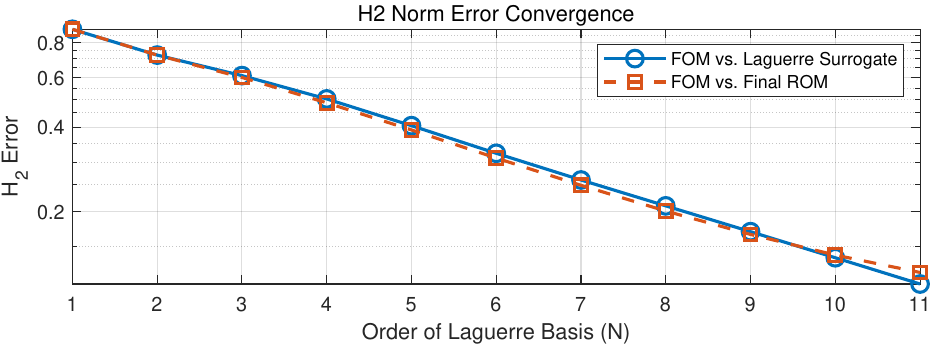}}

\medskip
\subfloat[Estimation error for $\mathbf{S}_0(s)_{31}$.\label{fig:s_31}]{%
\includegraphics[width=0.48\linewidth]{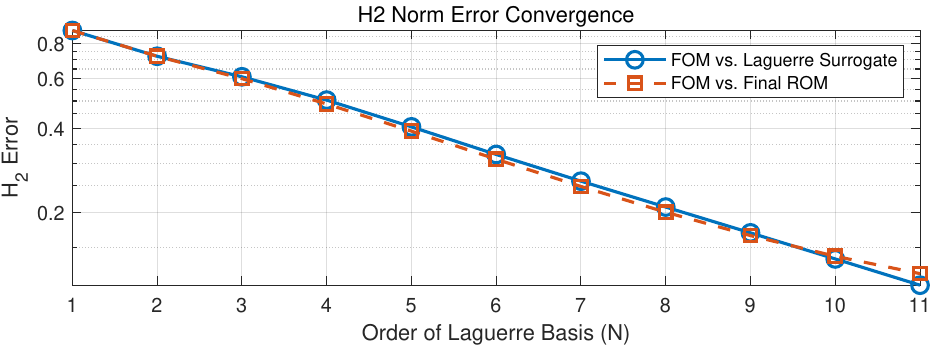}}\hfill
\subfloat[Estimation error for $\mathbf{S}_0(s)_{33}$.\label{fig:s_33}]{%
\includegraphics[width=0.48\linewidth]{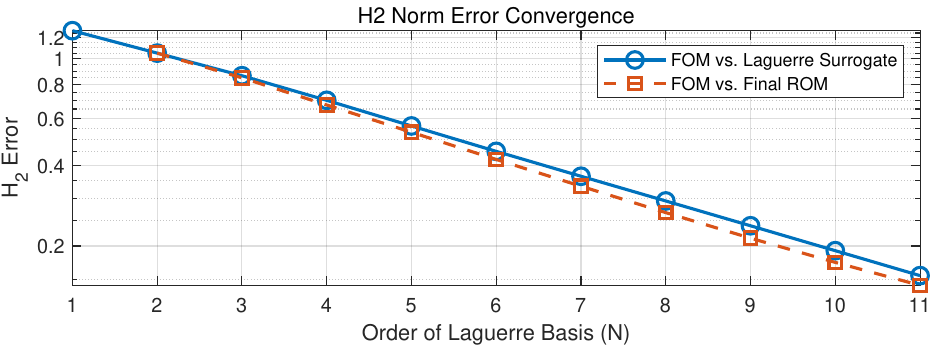}}
\caption{Convergence of the $\mathcal{H}_2$ estimation error for the non-zero elements of the sensitivity matrix $\mathbf{S}_0(s)$: comparison between Laguerre basis expansion and reduced-order model (both as functions of Laguerre basis order).}
\label{fig:error_convergence_comparison}
\end{figure}

\subsubsection{Validation of Residue Estimation}
Finally, we test the Monte Carlo based estimation of the residue term
\begin{align*}
&\mathbf{e}_{3\times3}\bigl(\ResLHP{\mathbf{S}_0(s)},1,1\bigr)
\end{align*}
(from Section~\ref{est_res}). The key parameters for this simulation are set as:
\begin{itemize}
\item \textbf{Radius of Sampling Sphere ($r$):} $1\times10^{-5}$.
\item \textbf{Sample Number:} $m=10,100,1000,10000$.
\end{itemize}
{\setlength{\emergencystretch}{1em}
Figure~\ref{fig:gradient_convergence_2x2} and Table~\ref{tab:error_comparison_simple} collectively demonstrate the method's performance. The heatmaps in Figure~\ref{fig:gradient_convergence_2x2} provide a qualitative visualization, showing the estimated gradient converging toward the true sparse structure as the number of Monte Carlo samples ($m$) increases. The quantitative results in Table~\ref{tab:error_comparison_simple} confirm this trend, with the relative error decreasing from 110.8\% at $m=10$ to 1.99\% at $m=10,000$. The reported percentage errors are calculated from the ratio\par}

\begin{align*}
&\frac{\|\text{Error Matrix}\|_F}{\|\text{Real Value}\|_F}.
\end{align*}

\begin{figure}[htbp]
\centering
\subfloat[Estimation result for $m=10$.\label{subfig:m10}]{%
\includegraphics[width=0.48\linewidth]{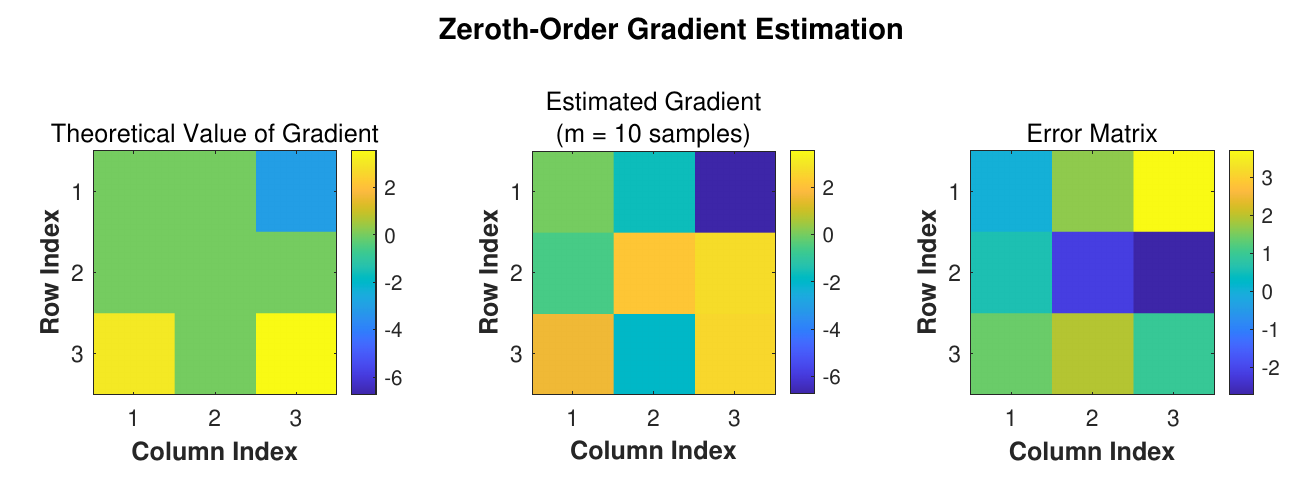}}\hfill
\subfloat[Estimation result for $m=100$.\label{subfig:m100}]{%
\includegraphics[width=0.48\linewidth]{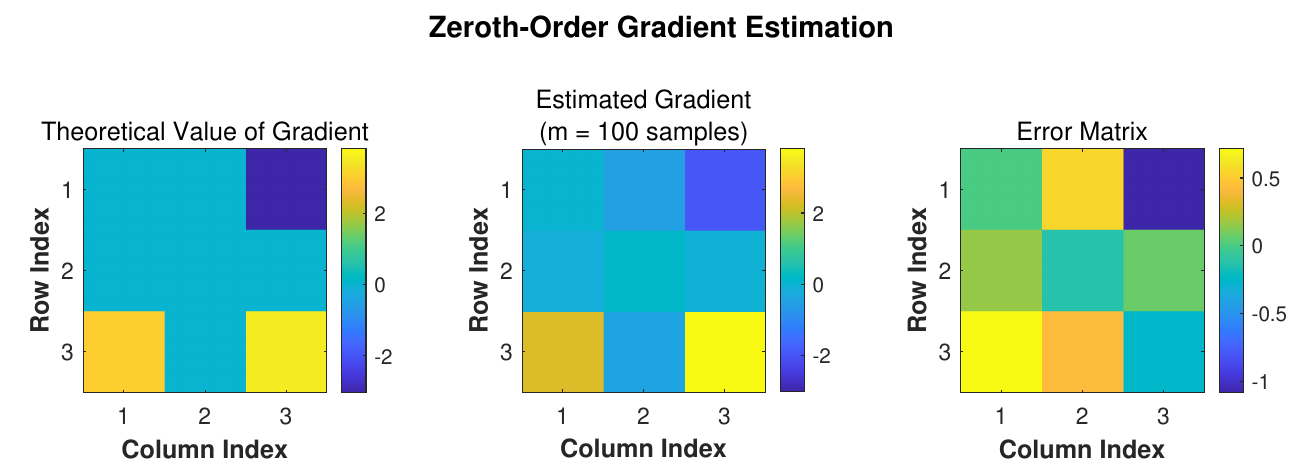}}

\medskip
\subfloat[Estimation result for $m=1,000$.\label{subfig:m1000}]{%
\includegraphics[width=0.48\linewidth]{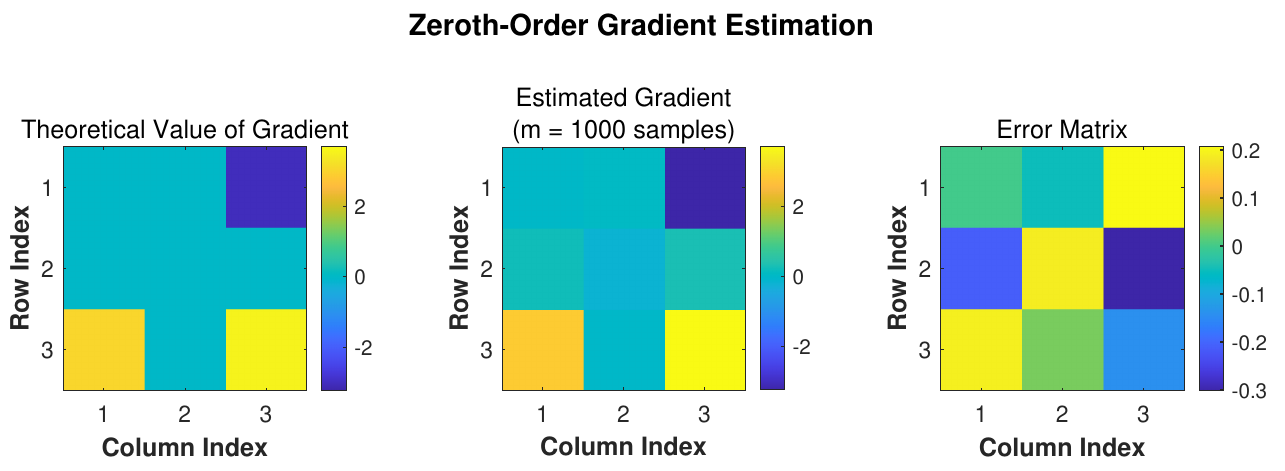}}\hfill
\subfloat[Estimation result for $m=10,000$.\label{subfig:m10000}]{%
\includegraphics[width=0.48\linewidth]{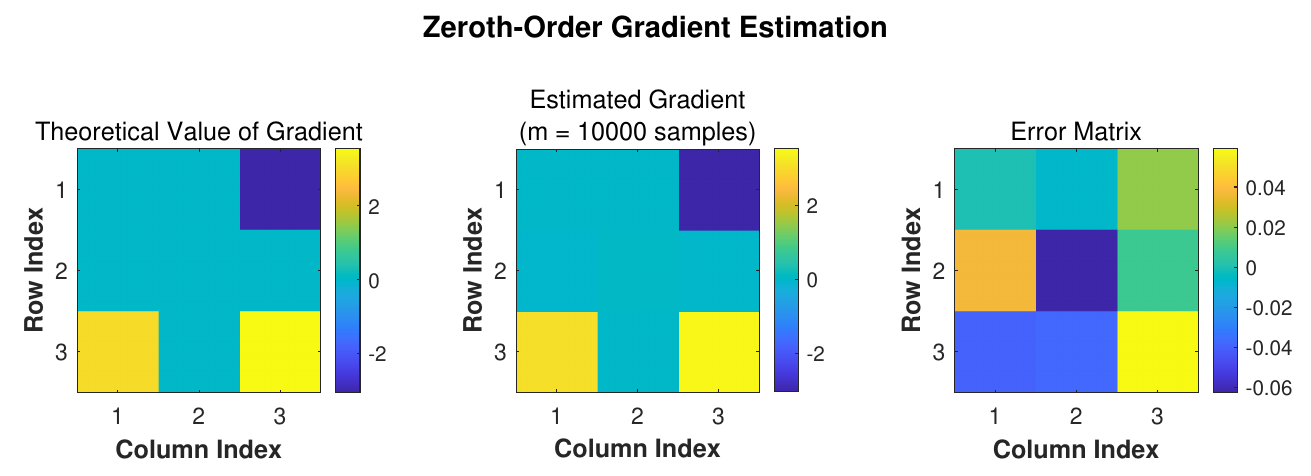}}
\caption{Qualitative convergence of the zeroth-order gradient estimate as the number of Monte Carlo samples ($m$) increases.}
\label{fig:gradient_convergence_2x2}
\end{figure}

\begin{table}[htbp]
\centering
\caption{Quantitative relative error of the zeroth-order gradient estimation under different sample sizes.}
\label{tab:error_comparison_simple}
\begin{tabular}{lcccc}
\toprule
Sample size ($m$) & 10 & 100 & 1,000 & 10,000\\
\midrule
Relative error (\%) & 110.82 & 18.84 & 7.03 & 1.99\\
\bottomrule
\end{tabular}
\end{table}
			
			\section{Conclusion and Future Directions}
			This paper investigates the direct policy optimization of LQG problem  with noisy observation. We proposed a verifiable necessary and sufficient condition for global optimality, which serves as a computationally tractable certificate and offers insights into the emergence of suboptimal stationary points in traditional parameterizations. Based on this condition, we developed a gradient-based algorithm in the infinite-dimensional $\mathcal{RH}_\infty$ space and \textcolor{black}{proved that its objective values converge sublinearly to the globally optimal value}.
			Future work should address the scalability of the proposed algorithm to large-scale systems and its robustness against unmodeled dynamics.
			In summary, this study analyzes the LQG problem from a policy optimization perspective, contributing to the understanding of the optimization landscape and algorithm design in this area.

			\appendix
			\section{Second-Order Optimality Condition}\label{hessian2}

\begin{proof}
					Based on Lemma 3 of \cite{9993305},  $D_{\Delta}^2 J\left(\widetilde{\mathcal{K}}(\lambda)\right)$ in the direction of
					\begin{equation*}
						\Delta=\left[\begin{array}{c|cc}
							0&0  & \Delta_{C}\\
							\hline
							0& 0 &\Delta_{A1}\\
							\Delta_{B}& \Delta_{A2}& \Delta_{A3}
						\end{array}\right]\in\mathcal{M}_{p}
					\end{equation*}
					can be expressed as
					\begin{equation*}
						\begin{aligned}
							&D_{\Delta}^2 J\left(\widetilde{\mathcal{K}}(\lambda)\right)\\
							=&2 \operatorname{tr}\left(2\left[ \begin{array}{cc}
								\Sigma_{\mathcal{K}} & 0\\
								0& 0
							\end{array}\right]P_{\mathcal{K}, \Delta}^{\prime}\left[\begin{array}{ccc}
								0 &  0 & B\Delta_{C} \\
								0& 0& \Delta_{A1}\\
								\Delta_{B}C & \Delta_{A2} & \Delta_{A3}
							\end{array}\right]+2{P}_{\mathcal{K}, \Delta}^{\prime}\left[\begin{array}{ccc}
								0&0 & 0 \\
								0&0 & {B}_{\mathcal{K}} \mathcal{V} \Delta_{B}^\top\\
								0&0&0
							\end{array}\right] \right).
						\end{aligned}
					\end{equation*}
					Here, $\mathcal{P}_{\mathcal{K}}$ and $\Sigma_{\mathcal{K}}$ are the solutions to the Lyapunov equations \eqref{eq_lyap1} and \eqref{eq_lyap2}, and $P_{\mathcal{K}, \Delta}^{\prime} \in \mathbb{R}^{(n+q+p) \times(n+q+p)}$ is the solution to the following Lyapunov equation
					\begin{equation}\label{lyp}
						\left[\begin{array}{ccc}
							A & B C_{\mathcal{K}} & 0 \\
							B_{\mathcal{K}} C & A_{\mathcal{K}}& 0\\
							0& 0& \lambda I_{p}
						\end{array}\right]^{\top} P_{\mathcal{K}, \Delta}^{\prime}+P_{\mathcal{K}, \Delta}^{\prime}\left[\begin{array}{ccc}
							A & B C_{\mathcal{K}} & 0 \\
							B_{\mathcal{K}} C & A_{\mathcal{K}}& 0\\
							0& 0& \lambda I_{p}
						\end{array}\right]+M_1\left(\mathcal{P}_{\mathcal{K}}, \Delta\right)=0
					\end{equation}
					with
					\begin{equation*}
						\begin{aligned}
							&M_1\left(\mathcal{P}_{\mathcal{K}}, \Delta\right)\\
							:=& \left[\begin{array}{ccc}
								0 &  0 & B\Delta_{C} \\
								0& 0& \Delta_{A1}\\
								\Delta_{B}C & \Delta_{A2} & \Delta_{A3}
							\end{array}\right]^{\top} \left[ \begin{array}{cc}
								\mathcal{P}_{\mathcal{K}} & 0\\
								0& 0
							\end{array}\right]+\left[ \begin{array}{cc}
								\mathcal{P}_{\mathcal{K}} & 0\\
								0& 0
							\end{array}\right]\left[\begin{array}{ccc}
								0 &  0 & B\Delta_{C} \\
								0& 0& \Delta_{A1}\\
								\Delta_{B}C & \Delta_{A2} & \Delta_{A3}
							\end{array}\right]\\
							&+\left[\begin{array}{ccc}
								0 & 0 & 0\\
								0& 0& {C}_{\mathcal{K}}^{\top} R \Delta_{C}\\
								0 & \Delta_{C}^\top R {C}_{\mathcal{K}}& 0
							\end{array}\right] .
						\end{aligned}
					\end{equation*}
					We now decompose the direction $\Delta$ into four parts:
					\begin{align*}
						\Delta^{(1)}=\left[\begin{array}{c|cc} 0 & 0 & \Delta_{C} \\ \hline 0 & 0 & 0 \\ 0 & 0 & 0 \end{array}\right],\ 
						\Delta^{(2)}=\left[\begin{array}{c|cc} 0 & 0 & 0 \\ \hline 0 & 0 & 0 \\ \Delta_{B} & 0 & 0 \end{array}\right],\\
						\Delta^{(3)}=\left[\begin{array}{c|cc} 0 & 0 & 0 \\ \hline 0 & 0 & 0 \\ 0 & 0 & \Delta_{A3} \end{array}\right],\ 
						\Delta^{(4)}=\left[\begin{array}{c|cc} 0 & 0 & 0 \\ \hline 0 & 0 & \Delta_{A1} \\ 0 & \Delta_{A2} & 0 \end{array}\right].
					\end{align*}
					With these definitions, $D_{\Delta}^2 J\left(\widetilde{\mathcal{K}}(\lambda)\right)$ expands as
					\begin{equation*}
						\begin{aligned}
							D_{\Delta}^2 J\left(\widetilde{\mathcal{K}}(\lambda)\right)=&D_{\Delta^{(1)}+\Delta^{(2)}+\Delta^{(3)}}^2 J\left(\widetilde{\mathcal{K}}(\lambda)\right)+2D_{\Delta^{(3)}}D_{\Delta^{(4)}}J\left(\widetilde{\mathcal{K}}(\lambda)\right)\\ &+2D_{\Delta^{(1)}+\Delta^{(2)}}D_{\Delta^{(4)}}J\left(\widetilde{\mathcal{K}}(\lambda)\right)+D_{\Delta^{(4)}}^2J\left(\widetilde{\mathcal{K}}(\lambda)\right).
						\end{aligned}
					\end{equation*}
					Define the transfer function matrices
					\begin{equation*}
						\begin{gathered}
							\mathbf{H}_{11}:=(B^\top P_1^\mathcal{K}+R\bar{C}_\mathcal{K})(sI-\mathcal{A})^{-1}((C\Sigma_1^\mathcal{K})^\top+\bar{B}_\mathcal{K}\mathcal{V}),\ \mathbf{H}_{22}:=P_2^\mathcal{K}(sI-\mathcal{A})^{-1}\Sigma_2^{\mathcal{K}\top},\\ \mathbf{H}_{12}:=(B^\top P_1^\mathcal{K}+R\bar{C}_\mathcal{K})(sI-\mathcal{A})^{-1}\Sigma_2^{\mathcal{K}\top},\
							\mathbf{H}_{21}:=P_2^\mathcal{K}(sI-\mathcal{A})^{-1}((C\Sigma_1^\mathcal{K})^\top+\bar{B}_\mathcal{K}\mathcal{V}).
						\end{gathered}
					\end{equation*} 
					Then, the optimality condition \eqref{value} is equivalent to
					\begin{equation*}
						\begin{aligned}
							&		(\mathcal{C}_1-\mathcal{B}_1\mathcal{P}_\mathcal{K})(sI-\mathcal{A})^{-1}(\mathcal{B}_0-\Sigma_\mathcal{K}\mathcal{C}_0)
							=\!
							\begin{bmatrix}
								\mathbf{H}_{11} & \mathbf{H}_{12}\\
								\mathbf{H}_{21} & \mathbf{H}_{22}
							\end{bmatrix}\equiv0,\quad \forall s\in\mathbb{C}.
						\end{aligned}
					\end{equation*}
					To complete the proof, it remains to show the equivalence between the condition
					\[
					D_{\Delta}^2 J(\widetilde{\mathcal{K}}(\lambda)) \equiv 0, \quad \forall \lambda< 0, \Delta\in\mathcal{M}_{p}
					\]
					and the condition \eqref{value}, i.e.,
					\[
					\mathbf{H}_{11} \equiv \mathbf{H}_{12} \equiv \mathbf{H}_{21} \equiv \mathbf{H}_{22} \equiv 0, \quad \forall s \in \mathbb{C}.
					\]
					Theorem 2 of \cite{9993305} implies that the first term $D_{\Delta^{(1)}+\Delta^{(2)}+\Delta^{(3)}}^2 J\left(\widetilde{\mathcal{K}}(\lambda)\right)\equiv 0,\ \forall \lambda< 0,~\Delta\in\mathcal{M}_{p}$ if and only if $\mathbf{H}_{11}\equiv 0,\ \forall s\in\mathbb{C}$. We therefore only analyze the subsequent terms $D_{\Delta^{(4)}}^2J\left(\widetilde{\mathcal{K}}(\lambda)\right)$,$D_{\Delta^{(1)}+\Delta^{(2)}}D_{\Delta^{(4)}}J\left(\widetilde{\mathcal{K}}(\lambda)\right)$ and $D_{\Delta^{(3)}}D_{\Delta^{(4)}}J\left(\widetilde{\mathcal{K}}(\lambda)\right)$. First, for $D_{\Delta^{(4)}}^2J\left(\widetilde{\mathcal{K}}(\lambda)\right)$, it holds that
					\begin{equation}\label{M1}
						\begin{aligned}
							M_1\left(\mathcal{P}_{\mathcal{K}}, \Delta^{(4)}\right) = \left[\begin{array}{ccc}
								0 &  0 & P^{\mathcal{K}}_{12}\Delta_{A1} \\
								0& 0& P^{\mathcal{K}}_{22}\Delta_{A1}\\
								\Delta_{A1}^\top( P^{\mathcal{K}}_{12})^\top & \Delta_{A1}^\top P^{\mathcal{K}}_{22}& 0
							\end{array}\right]. 
						\end{aligned}
					\end{equation}
					Taking \eqref{M1} into $D_{\Delta^{(4)}}^2J\left(\widetilde{\mathcal{K}}(\lambda)\right)$ and using the Lyapunov equation \eqref{lyp}, we have
					\begin{equation}\label{exp-hess1}
						\begin{aligned}
							&D_{\Delta^{(4)}}^2J\left(\widetilde{\mathcal{K}}(\lambda)\right)\\
							&= 4\int_{0}^{\infty} \operatorname{tr}\Bigg( 
							\begin{bmatrix}
								0 &  e^{\mathcal{A}^\top t}(P^{\mathcal{K}}_{2})^\top\Delta_{A1} e^{\lambda t} \\
								\Delta_{A1}^\top e^{\lambda t}P^{\mathcal{K}}_{2} e^{\mathcal{A}t} & 0
							\end{bmatrix}  \Delta^{(4)} 
							\begin{bmatrix}
								\Sigma_{\mathcal{K}} & 0\\
								0 & 0
							\end{bmatrix} \Bigg) dt\\
							&= 4\int_{0}^{\infty} \operatorname{tr}\left(
							\begin{bmatrix}
								0 & 0\\
								0 & \Delta_{A1}^\top P^{\mathcal{K}}_{2} e^{(\lambda I+\mathcal{A})t}(\Sigma^{\mathcal{K}}_2)^\top\Delta_{A2}^\top
							\end{bmatrix} \right) dt.
						\end{aligned}
					\end{equation}
					According to \eqref{exp-hess1}, it follows that  $D_{\Delta^{(4)}}^2J\left(\widetilde{\mathcal{K}}(\lambda)\right)\equiv0,\ \forall \lambda<0,~\Delta\in\mathcal{M}_{p}$ if and only if $\mathbf{H}_{22}\equiv 0,\ \forall s\in\mathbb{C}$. By similar arguments, for $D_{\Delta^{(1)}+\Delta^{(2)}}D_{\Delta^{(4)}}J\left(\widetilde{\mathcal{K}}(\lambda)\right)$, we have
					\begin{equation}\label{exp-hess2}
						\begin{aligned}
							&D_{\Delta^{(1)}+\Delta^{(2)}}D_{\Delta^{(4)}}J\left(\widetilde{\mathcal{K}}(\lambda)\right)\\
							=&2\int_{0}^{\infty} \operatorname{tr}\left(\left[\begin{array}{ccc}
								0& 0\\
								0& \Delta_{C}^\top(B^\top P^{\mathcal{K}}_1+R \bar{C}_\mathcal{K} )\exp((\lambda I+\mathcal{A})t)(\Sigma^{\mathcal{K}}_{2})^\top\Delta_{A2}^\top
							\end{array}\right]\right) dt\\
							&+2\int_{0}^{\infty} \operatorname{tr}\left(\left[\begin{array}{ccc}
								0& 0\\
								0& \Delta_{A1}^\top P^{\mathcal{K}}_2\exp((\lambda I+\mathcal{A})t)(C\Sigma^{\mathcal{K}}_{1}+\mathcal{V} \bar{B}_{\mathcal{K}}^\top )^\top\Delta_{B}^\top
							\end{array}\right]\right) dt.
						\end{aligned}
					\end{equation}
					Equation \eqref{exp-hess2} implies that the condition $D_{\Delta^{(1)}+\Delta^{(2)}}D_{\Delta^{(4)}}J\left(\widetilde{\mathcal{K}}(\lambda)\right)\equiv0,\ \forall \lambda<0,~\Delta\in\mathcal{M}_{p}$  holds if and only if $\mathbf{H}_{12}\equiv\mathbf{H}_{21}\equiv 0,\ \forall s\in\mathbb{C}$. Furthermore, for $D_{\Delta^{(3)}}D_{\Delta^{(4)}}J\left(\widetilde{\mathcal{K}}(\lambda)\right)$, it holds that
					\begin{align*}
						&D_{\Delta^{(3)}}D_{\Delta^{(4)}}J\left(\widetilde{\mathcal{K}}(\lambda)\right)\\
						=&2 \operatorname{tr}\left(\left[ \begin{array}{cc}
							\Sigma_{\mathcal{K}} & 0\\
							0& 0
						\end{array}\right]P_{\mathcal{K}, \Delta^{(4)}}^{\prime}\left[\begin{array}{ccc}
							0 &  0 & 0 \\
							0& 0& 0\\
							0 & 0 & \Delta_{A3}
						\end{array}\right]\right)\\
						&+ 2\operatorname{tr}\left(\left[ \begin{array}{cc}
							\Sigma_{\mathcal{K}} & 0\\
							0& 0
						\end{array}\right]P_{\mathcal{K}, \Delta^{(3)}}^{\prime}\left[\begin{array}{ccc}
							0 &  0 & 0 \\
							0& 0& \Delta_{A1}\\
							0 &  \Delta_{A2} & 0
						\end{array}\right]\right)  \equiv0,\ \forall \lambda<0,\Delta\in\mathcal{M}_{p}.
					\end{align*}
					In conclusion, $D_{\Delta}^2 J\left(\widetilde{\mathcal{K}}(\lambda)\right)\equiv 0$ for any $\lambda <0$ and $\Delta\in\mathcal{M}_{p}$ if and only if the condition \eqref{value} holds. Meanwhile, based on the specific forms of the derivatives, when \eqref{value2} does not hold, we show that there exist $\Delta\in\mathcal{M}_{p}$ and $\lambda<0$ such that $D_{\Delta}^2J(\widetilde{\mathcal{K}}(\lambda)) < 0$. This completes the proof.
			\end{proof}
			
			

			
			\section{Data-Driven Estimation Procedures for Key Operators}\label{data}

Fix the nominal controller $\mathcal K_0\in\mathbb K_q$ used in
Algorithm~\ref{alg:Q_update}. We describe input-output estimation of
$\mathbf M_{22}^{\mathcal K_0}$, $\mathbf S_k$, and
$\mathbf e_{(q+m_1)\times(q+m_2)}(\ResLHP{\mathbf S_k},m_1,m_2)$.
The variables $(\mathbf Q_k,Q_k)$ belong to~\eqref{eq:set_U_pretty},
with the inner product and norm in~\eqref{inner}--\eqref{norm}.
Here, $Q_k$ is the constant Youla component; the quadratic cost weights
remain fixed.

\subsection[Estimation of M22]{Estimation of $\mathbf M_{22}^{\mathcal K_0}$}\label{est_M}

Let $u_{\rm p}(t)\in\mathbb R^{m_1}$ and
$u_{\hat x}(t)\in\mathbb R^q$ be auxiliary inputs added to the nominal
control and controller-state equation, respectively:
\begin{align*}
& U_t=C_{\mathcal K_0}\widehat X_t+u_{\rm p}(t),\qquad
  {\rm d}\widehat X_t=A_{\mathcal K_0}\widehat X_t\,\mathrm dt
  +B_{\mathcal K_0}\,\mathrm dY_t+u_{\hat x}(t)\,\mathrm dt,\\
& u_{\rm e}(t)=\begin{bmatrix}u_{\rm p}(t)\\u_{\hat x}(t)\end{bmatrix}
  \in\mathbb R^{q+m_1}.
\end{align*}
The associated noise-free deterministic realization in~\eqref{m_k} is
\begin{align*}
& \dot\xi=\mathcal A_0\xi+
  \begin{bmatrix}B&0\\0&I_q\end{bmatrix}u_{\rm e},\qquad
  y_{\rm e}=\begin{bmatrix}C&0\\0&I_q\end{bmatrix}\xi,\qquad \xi(0)=0,\\
& \mathbf M_{22}^{\mathcal K_0}(s)=
  \begin{bmatrix}C&0\\0&I_q\end{bmatrix}
  (sI-\mathcal A_0)^{-1}\begin{bmatrix}B&0\\0&I_q\end{bmatrix}
  \in\mathcal{RH}_{\infty,0}^{(q+m_2)\times(q+m_1)}.
\end{align*}

For an observation window $[0,T]$, use the frequency coefficients
\begin{align*}
& \widehat u_{{\rm e},T}(\omega)
  =\frac1T\int_0^T e^{-j\omega t}u_{\rm e}(t)\,\mathrm dt,\\
& \widehat y_{{\rm e},T}(\omega)
  =\frac1T\begin{bmatrix}
  \displaystyle\int_0^T e^{-j\omega t}\,\mathrm dY_t\\[2pt]
  \displaystyle\int_0^T e^{-j\omega t}\widehat X_t\,\mathrm dt
  \end{bmatrix}.
\end{align*}
Their input-output relation takes the form
\begin{align}\label{resp_fre}
& \widehat y_{{\rm e},T}(\omega)
  =\mathbf M_{22}^{\mathcal K_0}(j\omega)
   \widehat u_{{\rm e},T}(\omega)+w_T(\omega),
\end{align}
where $w_T$ contains the process and measurement noise terms and the
finite-window boundary term.

For each channel $h=1,\ldots,q+m_1$, apply
\begin{align*}
& u_{\rm e}(t)=c_\nu\sin(\omega_\nu t)\,
  e_{(q+m_1)\times1}(h,1),\qquad c_\nu>0.
\end{align*}
Estimate the steady-state amplitude and phase after the transient
\cite{10.5555/557291}. Relation~\eqref{resp_fre} yields empirical samples
\begin{align*}
& \widehat m_{ih,\nu}\approx
  [\mathbf M_{22}^{\mathcal K_0}(j\omega_\nu)]_{ih},\qquad
  i=1,\ldots,q+m_2,\quad h=1,\ldots,q+m_1,\quad
  \nu=1,\ldots,L_\omega.
\end{align*}
Fit each entry using real polynomials with a monic denominator:
\begin{align*}
& N(s;a)=\sum_{\ell=0}^{n_{\rm num}}a_\ell s^\ell,\qquad
  D(s;b)=s^{n_{\rm den}}+
  \sum_{\ell=0}^{n_{\rm den}-1}b_\ell s^\ell,\qquad
  0\le n_{\rm num}<n_{\rm den}.
\end{align*}
Given the samples $\widehat m_{ih,\nu}$, solve
\begin{align}\label{eq:dd-normalized-fit}
& \min_{a,b}\ \sum_{\nu=1}^{L_\omega}
  \left|N(j\omega_\nu;a)
    -\widehat m_{ih,\nu}D(j\omega_\nu;b)\right|^2,\\
& a\in\mathbb R^{n_{\rm num}+1},\qquad
  b\in\mathbb R^{n_{\rm den}}.\notag
\end{align}
Problem~\eqref{eq:dd-normalized-fit} is a convex linear least-squares
surrogate for matching the sampled frequency responses
\cite{6198975,6198973,6198977,650025}.
Stability is checked separately: the fit $N/D$ is used in
$\mathcal{RH}_{\infty,0}$ only if $D$ is Hurwitz.
Assemble the fitted entries into $\widehat{\mathbf M}_{22}^{\mathcal K_0}$.

\subsection{Estimation of Sensitivity Operators}\label{est_S}

By Theorem~\ref{thm:gradient}, the dynamic component of the Fr\'echet
gradient is $2\mathbf S_k$. We estimate $\mathbf S_k$ by directional
performance evaluations. Let
\begin{align*}
& \psi_\ell(s)=\frac{\sqrt{2a}}{s+a}
  \left(\frac{s-a}{s+a}\right)^\ell,\qquad a>0,\quad \ell=0,1,\ldots,\\
& E_{ij}=e_{(q+m_1)\times(q+m_2)}(i,j),\qquad
  i=1,\ldots,q+m_1,\quad j=1,\ldots,q+m_2.
\end{align*}
The Laguerre functions $\{\psi_\ell\}_{\ell\ge0}$ form an orthonormal
basis for the stable strictly proper component
\cite{AGAMENNONI1992815}. For the perturbation
$(\Delta\mathbf Q,\Delta Q)=(c\psi_\ell E_{ij},0)$, we have
\begin{align}
& J_{\mathcal K_0}(\mathbf Q_k+c\psi_\ell E_{ij},Q_k)
  -J_{\mathcal K_0}(\mathbf Q_k,Q_k)\notag\\
& \qquad=2c\langle\mathbf S_k,\psi_\ell E_{ij}\rangle_{\mathcal H_2}
  +o(c).\label{per_di}
\end{align}
Thus,~\eqref{per_di} gives the projection coefficients
\cite{Gautschi2012}:
\begin{align*}
& c_\ell^{(k)}(i,j)
  =\langle\mathbf S_k,\psi_\ell E_{ij}\rangle_{\mathcal H_2},\\
& c_\ell^{(k)}(i,j)=\lim_{c\to0}
  \frac{J_{\mathcal K_0}(\mathbf Q_k+c\psi_\ell E_{ij},Q_k)
    -J_{\mathcal K_0}(\mathbf Q_k,Q_k)}{2c},\\
& [\widehat{\mathbf S}_{k,L}(s)]_{ij}
  =\sum_{\ell=0}^{L}c_\ell^{(k)}(i,j)\psi_\ell(s).
\end{align*}
Finite differences approximate the limit above. Perturbed controllers
are obtained using Lemma~\ref{lem_rea}.

For a reduced-order rational approximation, evaluate the known
Laguerre approximation on the frequency grid:
\begin{align*}
& \widehat s_{ij,\nu}^{(k)}
  =[\widehat{\mathbf S}_{k,L}(j\omega_\nu)]_{ij},\qquad
  \nu=1,\ldots,L_\omega.
\end{align*}
For each $(i,j)$, solve~\eqref{eq:dd-normalized-fit} with
$\widehat m_{ih,\nu}$ replaced by $\widehat s_{ij,\nu}^{(k)}$.
Use $n_{\rm num}<n_{\rm den}$ and retain a fitted entry only when its
denominator is Hurwitz. The assembled matrix
$\widehat{\mathbf S}_k\in
\mathcal{RH}_{\infty,0}^{(q+m_1)\times(q+m_2)}$
then approximates the dynamic update direction.

\subsection{Estimation of Residue Terms}\label{est_res}

The constant component is estimated by zeroth-order methods
\cite{Flaxman2004OnlineCO,nesterov2017random}.
For fixed $\mathbf Q_k$, define
\begin{align*}
& f_k(Q)=J_{\mathcal K_0}(\mathbf Q_k,Q),\qquad
  Q\in\mathbb R^{(q+m_1)\times(q+m_2)\mid m_1\times m_2},\\
& \nabla f_k(Q_k)
  =2\mathbf e_{(q+m_1)\times(q+m_2)}
    (\ResLHP{\mathbf S_k},m_1,m_2).
\end{align*}
Use the Frobenius sphere in this real vector space:
\begin{align*}
& \mathbb S_r=
  \left\{U\in\mathbb R^{(q+m_1)\times(q+m_2)\mid m_1\times m_2}:
  \|U\|_F=r\right\},\qquad r>0,\\
& d=(q+m_1)(q+m_2)-m_1m_2,\qquad
  U\sim\operatorname{Unif}(\mathbb S_r).
\end{align*}
Because $f_k$ is quadratic, the symmetric finite difference yields
\begin{align}\label{est_zero}
& \nabla f_k(Q_k)=\frac{d}{2r^2}
  \mathbb E_U\!\left[
  \bigl(f_k(Q_k+U)-f_k(Q_k-U)\bigr)U\right].
\end{align}
For the sphere-based construction, see also
\cite[Appendix~D.3]{Fazel2018GlobalCO}.
With independent samples $U_1,\ldots,U_m\sim
\operatorname{Unif}(\mathbb S_r)$, approximate~\eqref{est_zero} by
\begin{align}\label{emp_est}
& \widehat g_k=\frac{d}{2mr^2}\sum_{i=1}^m
  \bigl(f_k(Q_k+U_i)-f_k(Q_k-U_i)\bigr)U_i.
\end{align}
The estimate~\eqref{emp_est} gives the projected residue sum and
the constant update in Algorithm~\ref{alg:Q_update}:
\begin{align*}
& \widehat R_k=\tfrac12\widehat g_k
  \approx\mathbf e_{(q+m_1)\times(q+m_2)}
       (\ResLHP{\mathbf S_k},m_1,m_2),\\
& Q_{k+1}=Q_k-\eta\widehat R_k
  =Q_k-\frac\eta2\widehat g_k.
\end{align*}
Zeroth-order performance evaluations of this type are used in policy
optimization
\cite{Fazel2018GlobalCO,doi:10.1137/20M1382386,pmlr-v144-tang21a}.

\bibliographystyle{plain}
			\bibliography{references}
		\end{document}